\documentclass[leqno,12pt]{amsart}

\usepackage{amssymb}
\usepackage{amsmath}
\usepackage[usenames, dvipsnames]{color}
\usepackage{enumitem}
\usepackage{amsthm}
\usepackage{mathrsfs}
\usepackage{amsrefs}
\numberwithin{equation}{section}

\usepackage{marginnote}
\usepackage{todonotes}
\usepackage{hyperref}
\usepackage{esint}
\usepackage{verbatim}
\usepackage{tikz}

\usepackage{hyperref}

\usepackage{geometry} 

\usepackage{etoolbox}

\theoremstyle{plain}
\newtheorem{theorem}{Theorem}[section]
\newtheorem{lemma}[theorem]{Lemma}

\newtheorem{proposition}[theorem]{Proposition}

\newtheorem{definition}[theorem]{Definition}

\newtheorem{remark}[theorem]{Remark}

\newcommand{\R}{\mathbb{R}}
\newcommand{\N}{\mathbb{N}}

\newcommand{\bA}{\mathbf{A}}
\newcommand{\bB}{\mathbf{B}}

\newcommand{\E}{\mathcal{E}}

\newcommand{\M}{\mathcal{M}}

\newcommand{\W}{\mathcal{W}}

\newcommand{\hW}{\hat{\mathcal{W}}}

\makeatletter
\@namedef{subjclassname@2020}{\textup{2020} Mathematics Subject Classification}
\makeatother

\newcommand{\wei}[1]{\langle #1 \rangle}

\def\XXint#1#2#3{{\setbox0=\hbox{$#1{#2#3}{\int}$ }
\vcenter{\hbox{$#2#3$ }}\kern-.6\wd0}}

\title[Parabolic Equations of FKS Type]{Well-Posedness and Regularity Theory for Parabolic Equations of Fabes-Kenig-Serapioni Type}

\author[T. Phan]{Tuoc Phan}
\address[T. Phan]{Department of Mathematics, University of Tennessee, 227 Ayres Hall,
1403 Circle Drive, Knoxville, TN 37996-1320}
\email{tphan2@utk.edu}


\begin{document}

\begin{abstract} We investigate Dirichlet boundary value problems for a class of second-order parabolic equations in divergence-form with coefficient matrices that exhibit singular and degenerate behaviors characterized by a Muckenhoupt weight class. This framework serves as the parabolic analogue to the singular-degenerate elliptic equations pioneered by Fabes, Kenig, and Serapioni. Under a smallness assumption on the partially weighted mean oscillation of the coefficients, we establish the existence, uniqueness, and local interior and boundary regularity estimates for weak solutions within appropriately defined weighted Sobolev spaces. The proofs rely on the freezing coefficient technique alongside the level-set method introduced by Caffarelli and Peral. Additionally, we develop the necessary weighted Sobolev space framework and related weighted inequalities. Finally, a compactness argument is utilized to demonstrate that solutions to these equations remain locally close, in the weighted Sobolev norm, to their frozen-coefficient counterparts.
\end{abstract}
\maketitle
\section{Introduction} 
\subsection{Problem settings and motivations} \label{mov-sec} We study the Dirichlet boundary value problems for a class of second-order linear parabolic equations whose leading coefficients are of Fabes-Kenig-Serapioni singular-degenerate type
\begin{equation} \label{main-eqn}
\left\{
\begin{array}{cccl}
u_t - \textup{div}(\bA(x,t) \nabla u) & = & \textup{div}(\beta(x) F(x,t))  & \quad  \text{in} \quad \Omega_T, \\[4pt]
u & =  & 0 & \quad \text{on} \quad \partial' \Omega_T.
\end{array} \right.
\end{equation}
Here in \eqref{main-eqn}, $\Omega_T = \Omega \times (0, T]$ with some given $T \in (0,\infty)$ and some non-empty open bounded domain $\Omega \subset \mathbb{R}^n$ with boundary $\partial \Omega$, $F: \Omega_T \rightarrow \mathbb{R}^n$ is a given forcing term, $u: \Omega_T \rightarrow \mathbb{R}$ is an unknown solution, and $\partial' \Omega_T$ is the parabolic boundary of $\Omega_T$:
\begin{equation*} 
 \partial' \Omega_T = (\overline{\Omega} \times \{0\}) \cup (\partial \Omega \times (0, T]).
\end{equation*}
Moreover, $\bA: \Omega_T \rightarrow \mathbb{R}^{n\times n}$ is a symmetric measurable matrix-valued function satisfying the Fabes-Kenig-Serapioni  singular-degenerate type conditions: there are a constant $\nu \in (0,1)$ and a non-negative locally integrable function $\beta: \mathbb{R}^n \rightarrow [0, \infty)$ such that
\begin{equation} \label{ellip-cond}
\nu \beta(x) |\xi|^2 \leq \wei{\bA(x,t) \xi, \xi} \quad \text{and} \quad |\bA(x,t)| \leq \nu^{-1} \beta(x), 
\end{equation}
for all $(x,t) \in \Omega_T$ and for all $\xi=(\xi_1, \xi_2,\ldots, \xi_n) \in \mathbb{R}^n$. We assume that $\beta$ belongs to the Muckenhoupt class $A_{1+\frac{1}{n_0}}$ of weights with $n_0 = \max\{n-1, 1\}$.  As such, the eigenvalues of the coefficient matrix $\bA$ in \eqref{main-eqn} can vanish (be degenerate) or blow up (be singular) in $\overline{\Omega}$.  When $\beta =1$, the equation \eqref{main-eqn} is reduced to the standard second order parabolic equations in divergence-form. 

\smallskip
Parabolic equations with singular-degenerate coefficients arise naturally in physical and biological systems where diffusion and transport depend heavily on heterogeneous media or varying state densities. In mathematical biology, for example, the continuum framework introduced in \cite{EPL} models the spatio-temporal expansion of bacterial biofilms via nonlinear, density-dependent diffusion. In this setting, the governing equation exhibits porous medium-type degeneracy as the biomass density vanishes, alongside singular super-diffusion as it approaches its maximum physical threshold. These dynamics connect directly to the broader theory of the porous medium equation, developed extensively in \cite{Vazquez, JX, JX1}. Similar structures also appear in heat transfer problems involving extreme phase transitions, composite materials, or low-density gas vacuums where uniform parabolicity breaks down \cite{DIT, FKRH, Li-Xi, PSW, Visintin}. Beyond physics and biology, such operators play a prominent role in geometric analysis and mathematical finance \cite{Le, Pop-2}.

\smallskip
Muckenhoupt weight classes were originally introduced in \cite{muckenhoupt} to characterize the boundedness of operators. These classes provide a robust and rigorous framework capable of capturing spatial singularities and vanishing behaviors simultaneously. This weighted approach has proven crucial in establishing priori estimates, weighted Sobolev embeddings, and regularity theory for both linear, nonlinear elliptic and parabolic problems, and in nonlocal fractional operators. See \cite{Fabes, Fabes-1, STV1, STV2} for H\"{o}lder, Schauder regularity theory and other quantitative analysis estimates, \cite{DP, DP1, DP3, DPT1, DPT2, DPS} for well-posedness and regularity theory in Sobolev spaces, and \cite{JS, STT} for regularity estimates of nonlinear linear and fully nonlinear equations. Consequently, the study of weak solutions in these weighted spaces bridges the gap between modern harmonic analysis and non-uniform physical phenomena.

\smallskip
This paper establishes sufficient conditions on $\mathbf{A}, \beta$, and $F$ to develop a rigorous framework for weighted Sobolev spaces, through which we prove the existence, uniqueness, interior regularity, and local boundary estimates for weak solutions to \eqref{main-eqn}. Analogous results for the elliptic counterpart were studied in \cite{CMP, CMP-1}. Furthermore, qualitative H\"{o}lder regularity theory was previously investigated in the seminal elliptic works \cite{Fabes, Fabes-1} and extended to parabolic equations in \cite{Chia-1, Chia-2, Chia-3}. See also \cite{BS, Tru-1, Tru-2}  in which related classes of elliptic equations featuring singular-degenerate coefficients under alternative integrability conditions were investigated. The present work serves as a natural bridge between these distinct directions by establishing quantitative Sobolev regularity for the parabolic structure. Moreover, this paper represents a major continuation of the broader research program developed in \cite{CMP, CJP, DPS, DP, DP-0, DP1, DP3, DPT1, DPT2, FP-1, FP-2, MP}, which systematically investigates the fine quantitative properties and optimal regularity profiles of solutions to elliptic and parabolic equations whose coefficients deviate from classical uniform ellipticity and boundedness.
\subsection{Main results} To state the main results, let us introduce notations and definitions of functional spaces that are relevant to \eqref{main-eqn}. For each $p \in (1, \infty)$ and a non-empty open set $\Omega \subset \mathbb{R}^n$, $W^{1,p}(\Omega)$ denotes the usual Sobolev space, and $W^{1,p}_0(\Omega)$ is the completion of the set of compactly supported smooth functions $C_0^\infty(\Omega)$ in $W^{1,p}(\Omega)$.  Also, $W^{-1,p}(\Omega)$ is the dual space of $W^{1,p'}_0(\Omega)$, where $\frac{1}{p'}+ \frac{1}{p} =1$.  In addition, for a given non-negative and locally integrable function $\beta$ defined on $\Omega$, $L^p(\Omega, \beta)$ denotes the weighted Lebesgue space consisting of measurable functions $u: \Omega \rightarrow \mathbb{R}$ such that
\[
\|u\|_{L^p(\Omega, \beta)} = \Big(\int_{\Omega} |u(x)|^p \beta(x) dx \Big)^{1/p} <\infty.
\]
We denote $W^{1,p}(\Omega, \beta)$ the weighted Sobolev space consisting of measurable functions $u \in L^p(\Omega, \beta)$ such that their weak derivatives $u_{x_i} \in L^p(\Omega, \beta)$ for all $i=1, 2,\ldots, n$.  The space $W^{1,p}(\Omega, \beta)$ is endowed with the norm
\[
\|u\|_{W^{1,p}(\Omega, \beta)}=  \|u\|_{L^p(\Omega, \beta)} + \|\nabla u\|_{L^p(\Omega, \beta)}. 
\]
In the same way, $W^{1,p}_0(\Omega, \beta)$ is the closure in $W^{1,p}(\Omega, \beta)$ of $C_0^\infty(\Omega)$, and $W^{-1,p}(\Omega, \beta)$ is the dual of the space $W^{1,p'}_{0}(\Omega, \beta)$ in which $\frac{1}{p} + \frac{1}{p'} =1$.

\smallskip
Next, for given $T>0$, we write $\Omega_T = \Omega \times (0, T]$. We denote $\W^{1,p}(\Omega_T, \beta)$ the weighted parabolic Sobolev space consisting of measurable functions $u: \Omega_T \rightarrow \mathbb{R}$ such that
\[
u \in L^p((0, T), W^{1,p}(\Omega, \beta)) \quad \text{and} \quad u_t \in L^p((0, T), W^{-1, p}(\Omega, \beta)).
\]
The space $\W^{1,p}(\Omega_T, \beta)$ is endowed with the norm
\begin{equation}\label{Norm-W1p}
\|u\|_{\W^{1,p}(\Omega_T, \beta)} = \|u\|_{L^p((0, T), W^{1,p}(\Omega, \beta))} + \|u_t\|_{L^p((0, T), W^{-1, p}(\Omega, \beta))}, 
\end{equation}
for $u \in \W^{1,p}(\Omega_T, \beta)$. 
In addition, $\W^{1,p}_*(\Omega_T, \beta)$ the closure in  $\W^{1,p}(\Omega_T, \beta)$ of the set
\[
\big\{ \varphi_{|\Omega_T}: \varphi \in C^\infty_0((0, \infty), C_0^\infty(\Omega)) \big\}.
\]
The space $\W^{1,p}_*(\Omega_T, \beta)$ is endowed with the same norm as in \eqref{Norm-W1p}, that is,
\[
\|u\|_{\W^{1,p}_{*} (\Omega_T, \beta)} = \|u\|_{L^p((0, T), W^{1,p}(\Omega, \beta))} + \|u_t\|_{L^p((0, T), W^{-1, p}(\Omega, \beta))}
\]
for $u \in \W^{1,p}_*(\Omega_T, \beta)$. 

\smallskip
Now, for $p \in (1, \infty)$ and for $\beta$ defined as in \eqref{ellip-cond}, a function $u \in \W^{1,p}_{*} (\Omega_T, \beta)$ is said to be a weak solution to \eqref{main-eqn} if
\begin{align*}
& -\int_{\Omega_T} u(x,t) \partial_t \varphi(x,t) dx dt + \int_{\Omega_T} \wei{\bA(x,t) \nabla u(x,t), \nabla \varphi (x,t)} dx dt \\
& = -\int_{\Omega_T} \wei{\beta(x) F(x,t), \nabla \varphi(x,t)} dx dt, 
\end{align*}
for every $\varphi \in C^\infty([0, T], C^\infty_0(\Omega))$ satisfying $\varphi =0$ in $\Omega \times [T-\tau, T]$ for some sufficiently small $\tau \in (0,T)$.

\smallskip
We assume throughout the paper that the weight $\beta$ in \eqref{main-eqn}  satisfies the following conditions
\begin{equation} \label{beta-cond}
\beta \in A_{1+\frac{1}{n_0}} \quad \text{and} \quad [\beta]_{A_{1+\frac{1}{n_0}}} \leq M_0,  \quad \text{for} \quad n_0 = \max\{n-1, 1\},
\end{equation}
with for some given $M_0 \geq 1$. See Definition \ref{A-p-def} below for the definition of $A_{q}$-Muckenhoupt class of weights with $q \in (1, \infty)$.  In addition,  the \textit{non-homogeneous weighted parabolic cylinder} $Q_{\rho,\beta}(z_0)$ of radius $\rho>0$ and centered at $z_0 = (x_0, t_0) \in \mathbb{R}^n \times \mathbb{R}$ is defined by
\begin{equation} \label{cylinder-introc}
\begin{split}
& Q_{\rho,\beta}(z_0) = B_\rho(x_0) \times (t_0 -\rho^2 \E_{\beta}(x_0, \rho), t_0] \quad \text{where} \\
& \E_{\beta}(x_0, \rho) = \Big(\fint_{B_\rho(x_0)} \beta(x)^{-n_0} dx \Big)^{\frac{1}{n_0}}.
\end{split}
\end{equation}
As $\beta \in A_{1+\frac{1}{n_0}}$, it follows that $ \E_{\beta}(x_0, \rho)$ is well-defined for every $x_0 \in \mathbb{R}^n$ and for $\rho>0$. Moreover, as discussed in Sections \ref{cylinder-distance-sec} and \ref{Fn-spaces} below, the parabolic cylinders $Q_{r,\beta}(z_0)$ are generically invariant under the scaling and dilation properties for the class of equations \eqref{main-eqn}. In fact, there is a quasi-distance associated with the class of non-homogeneous parabolic cylinders $\{Q_{\rho, \beta}(z_0)\}_{\rho, z_0}$. 

\smallskip
Besides \eqref{ellip-cond}, we assume that the \textit{partial mean oscillations of the matrix $\bA$ with respect to the weight $\beta$} is sufficiently small. We would like to point out that the partial mean oscillations of functions with respect to the weights is the concept that is combined by the concept on mean oscillations with respect to the weights  introduced in \cite{M-W}, and the concept on partial mean oscillations introduced in \cite{KK1, KK2}. For a non-empty open set $G\subset\mathbb R^n$, let us denote 
\[
\begin{split}
& (\bA)_{B_{\rho}(x_0) \cap G} (t) = \frac{1}{|B_{\rho}(x_0) \cap G|} \int_{B_{\rho}(x_0) \cap G} \bA(x,t) dx \quad \text{and}  \\
& \Theta_{\bA, \rho, x_0} (x,t)= |\bA(x,t) - (\bA)_{B_{\rho}(x_0) \cap G} (t)| \beta(x)^{-1}, \quad \rho>0, \ x_0 \in \overline{G}.
\end{split}
\] 
We also write $\beta(dx) = \beta(x) dx$, and $\beta(dxdt) = \beta(x) dxdt$. Moreover, for every open non-empty bounded set $Q \subset \mathbb{R}^n \times \mathbb{R}$ and for every locally integrable function $f$ defined on $Q$, the weighted average of $f$ on $Q$ with respect to $\beta(dxdt)$ is defined by
\[
\fint_{Q} f(x,t) \beta(dxdt) = \frac{1}{\beta(Q)} \int_{Q} f(x,t) \beta(x) dxdt, \quad \text{where} \quad \beta(Q) = \int_Q \beta(x) dxdt.
\]
Then, for a fixed number $R_0>0$, and a non-empty parabolic domain $D=G\times(S,T]$ with $S<T$, the \textit{partial mean oscillations of the matrix $\bA$ with respect to the weight $\beta$ in $D$} is defined by
\begin{equation*} 
[[\bA]]_{\textup{BMO}}(D, R_0) = \sup_{\rho \in (0, R_0)}\ \sup_{\substack{z_0 =(x_0, t_0) \in D}}\Big(\fint_{Q_{\rho, \beta}(z_0) \cap D} \Theta_{\bA, \rho, x_0} (x,t)^2 \beta(dxdt)\Big)^{\frac{1}{2}}.
\end{equation*}

\smallskip
The following theorem on existence, uniqueness and global regularity estimates of weak solutions to the boundary value problem \eqref{main-eqn} is one of the main results of the paper.
\begin{theorem} \label{main-theorem} For every $\nu \in (0,1)$, $M_0 \in [1, \infty)$, and $p \in (1, \infty)$, there exists a sufficiently small positive constant $\delta = \delta (n, \nu, p, M_0)$ such that the following assertions hold. Suppose that \eqref{ellip-cond} holds, $\beta \in A_{q_p}$ with $[\beta]_{A_{q_p}}\leq M_0$, where $q_p=\min\{p,1+\frac{1}{n_0}\}$ with $n_0 = \max\{n-1, 1\}$, $\partial \Omega \in C^1$, $T \in (0, \infty)$, and 
\begin{equation} \label{BMO-defi}
\begin{split}
[[\bA]]_{\textup{BMO}}(\Omega_T, R_0) \leq \delta,
\end{split}
\end{equation}
with some $R_0 \in (0,1)$. Then, for every $F \in L^p(\Omega_T, \beta)^n$, there is a unique weak solution $u \in \W^{1,p}_{*} (\Omega_T, \beta)$ solving \eqref{main-eqn}. Moreover, there is a positive constant $N = N(n, \nu, p, M_0, R_0, \Omega, T/\lambda_0)$ such that
\begin{equation} \label{main-thm-est}
\|u\|_{\W^{1,p}_{*}(\Omega_T,\beta)}\leq N \|F\|_{L^p(\Omega_T, \beta)}
\end{equation}
where $\lambda_0 = \big(\int_\Omega \beta(x)^{-n_0} dx\big)^{1/n_0}$. 
\end{theorem}
\smallskip \noindent
We remark that condition \eqref{BMO-defi} cannot be dropped in general for \eqref{main-thm-est} when $p\neq2$, as demonstrated in \cite[Section 2.2.1]{CMP} for the elliptic case. We also note that the condition $\partial \Omega \in C^1$ can be relaxed to Lipschitz domains with small Lipschitz constants, as in \cite{HNP}, or to Reifenberg flat domains, as in \cite{CMP, BW1}. Additionally, the main results could be extended to the matrix weight setting as in \cite{BDGP}, to systems of equations as in \cite{CMP-1}, or to the weighted mixed-norm space framework as in \cite{DK, DP1, DP3}. We do not pursue these directions here to avoid 
additional technical complexities. We also note that it is not clear if the condition $\beta \in A_{1+\frac{1}{n_0}}$ as in \eqref{beta-cond} is optimal.

\smallskip
Besides Theorem \ref{main-theorem}, we also establish local interior and local boundary regularity estimates for $\W^{1,2}$-weak solutions of \eqref{main-eqn}. For simplicity in stating the results,  let $\Lambda = \Lambda(n, M_0) \geq 1$ be the constant given in Remark \ref{quasi-metric lemma} below. To investigate the local interior regularity estimates, we consider the equation
\begin{equation} \label{Q2-eqn}
 u_t - \textup{div}(\bA (x,t) \nabla u) = \textup{div}(\beta(x) F(x,t))  \quad \text{in} \quad Q_{2\Lambda, \beta},
\end{equation}
For a given $F \in L^2_{\text{loc}}(Q_{2\Lambda}, \beta)$ and under the assumptions \eqref{ellip-cond} and \eqref{beta-cond}, a function $u \in \W^{1,2}(Q_{2\Lambda, \beta}, \beta)$ is said to be a weak solution of \eqref{Q2-eqn} if 
\begin{align*}
& -\int_{Q_{2\Lambda, \beta}} u(x,t) \partial_t\varphi(x,t) dx dt + \int_{Q_{2\Lambda, \beta}} \wei{\bA(x,t) \nabla u(x,t), \nabla \varphi(x,t)} dx dt \\
& = -\int_{Q_{2\Lambda, \beta}} \wei{F(x,t), \nabla \varphi(x,t)} \beta(x) dx dt, \quad \forall \ \varphi \in C_0^\infty(Q_{2\Lambda, \beta}).
\end{align*}
The following result on local interior regularity estimates for $\W^{1,2}(Q_{2\Lambda, \beta}, \beta)$-weak solutions to the class of equations \eqref{Q2-eqn} is another main result of the paper.
\begin{theorem} \label{inter-theorem} For every $\nu \in (0,1)$, $M_0 \in [1, \infty)$, and $p \in [2,\infty)$, there exist a sufficiently small positive constant $\delta_0 = \delta_0(n, \nu, p, M_0)$ and a positive constant $N = N(n, \nu, p, M_0)$ such that the following assertions hold. Suppose that $\bA$ satisfies \eqref{ellip-cond} in $Q_{2\Lambda,\beta}$ and its corresponding $\beta$ satisfies \eqref{beta-cond}. Suppose also that
\[
[[\bA]]_{\textup{BMO}}(Q_{2\Lambda,\beta}, 1)  \leq \delta_0,
\]
and $u \in \W^{1,2}(Q_{2\Lambda, \beta}, \beta)$ is a weak solution of the equation \eqref{Q2-eqn}. If $F \in L^p(Q_{2\Lambda,\beta}, \beta)^n$, then it holds that
\begin{align*}
&\| \nabla u\|_{L^p(Q_{1,\beta}, \beta)}+\|u_t\|_{L^p(\Gamma_{\beta}(1),W^{-1,p}(B_1, \beta))}\\
&\leq N\left(\|\nabla u\|_{L^2(Q_{2\Lambda,\beta}, \beta)}+\|F\|_{L^p(Q_{2\Lambda,\beta}, \beta)}\right)
\end{align*}
where $\Gamma_{\beta}(1) = (-\E_\beta(0, 1), 0)$.
\end{theorem}

To study local boundary regularity estimates for the class of equations \eqref{main-eqn}, we consider the equation in the upper-half parabolic cylliner
\begin{equation} \label{eqn-bdr-10-4}
\left\{
\begin{aligned}
u_t - \textup{div}(\bA (x,t) \nabla u) &= \textup{div}(\beta(x) F(x,t)) \quad &&\text{in} \quad Q_{2\Lambda, \beta}^+, \\[4pt]
u &=  0 \quad &&\text{on}\quad  T_{2\Lambda}\times \Gamma_{\beta}(2\Lambda).
\end{aligned} \right.
\end{equation}
Here, for each $x_0 \in \mathbb{R}^n$, and $\rho>0$, $T_{\rho}(x_0) =B_\rho(x_0) \cap (\mathbb{R}^{n-1} \times \{0\})$, $B_\rho^+(x_0) = B_\rho(x_0) \cap (\mathbb{R}^{n-1} \times (0, \infty)$, and $Q_{\rho}^+ (z_0) = B_\rho^+(x_0) \times \Gamma_{\beta}(x_0, \rho)$ with $\Gamma_{\beta}(z_0, \rho) = (t_0 - r^2\E_{\beta}(x_0, \rho), t_0]$.  When $x_0 =0$ and $t_0 =0$, we write $T_\rho = T_\rho(0), B_\rho^+ = B_\rho^+(0), \Gamma_\beta(\rho) = \Gamma_\beta((0,0), \rho)$, et cetera.

\smallskip
Next, we introduce functional spaces for weak solutions of \eqref{eqn-bdr-10-4}. For a given upper-half ball $B^+ = B_r^+$ with some $r>0$, and for two numbers $S < T$, and $p \in (1, \infty)$,  a function $u \in L^p((S, T), W^{1,p}(B^+,\beta))$ is said to be in $\hW^{1,p}(B^+ \times (S,T), \beta)$ if $u \vert_{T_r \times (S, T)} =0$ in the sense of trace, and  
\[  
u_t \in L^{p}((S, T), W^{-1, p}(B^+, \beta)), 
\]  
where $W^{-1,p}(B^+, \beta)$ denotes the dual space of $W^{1, p'}_0(B^+,\beta)$  with $\frac{1}{p} + \frac{1}{p'} =1$. We write
\begin{align*}
\|u\|_{\hat{\W}^{1,p}(B^+ \times (S,T), \beta)} & = \|u\|_{L^p((S,T), W^{1,p}(B^+, \beta))}  +  \|u_t\|_{ L^{p}((S, T), W^{-1, p}(B^+, \beta))}.
\end{align*}
As before, a function $u \in \hat{\W}^{1,2}(Q_{2\Lambda, \beta}^+, \beta)$ is a weak solution of \eqref{eqn-bdr-10-4} if
\begin{align*}
& -\int_{Q_{2\Lambda, \beta}^+} u(x,t) \partial_t\varphi(x,t) dx dt + \int_{Q_{2\Lambda, \beta}^+} \wei{\bA(x,t) \nabla u(x,t), \nabla \varphi(x,t)} dx dt \\
& = -\int_{Q_{2\Lambda, \beta}^+} \wei{F(x,t), \nabla \varphi(x,t)} \beta(x) dx dt, \quad \forall \ \varphi \in C_0^\infty(Q_{2\Lambda, \beta}^+).
\end{align*}
The following theorem on local boundary regularity estimates is another main result of the paper.
\begin{theorem} \label{inter-theorem-bdry} For every $\nu \in (0,1)$, $M_0 \geq 1$, and $p \in [2,\infty)$, there exist a sufficiently small positive constant $\bar{\delta}_0= \bar{\delta}_0(n, \nu, p, M_0)$ and a positive constant $N=N(n, \nu, p, M_0)$ such that the following assertion holds. Suppose that $\bA$ satisfies \eqref{ellip-cond} in $Q_{2\Lambda,\beta}^+$ and its corresponding $\beta$ satisfies \eqref{beta-cond}.  Suppose also that
\[
[[\bA]]_{\textup{BMO}}(Q_{2\Lambda,\beta}^+, 1) \leq \bar{\delta}_0
\]
and $u \in \hW^{1,2}(Q_{2\Lambda, \beta}^+, \beta)$ is a weak solution of  \eqref{eqn-bdr-10-4}. If  $F \in L^p(Q_{2\Lambda,\beta}^+,\beta)^n$, then $u \in \hW^{1,p}(Q_{1,\beta}^+, \beta)$ and
\begin{align*}
\|u\|_{ \hW^{1,p}(Q_{1,\beta}^+, \beta)} \leq N\Big(\|\nabla u\|_{L^2(Q_{2\Lambda,\beta}^+, \beta)}+\|F\|_{L^p(Q_{2\Lambda,\beta}^+,\beta)}\Big).
\end{align*}
\end{theorem}
\noindent
We point out that by using the scalings and dilations described in Section \ref{Fn-spaces} below, results in Theorems \ref{inter-theorem}-\ref{inter-theorem-bdry} for cylinders of arbitrary sizes can be derived. We also provide some examples of admissible weights $\beta$. 
\subsection{Examples}  Theorems \ref{main-theorem}, \ref{inter-theorem}, \ref{inter-theorem-bdry}) are applicable when $\bA(x,t) = \beta(x) \bar{\bA}(t)$ where $\bar{\bA}(t)$ is an $n\times n$ symmetric matrix of measurable functions and it satisfies the uniformly elliptic and boundedness conditions, and $\beta$ are given below.
\begin{itemize}
\item[\textup{(a)}]  Let $\phi: \mathbb{R}^n\setminus \{0\} \rightarrow \mathbb{R}$ be defined by
\[
\phi(x) = \left\{
\begin{array}{ll}
-\ln|x| & \quad \textup{if} \quad |x| \leq e^{-1},\\
1& \quad \textup{otherwise}.
\end{array} \right.
\]
It is known that $\phi \in A_1 \subset A_{1+\frac{1}{n_0}}$ with $[\phi]_{A_1} = M_0$ for some $M_0 = M_0(n)$, see \cite[Example 7.1.8, p. 507]{Grafakos-2}. Moreover,  it follows from \cite[Remark 1.2]{FP-2} that there exists a constant $N = N(n)>0$ such that
\begin{equation*}
[[\phi]]_{\textup{BMO}}(B_{2r_0}, r_0) \leq \frac{N}{|\ln(4r_0)|} \quad \text{for} \quad  r_0 \in (0, \frac{1}{10e}).
\end{equation*}
Hence, by choosing $r_0$ sufficiently small and take $\beta = \phi$, our the theorems are applicable.
\item[\textup{(b)}] Let $\phi_\alpha(x) = |x|^\alpha$ with $x \in \mathbb{R}^n$. As in \cite[Example 7.1.7, p. 506]{Grafakos-2}, it is known that $\phi_\alpha \in A_{1+\frac{1}{n_0}}$ when $\alpha \in (-n,\frac{n}{n_0})$. In addition, as calculated in \cite{CMP, CJP}, there exist $M_0 = M_0(n)\geq 1$ and $N = N(n) >0$ such that
\[
[\phi_\alpha]_{A_{1+\frac{1}{n_0}}} \leq M_0, \quad \text{and} \quad [[\phi_\alpha]]_{\textup{BMO}}(B_2, 1) \leq N|\alpha|, \quad \forall \ |\alpha| \leq 1/2.
\]
Therefore, the main theorems are applicable for $\beta = \phi_\alpha$ with $|\alpha|$ is sufficiently small. 

\item[\textup{(c)}] For given $m$-distinct points $\{x_k\}_{k=1}^m$ in $B_1$,  $c_k >0$, and $|\alpha_k|$ sufficiently small for $k \in \{1, 2,\ldots, m\}$, we define
\begin{align*}
\beta_1(x)  & =  \prod_{k=1}^{m}  \phi(x-x_k), \quad \beta_2(x) = \sum_{k=1}^{m} c_k \phi_{\alpha_k}(x-x_k), \quad \text{and} \\\
\beta_3(x) & = \prod_{k=1}^{m} \phi_{\alpha_k}(x-x_k) \quad \text{for}\quad x \in \mathbb{R}^n.
\end{align*}
By using localization, dilation, translation, and a covering argument, it follows from (a) and (b) that the main theorems are applicable with $\beta \in \{\beta_1, \beta_2, \beta_3\}$.
\end{itemize}
\subsection{Literature Review and State of the Art} 
 Regularity theory for second-order equations in divergence-form with singular-degenerate coefficients was initiated by Fabes, Kenig, and Serapioni in \cite{Fabes, Fabes-1}. In these papers, the authors considered the equations $\operatorname{div}(\mathbf{A}(x)\nabla u) = 0$, where the coefficient matrix $\mathbf{A}(x)$ satisfies the conditions in \eqref{ellip-cond} with an $A_2$ weight $\beta$. By establishing weighted Sobolev inequalities and utilizing Moser's iteration technique, they proved  interior Harnack inequalities, and then deducing that weak solutions are locally H\"{o}lder continuous. In the parabolic setting, this framework accommodates space-time operators where the time derivative and spatial diffusion are balanced against singular or degenerate Muckenhoupt weights. See, for instance, \cite{Chia-1, Chia-2, Chia-3}, where Harnack inequalities and H\"{o}lder regularity estimates are established under more restrictive assumptions on the weights. 

\smallskip
When the coefficients possess higher regularity adapted to the weights, a corresponding weighted Schauder theory ensures that solutions gain regularizing benefits within anisotropic, weight-dependent H\"{o}lder spaces that capture the underlying geometry of the degeneracies or singularities. These refinements have been studied in recent works such as \cite{AFV2, AFV, STV1, STV2}, which analyze operators featuring a explicit directional degeneracy of the form $\mathbf{A}(x) = x_n^\alpha \bar{\mathbf{A}}(x)$ for $x = (x', x_n) \in \mathbb{R}^{n-1} \times (0, \infty)$, where $\bar{\mathbf{A}}$ is uniformly elliptic, bounded, and sufficiently smooth, and the exponent satisfies $\alpha \in (-1, \infty)$ or $\alpha \in (-\infty, 1)$ depending on the prescribed boundary conditions. 

\smallskip
Transitioning from qualitative H\"{o}lder continuity to quantitative integrability, the development of a Calder\'on-Zygmund theory for equations in divergence-form with singular-degenerate coefficients requires adapting classical harmonic analysis to weighted functional spaces. For a degenerate elliptic system or equation $\text{div}(\bA(x)\nabla u) = \text{div}(F)$ where $\bA(x)$ satisfies \eqref{ellip-cond} with a Muckenhoupt $A_2$ weight $\beta$, it requires that the matrix $\bA(x)$ does not oscillate too violently in order to for  the gradient $\nabla u$ to be in higher weighted Lebesgue spaces $L^p(\Omega, \beta)$. Under a smallness conditions on the weighted mean oscillation of the coefficients $\bA$ with respect to the weight $\beta$, recent works starting \cite{CMP, CMP-1} have successfully established weighted $W^{1,p}$-estimates of Calder\'on-Zygmund type for elliptic equations and systems. See also \cite{BDGP} for a development in case of matrix weights with a different type of weighted estimates, and \cite{DPS} for results on quasi-linear elliptic equations with the weights as the power of the distance to the considered spatial domains which may not be in $A_2$. In the parabolic setting, this manifests as a maximal $L^p$-regularity framework where the time derivative $\beta u_t$ and the spatial diffusion term $\text{div}(\beta \bar{\bA} \nabla u)$ are controlled in weighted spaces, establishing existence, uniqueness, and optimal solvability under minimal regularity assumptions on the underlying Muckenhoupt weights. However, this direction is mostly achieved when the weight $\beta$ is specific such as the power of the distance to the considered spatial domains. A typical example in this framework is the following class of parabolic equations
\begin{equation} \label{DP-eqn}
\mu(x_n) u_t - \text{div}(\mu(x_n) \bar{\bA}(x,t) \nabla u) = \text{div}[\mu(x_n)F].
\end{equation}
where $x = (x', x_n) \in \mathbb{R}^{n-1} \times (0, \infty)$, $\mu(x_n) \sim x_n^\alpha$ with $\alpha \in \mathbb{R}$, and $\bar{\bA}$ is uniformly elliptic and bounded. Here, note that the singularity or degeneracy appears only on the boundary $\{x_n =0\}$ of the considered domain. Theories on well-posedness and regularity estimates in mixed-norm weighted Sobolev spaces have recently been developed in the series of papers \cites{DP, DP-0, DP1, MNS-1} under various conditions on $\alpha$ depending on  boundary conditions on $\{x_n =0\}$ for both equations in divergence-form as in \eqref{DP-eqn} and its corresponding non-divergence form.

\smallskip
In contrast to the unified framework of \eqref{DP-eqn}, isolating the weights exclusively within either the time derivative or the diffusion terms introduces distinct structural challenges. Consequently, both the selection of appropriate functional spaces and the development of a regularity theory in Sobolev spaces become highly non-trivial. Indeed, the classical work \cite{Chia-1} demonstrates that standard H\"{o}lder regularity estimates can fail in these mismatched settings. This obstacle has spurred recent investigations, such as those in \cite{DPT2,FP-1} for equations in divergence-form and in \cite{ DPT2, CJP, FP-2} for equations in non-divergence form. In particular, \cite{FP-1} establishes the well-posedness and regularity estimates in suitable weighted Sobolev spaces for the following class of equations under a partial VMO condition on $\bar{\mathbf{A}}$:
\begin{equation} \label{FP-2.eqn}
\omega(x) u_t - \operatorname{div}[\bar{\mathbf{A}}(x,t) \nabla u] = \operatorname{div}[F(x,t)]
\end{equation}
where $\bar{\mathbf{A}}(x,t)$ is uniformly elliptic and bounded, and the weight $\omega$ satisfies $\frac{1}{\omega} \in A_{1+\frac{2}{n_1}}$ with $n_1 = \max\{n, 2\}$.  The special case when $\omega(x) = x_n^\alpha$ for $x = (x', x_n) \in \mathbb{R}^{n-1} \times (0, \infty)$ and $\alpha \in (-2, 0)$, similar results in mixed-norm weighted Sobolev spaces are proved in \cite{DPT2} for equations in divergence-form and in \cite{ DPT1} for the corresponding ones in non-divergence form.

\smallskip
It is worth noting that the class of equations \eqref{FP-2.eqn} structurally departs from our model \eqref{main-eqn}. More precisely, the singularity or degeneracy in \eqref{main-eqn} are embedded directly within the spatial diffusion coefficients inside the divergence operator, meanwhile the singularity or degeneracy in \eqref{FP-2.eqn} are due to the volumetric heat capacity coefficients. Consequently, the functional spaces required here for \eqref{main-eqn} fundamentally differ from those developed in \cite{FP-1, DPT2}. The results of the present paper are entirely new and naturally complement this existing line of literature. Interestingly, in the non-divergence-form setting, the counterparts to these two classes of equations become equivalent. We refer interested readers to \cite{CJP, FP-2}, where Krylov-Safonov Harnack inequalities and $W^{2,\epsilon}$ Lin-type estimates are successfully established.

\smallskip
To conclude this subsection, we note that from a different geometric perspective, a parallel line of research considers settings where the minimum and maximum eigenvalues of $\mathbf{A}(x)$ satisfy specific integrability conditions rather than Muckenhoupt structural bounds. H\"{o}lder regularity estimates in this framework were pioneered in the classical work of Trudinger \cite{Tru-1, Tru-2}, with recent breakthroughs achieved in \cite{BS}. However, the Sobolev theory for the equations in the setting \cite{Tru-1, Tru-2} remains completely open.
\subsection{Main Difficulties and Key Strategies} 
We employ the freezing coefficient technique in conjunction with the level-set method introduced by Caffarelli and Peral \cite{CP}. Within the framework of well-posedness and regularity estimates in Sobolev spaces, the primary challenge in analyzing the class of equations \eqref{main-eqn} stems from a structural imbalance: weights appear exclusively in either the time derivative or the diffusion terms. This asymmetry is clearly reflected in the Caccioppoli inequalities, where both weighted and unweighted $L^2$-norms appear on the right-hand side (see Lemma \ref{Caccio-1} and Lemma \ref{bdry-Caccioppoli} below). This phenomenon is absent in the elliptic setting \cite{CMP, CMP-1, MP, BDGP} as well as in the classes of parabolic equations studied in \cite{DP, DP1, DP3}. To overcome this difficulty, we identify a suitable Muckenhoupt class of weights and establish corresponding weighted inequalities to control these mismatched terms. Furthermore, we introduce a family of inhomogeneous parabolic cylinders $\{Q_{\rho, \beta}(z)\}$ that is invariant under the scaling properties of \eqref{main-eqn} to inherently capture the imbalance in the these mismatched terms, adapting an approach successfully utilized in recent works \cite{CJP, FP-1, FP-2}. 

\smallskip
We note that locally, weak energy solutions to \eqref{main-eqn} naturally belong to the weighted Sobolev space $\mathcal{W}^{1,2}(Q_{1,\beta}, \beta)$, whereas weak solutions to the corresponding frozen coefficient equations reside in unweighted spaces $\mathcal{W}^{1,p_0}(Q_{1,\beta})$ for some $p_0 \in (1,2)$. Consequently, to establish the local closeness of these two solutions in weighted $L^2(Q_{1, \beta}, \beta)$, we implement a compactness argument. This limiting procedure yields sequences of coefficient matrices $\{\mathbf{A}_k\}$, weights $\{\beta_k\}$, and weak solutions $u_k \in \mathcal{W}^{1,2}(\Omega_T, \beta_k)$. To prevent the underlying functional spaces $\mathcal{W}^{1,2}(\Omega_T, \beta_k)$ from varying along the sequence, we prove various weighted inequalities that hold uniformly with respect to the changing measure $\beta_k(x) dxdt$. To this end, a new and special type of Poincar\'e-type inequality for weak solutions of \eqref{main-eqn} is established (see Lemmas \ref{L-2-inter-u-est} and \ref{L-2-bdry-u-est}). In addition, a parabolic duality argument, adapted from the elliptic framework introduced by H.~Brezis \cite{Brezis}, is employed to obtain the regularity required for weak $\mathcal{W}^{1, p_0}$-solutions of the frozen-coefficient equations. Ultimately, the family of inhomogeneous parabolic cylinders $\{Q_{\rho, \beta}(z)\}$ plays a crucial role both in implementing this duality argument and in executing the level-set method.
\subsection{Paper Organization}  The rest of the paper is organized as follows. In Section~\ref{wei-ses}, we recall the definitions of Muckenhoupt weight classes, summarize their essential 
properties, and introduce the relevant function spaces. We also state and  prove several key weighted inequalities in this section. Section~\ref{Lips-est-sec} is devoted to establishing interior and boundary Lipschitz estimates for $\W^{1, p_0}$-solutions of equations with frozen coefficients, where $p_0 \in (1, 2)$, utilizing a parabolic duality argument adapted from the elliptic framework of H.~Brezis~\cite{Brezis}. In Section~\ref{interior section}, we establish the local interior energy estimates and implement local interior perturbation techniques, culminating in the proof of Theorem~\ref{inter-theorem}. Section~\ref{bdr-section} presents the corresponding boundary theory and the proof of Theorem~\ref{inter-theorem-bdry}. Finally, in Section~\ref{global-section}, we complete the proof of Theorem~\ref{main-theorem} by combining Theorem~\ref{inter-theorem} and Theorem~\ref{inter-theorem-bdry} with boundary flattening techniques and a partition of unity.
\section{preliminary analysis and definitions}  \label{wei-ses}
\subsection{Muckenhoupt  Weights and Weighted Inequalities} Let us begin by recalling with the definition of $A_q$ class of Muckenhoupt  weights.
\begin{definition} \label{A-p-def}  Let $q \in (1, \infty)$, and $\mu : \mathbb{R}^n \rightarrow [0, \infty)$ be a locally integrable function. We say that $\mu$ belongs to the class of \emph{$A_q$ Muckenhoupt  weights} if $[\mu]_{A_q} < \infty$, where
\[
[\mu]_{A_q}=\sup_{r>0,\, x_0 \in \mathbb{R}^n}\left(\fint_{B_r(x_0)}\mu(x)\, dx \right)\left(\fint_{B_r(x_0)}\mu(x)^{-\frac{1}{q-1}}\, dx\right)^{q-1}.
\]
\end{definition}
\noindent
Here, we note that $B_\rho(x_0)$ denotes the ball in $\mathbb{R}^n$ of radius $\rho>0$ and centered at $x_0 \in \mathbb{R}^n$. Note that if $\beta \in A_{1+\frac{1}{n_0}}$, then it follows directly from the definition that $\bar{\beta}(x) = [\beta(x)]^{-n_0} \in A_{n_0+1}$. As such, $\bar{\beta}$ satisfies several important properties of Muckenhoupt  weights as demonstrated in \cite[Proposition 7.2.8, p. 521]{Grafakos-2}.
\begin{lemma} \label{property} Let $M_0 \geq 1$, and $\beta$ be a weight satisfying \eqref{beta-cond}. Define $\bar{\beta}(x) = [\beta(x)]^{-n_0}, \ x \in \R^n$. The following assertions hold.
\begin{itemize}
\item[\textup{(i)}] There exists a constant $\eta_0 = \eta_0(n, M_0) > 1$, called the doubling constant, such that
\begin{equation*}
\bar{\beta}(B_{2\rho}(x)) \leq \eta_0\, \bar{\beta}(B_\rho(x)) \quad \text{for all balls } B_\rho(x) \subset \mathbb{R}^n.
\end{equation*}
\item[\textup{(ii)}] There exist constants $N = N(n, M_0) > 0$ and $\zeta_0 = \zeta_0(n, M_0) \in (0,1)$ such that
\begin{equation*}
\bar{\beta}(S_1) \leq N \left( \frac{|S_1|}{|S_2|} \right)^{\zeta_0} \bar{\beta}(S_2)
\end{equation*}
for every ball $S_2\subset\mathbb{R}^n$ and every measurable set $S_1\subset S_2$.
\item[\textup{(iii)}] For $\theta \in (0,1)$, there exists a constant $\eta = \eta(n, M_0, \theta) \in (0,1)$ such that
\[
\bar{\beta}(S_1) \leq \eta\, \bar{\beta}(S_2)
\]
for every ball $S_2\subset\mathbb{R}^n$ and every measurable set $S_1\subset S_2$ with $|S_1| \leq \theta\, |S_2|$. 
\end{itemize}
\end{lemma}
\noindent
The following lemma on the reverse H\"{o}lder property of $A_q$-weights is well known. We also refer the reader to~\cite[Theorem~7.2.2, p.~514; Corollary~7.2.6, p.~519]{Grafakos-2} for the proof.
\begin{lemma}\label{R-Holder} For every $M_0 \geq 1$, $q \in (1, \infty)$, there exist a sufficiently small constant $\gamma = \gamma(n,q, M_0)>0$ and a constant $N = N(n, q, M_0)\geq 1$ such that the following assertions hold.
\begin{itemize}
\item[\textup{(i)}] If $\mu \in A_{q}$ satisfies $[\mu]_{A_q} \leq M_0$, then for every ball $B\subset\mathbb{R}^n$,
\[
\left( \fint_{B} \mu(x)^{1+\gamma} dx \right)^{\frac{1}{1+\gamma}} \leq N \fint_{B} \mu(x) dx.
\]

\item[\textup{(ii)}] If $\mu \in A_{q}$ satisfies $[\mu]_{A_q} \leq M_0$, then $\mu \in A_{q-\gamma}$ and
\[ [\mu]_{A_{q-\gamma}} \leq N.
\]
\end{itemize}
\end{lemma}

\smallskip
Next, we state and prove the following lemma that provides the relationships between the weighted Lebesgue spaces and the un-weighted Lebesgue spaces.
\begin{lemma}\label{L-q-2-wei} Let $q \in (1, \infty)$ and $M_0 \geq 1$. There exist a sufficiently small constant $\gamma = \gamma(n,  q, M_0) \in (0, q-1)$ and a constant $ N = N(n, q, M_0) > 0$ such that for every $ \mu \in A_q $ with $[\mu]_{A_q} \leq M_0 $ the following assertions hold
\begin{itemize}
\item[\textup{(i)}] For every ball $ B \subset \mathbb{R}^n $, $p\in[1,\infty)$, and for $ g \in L^p(B, \mu)$,
\[
\left( \fint_{B} |g(x)|^{\frac{p}{q - \gamma}} \, dx \right)^{\frac{q - \gamma}{p}} 
\leq N \left( \frac{1}{\mu(B)} \int_{B} |g(x)|^p \mu(x) \, dx \right)^{1/p}.
\]
\item[\textup{(ii)}] For every ball $ B \subset \mathbb{R}^n $ and for $ g \in L^2(B, \mu)$
\[
\left(\fint_{B} |g(x) \mu(x)|^{q_0} \, dx \right)^{\frac{1}{q_0}} 
\leq N (\mu)_B\left( \frac{1}{\mu(B)} \int_{B} |g(x)|^2 \mu(x) \, dx \right)^{1/2}
\]
for $q_0 = \frac{2(1+\gamma)}{2+\gamma} >1$, and $(\mu)_B = \frac{1}{|B}\int_{B}\mu(y) dy$.
\end{itemize}
The same assertions also holds for upper-half balls $B= B_\rho^+(x_0) = B_\rho(x_0) \cap (\mathbb{R}^{n-1} \times (0, \infty)$ with $x_0 \in \mathbb{R}^{n-1} \times [0, \infty)$.
\end{lemma}
\begin{proof} By Lemma \ref{R-Holder}, there exists a sufficiently small constant $\gamma = \gamma(n, q, M_0) \in (0, q-1)$ such that 
\begin{equation}\label{mu-q-gamma-1}
[\mu]_{A_{q-\gamma}} \leq N_0 \quad \text{and} \quad \left( \fint_{B} \mu(x)^{1+\gamma} dx \right)^{\frac{1}{1+\gamma}} \leq N \fint_{B} \mu(x) dx.
\end{equation}
where $N_0 = N_0(n, q, M_0)\geq 1$. 

\smallskip
To prove (i), we write
\[
\fint_{B} |g(x)|^{\frac{p}{q-\gamma}} dx = \fint_{B} \bigl[|g(x)|^p\mu(x)\bigr]^{\frac{1}{q-\gamma}} \mu(x)^{-\frac{1}{q-\gamma}}dx,
\]
and apply H\"{o}lder's inequality together with \eqref{mu-q-gamma-1} to obtain
\begin{align*}
\left(\fint_{B} |g(x)|^{\frac{p}{q-\gamma}} dx \right)^{\frac{q-\gamma}{p}}& \leq \left(\fint_{B} |g(x)|^p \mu(x) dx \right)^{\frac{1}{p}} \left(\fint_{B} \mu(x)^{-\frac{1}{q-\gamma-1}} dx \right)^{\frac{q- \gamma-1}{p}} \\
& \leq [\mu]_{A_{q-\gamma}}^{\frac{1}{p}}\left(\fint_B\mu(x)dx\right)^{-\frac{1}{p}} \left(\fint_{B} |g(x)|^p \mu(x) dx \right)^{\frac{1}{p}} \\
& \leq N\left(\frac{1}{\mu(B)}\int_{B} |g(x)|^p \mu(x) dx \right)^{\frac{1}{p}},
\end{align*}
where $N = N(n, q, p, M_0)\geq 1$. Hence, (i) is proved.

\smallskip 
We now prove (ii). By the H\"{o}lder inequality and Lemma \ref{R-Holder}, we have
\begin{align*}
\left( \fint_{B} |g(x) \mu(x)|^{q_0} \, dx\right)^{\frac{1}{q_0}} & \leq \left( \fint_{B} |g(x)|^2 \mu(x) \, dx \right)^{\frac{1}{2}} \left(\fint_{B} \mu(x)^{1+\gamma} \right)^{\frac{1}{2(1+\gamma)}} \\
&\leq N \left( \fint_{B} |g(x)|^2 \mu(x) \, dx \right)^{\frac{1}{2}} \left(\fint_{B} \mu(x) \right)^{\frac{1}{2}} \\
& = N \left( \frac{1}{\mu(B)}\int_{B} |g(x)|^2 \mu(x) \, dx \right)^{\frac{1}{2}} \left(\fint_{B} \mu(x) \right).
\end{align*}
Note also that the proof of the assertions for upper-half balls follow the same way. The proof of the lemma is completed.
\end{proof}

\smallskip
The rest of this subsection is about weighted Poincar\'{e} and Sobolev-Poincar\'{e} inequalities. Let us denote $\eta \in C_c^\infty(B_1)$ be the standard cut-off function satisfying
\begin{equation} \label{eta-def}
\eta ={c_n} \ \text{in} \ B_{4/5}, \ 0 \leq \eta \leq c_n \ \text{and} \  |\nabla \eta| \leq N_0 \ \text{in} \ B_1, \ \text{and} \ \fint_{B_1} \eta(x)^2 dx =1.
\end{equation}
Here $1\leq c_n\leq (5/4)^{n/2}$ and $N_0 = N_0 (n)>0$ are fixed numbers.  For each $r >0$ and $x_0 \in \mathbb{R}^n$, let us denote
\[
\eta_{x_0, r}(x) = \eta((x- x_0)/r).
\]
We note that
\begin{equation} \label{integral-eta-r}
\int_{B_r(x_0)} \eta_{x_0, r}(x)^2 dx= |B_r(x_0)| \fint_{B_1}\eta(x)^2 dx =|B_r(x_0)|.
\end{equation}
Then, for any function $u$ defined on a neighborhood of $B_r(x_0)$, we denote its average on $B_r(x_0)$ with respect to the weight $\eta_{x_0, r}(x)^2$ that was introduced in \cite{GS}
\begin{equation} \label{average-eta-weight}
[u]_{x_0, r} =  \fint_{B_r(x_0)} u(x) \eta_{x_0, r}(x)^2 dx.
\end{equation}
For readers convenience, we recall that for a given non-negative locally finite Borel measure $\omega$ on $\mathbb{R}^n$ and an open non-empty bounded set $E \subset \mathbb{R}^n$, we write
\[
\fint_{E} f(x) d\omega(x)= \frac{1}{\omega(E)} \int_E f(x) d\omega(x) \quad \text{where} \quad \omega(E) = \int_E \omega(x) dx.
\]
In particular,
\[
\fint_{E} f(x) dx = \frac{1}{|E|} \int_{E}f(x) dx \quad \text{and} \quad \fint_{E} f(x) \beta(dx) = \frac{1}{\beta(E)}\int_{E}  f(x) \beta(x) dx.
\]
Below, we also note that $B_r = B_r(0)$ and $B_r^+ = B_r \cap (\mathbb{R}^{n-1} \times (0, \infty))$ with $r>0$. The following weighted Sobolev-Poincar\'{e} inequalities are needed in the paper.
\begin{lemma} \label{Sobolev-Poincare} For every $M_0 \geq 1$, there is a sufficiently small positive number $\gamma_0 = \gamma_0(n, M_0)$ such that for every $p \in [1, \frac{n}{n-1} +\gamma_0)$ when $n\geq2$ and $p\in[1,\infty)$ when $n=1$, there is $N = N(n, M_0, p)>0$  such that the following assertions hold for every $\beta \in A_2$ with $[\beta]_{A_2} \leq M_0$.
\begin{itemize}
\item[\textup{(i)}] For every $\rho \in (0,1]$, and for every $u \in C^1(B_r)$, we have
\[
\left(\fint_{B_r} |u(x) - [u]_{0, r\rho}|^{2p} \beta(dx)  \right)^{\frac{1}{2p}}  \leq N\big(1+ \rho^{-n}\big) r \left(\fint_{B_r} |\nabla u(x)|^2 \beta(dx)\right)^{\frac{1}{2}}.
\]
\item[\textup{(ii)}] For every $u \in C^1(B_r^+)$ satisfying $u_{|B_r \cap (\mathbb{R}^{n-1} \times \{0\})}=0$, we have
\[
\left(\fint_{B_r^+} |u(x)|^{2p} \beta(dx)  \right)^{\frac{1}{2p}}  \leq Nr \left(\fint_{B_r^+} |\nabla u(x)|^2 \beta(dx)\right)^{\frac{1}{2}}.
\]
\end{itemize}
 \end{lemma}
\begin{proof}  Note that the assertion (ii) is due to \cite[Theorem 1.6]{Fabes}. Hence, we only need to prove (i). Using the dilation, it is sufficient to assume that $r =1$.  Note also that from \cite[Theorem 1.5]{Fabes}, there are $N = N(n, M_0, p)$ and $\gamma = \gamma(n, M_0) \in (0,1)$ such that
\begin{equation} \label{Poincare-1.eqn}
\left(\fint_{B_r} |u(x) - \bar{u}|^{2p} \beta(dx)  \right)^{\frac{1}{2p}}  \leq N r \left(\fint_{B_r} |\nabla u(x)|^2 \beta(dx)\right)^{\frac{1}{2}}
\end{equation}
where $\bar{u}$ is either 
\[
\fint_{B_r} u(x) dx \quad \text{or} \quad  \fint_{B_r} u(x) \beta(dx).
\]
Let us recall that
\[
(u)_{B_1} = \frac{1}{|B_1|}\int_{B_1} u(x) dx \quad \text{and}\quad [u]_{0, \rho} =\fint_{B_\rho} u(x) \eta_{0,\rho}(x)^2 dx.
\]
Then, it follows from triangle inequality that
\begin{align*}
& \left(\fint_{B_1}|u(x) -[u]_{0, \rho}|^{2p} \beta(dx)\right)^{1/2p} \\
& \leq \left(\fint_{B_1} |u(x) - (u)_{B_1}|^{2p} \beta(dx) \right)^{1/(2p)} + \left(\fint_{B_1} |(u)_{B_1} - [u]_{0, \rho}|^{2p} \beta(dx)\right)^{1/2p}\\
& =\left(\fint_{B_1} |u(x) - (u)_{B_1}|^{2p} \beta(dx) \right)^{1/2p} +  |(u)_{B_1} - [u]_{0, \rho}| \\
& \leq N \left(\fint_{B_1}|\nabla u(x)|^2 \beta(dx) \right)^{1/2} +  |(u)_{B_1} - [u]_{0, \rho}|,
\end{align*}
where in the last step we used \eqref{Poincare-1.eqn}. It then remains to control the last term in the right hand side of the last estimate. It follows from the definition of $[u]_{0,\rho}$, \eqref{integral-eta-r}, and H\"{o}lder inequality
\begin{align*}
|(u)_{B_1} - [u]_{0, \rho}| & = \left| (u)_{B_1} -  \frac{1}{|B_\rho|} \int_{B_1} u(x) \eta_{0, \rho}(x)^2 dx \right| \\
& \leq  \frac{1}{|B_\rho|}\int_{B_1} | u(x) - (u)_{B_1}| \eta_{0,\rho}(x)^2 dx  \\
&=  \frac{1}{|B_\rho|}\int_{B_1} | u(x) - (u)_{B_1}| \beta(x)^{1/2} \eta_{0,\rho}(x)^2 \beta(x)^{-1/2} dx\\
& \leq \frac{1}{|B_\rho|} \|u - (u)_{B_1}\|_{L^{2}(B_1, \beta)} \left( \int_{B_1} \eta_{0,\rho}(x)^{4} \beta(x)^{-1} dx \right)^{\frac{1}{2}} \\
& \leq \frac{N}{\rho^n} \left(\fint_{B_1} |\nabla u(x){|}^2 \beta(dx) \right)^{\frac{1}{2}} \left(\fint_{B_1} \beta(x)dx \right)^{\frac{1}{2}}  \left( \fint_{B_1} \beta(x)^{-1} dx \right)^{\frac{1}{2}} \\
& \leq \frac{N M_0}{\rho^n}\left(\fint_{B_1} |\nabla u(x){|}^2 \beta(dx) \right)^{\frac{1}{2}}.
\end{align*}
Note that in the previous steps, we also used \eqref{Poincare-1.eqn}, and the fact that $0 \leq \eta_{0,\rho} \leq c_n$ as in \eqref{eta-def}. The proof of the lemma is completed.
\end{proof}

The following embedding result is essential in the paper, and its proof requires $\beta \in A_{1+\frac{1}{n_0}}$ with $n_0 = \max\{n-1, 1\}$.
\begin{lemma} \label{imbed-wei-0309} For every $M_0 \geq 1$, there are a sufficiently small number $\kappa = \kappa (n, M_0) \in (0,1)$ and a number $N= N(n, M_0)>0$ such that for every $\beta \in A_{1+ \frac{1}{n_0}}$ satisfying $[\beta]_{A_{1+\frac{1}{n_0}}} \leq M_0$ for $n_0 = \max\{1, n-1\}$, the following holds.
\begin{itemize}
\item[\textup{(i)}] For all $u \in W^{1,2}(B_r, \beta)$ and $\rho \in (0,1]$, we have
\[
\left(\fint_{B_r}|u(x) -[u]_{0, \rho r}|^{2+\kappa} dx\right)^{\frac{1}{2+\kappa}} \leq  rN(1+\rho^{-n}) \left(\fint_{B_r}|\nabla u(x)|^2 \beta(dx)\right)^{\frac{1}{2}}.
\]
\item[\textup{(ii)}] For all $u \in W^{1,2}(B_r^+, \beta)$ such that $u_{|B_r \cap (\mathbb{R}^{n-1} \times \{0\})} =0$ in sense of trace, we have
\[
\left(\fint_{B_r^+}|u(x)|^{2+\kappa} dx\right)^{\frac{1}{2+\kappa}} \leq  rN \left(\fint_{B_r^+}|\nabla u(x)|^2 \beta(dx)\right)^{\frac{1}{2}}.
\]
\end{itemize}
\end{lemma}
\begin{proof} We only provide the proof of (i) as that of (ii) is similar using Lemma \ref{Sobolev-Poincare}-(ii) instead of Lemma \ref{Sobolev-Poincare}-(i). It is sufficient to prove the lemma with $r=1$ and $u \in C^1(B_1)$. We only consider the case $n \geq 2$ as the case $n=1$ the proof is simpler. Note that when $n \geq 2$, we have $n_0 = n-1$. Also $\beta \in A_{1+\frac{1}{n-1}} \subset A_2$, it follows from Lemma \ref{Sobolev-Poincare}-(i) that there is $\gamma_0 = \gamma_0(n, M_0) \in (0,1)$ sufficiently small that
\[
\left(\fint_{B_1}|u(x)-[u]_{0,\rho}|^{2p} \beta(dx) \right)^{\frac{1}{2p}} \leq N(1+ \rho^{-n})\left(\fint_{B_1} |\nabla u(x)|^2 \beta(dx) \right)^{\frac{1}{2}},
\]
for $N = N(n, M_0)>0$ and for $p = 1+ \frac{1}{n-1} + \frac{\gamma_0}{2}$. On the other hand, by using Lemma \ref{L-q-2-wei}-(i) with $q = 1+ \frac{1}{n-1}$, we can find a sufficiently small number $\gamma = \gamma(n, M_0)\in (0,1)$ that
\[
\left(\fint_{B_1}|u(x)- [u]_{0,\rho}|^{\frac{2p}{q-\gamma}}  dx \right)^{\frac{q-\gamma}{2p}} \leq  N \left(\fint_{B_1}|u(x) - [u]_{0,\rho}|^{2p} \beta(dx) \right)^{\frac{1}{2p}}.
\]
Note that with  $q = 1+ \frac{1}{n-1}$, and $p = 1+\frac{1}{n-1} + \frac{\gamma_0}{2}$, we can find some sufficiently small number $\kappa = \kappa (n, M_0) \in (0,1)$ so that
\[  \frac{2p}{q-\gamma} = \frac{2(q+\gamma_0/2)}{q-\gamma} = 2 + \kappa. \]
The proof of the lemma is completed.
\end{proof}
As a corollary of Lemma \ref{imbed-wei-0309}, we have the following interpolation inequality showing that $W^{1,2}(\Omega, \beta) \subset L^2(\Omega)$ when $\beta \in A_{1+\frac{1}{n_0}}$.
\begin{lemma} \label{imbed-epsilon-0309} For every $M_0 \geq 1$, there are numbers $\theta=\theta(n,M_0)\in(0,1)$ and $N=N(n,M_0)>0$ such that for every $\beta \in A_{1+ \frac{1}{n_0}}$ satisfying $[\beta]_{A_{1+\frac{1}{n_0}}} \leq M_0$ with $n_0 = \max\{1, n-1\}$, we have the followings
\begin{itemize}
\item[\textup{(i)}] For every $u \in W^{1,2}(B_r, \beta)$ with some $r >0$, and for $\rho \in (0,1]$, it holds that
\begin{align*}
& \left(\fint_{B_r}|u(x) - [u]_{0, \rho r}|^{2} dx\right)^{\frac{1}{2}} \\
& \leq  N r^{\theta} (1+\rho^{-n})^{\theta} \left(\fint_{B_r}|u(x) - [u]_{0,\rho r}|^2 \beta(dx)\right)^{(1-\theta)/2} \left( \fint_{B_r}|\nabla u(x)|^2 \beta(dx) \right)^{\frac{\theta}{2}}.
\end{align*}
\item[\textup{(ii)}] For every $r>0$ and $u \in W^{1,2}(B_r^+, \beta)$ satisfying $u_{|B_r \cap (\mathbb{R}^{n-1} \times \{0\})} =0$ in sense of trace, it holds that
\begin{equation*}
\left(\fint_{B_r^+}|u(x)|^{2} dx\right)^{\frac{1}{2}}  \leq  N r^{\theta} \left(\fint_{B_r^+}|u(x)|^2 \beta(dx)\right)^{(1-\theta)/2} \left( \fint_{B_r^+}|\nabla u(x)|^2 \beta(dx) \right)^{\frac{\theta}{2}}.
\end{equation*}
\end{itemize}
\end{lemma}
\begin{proof}  We only provide the proof of (i) as that of (ii) is similar using Lemma \ref{L-q-2-wei}-(i) and Lemma \ref{imbed-wei-0309}-(ii). It is sufficient to assume that $r=1$. Also, we only consider the case $n \geq 2$ and skip the case $n=1$ because the proof is similar. Note that as $n \geq 2$, $n_0 = n-1$. We write $w(x) = u(x) - [u]_{0, \rho}$. As $\beta \in A_{1+\frac{1}{n-1}} \subset A_2$, it follows from Lemma \ref{L-q-2-wei}-(i) that there is $p_0 = p_0(n, M_0) \in (1,2)$ sufficiently closed to $1$ such that
\begin{equation} \label{ineq-0309}
\left(\fint_{B_1} |w(x)|^{p_0} dx \right)^{\frac{1}{p_0}} \leq N \left(\fint_{B_1} |w(x)|^{2} \beta(dx) \right)^{\frac{1}{2}}
\end{equation}
for $N= N(n, M_0)>0$. Now, let $\kappa = \kappa (n, M_0) \in (0, 1)$ be the number defined in Lemma \ref{imbed-wei-0309}-(i). It follows from H\"{o}lder inequality, \eqref{ineq-0309}, and Lemma \ref{imbed-wei-0309} that
\begin{align*}
& \left(\fint_{B_1} |w(x)|^{2} dx \right)^{\frac{1}{2}} \leq \left(\fint_{B_1} |w(x)|^{p_0} dx \right)^{\frac{1-\theta}{p_0}} \left(\fint_{B_1} |w(x)|^{2+\kappa} dx \right)^{\frac{\theta}{2+\kappa}} \\
& \leq N (1+\rho^{-n})^{\theta} \left(\fint_{B_1} |w(x)|^{2} \beta(dx) \right)^{\frac{1-\theta}{2}}  \left(\fint_{B_1}|\nabla w(x)|^2 \beta(dx)\right)^{\frac{\theta}{2}}
\end{align*}
for some $\theta = \theta(p_0, \kappa) \in (0,1)$ satisfying $1/2=(1-\theta)/p_0+\theta/(2+\kappa)$. The proof is then completed.
\end{proof}
\subsection{Weighted Parabolic Cylinders}  \label{cylinder-distance-sec} Let us discuss more about the class of weighed parabolic cylinders introduced in \eqref{cylinder-introc}. Let $M_0 \geq 1$ be a fixed number and let $\beta: \mathbb{R}^n \rightarrow [0, \infty)$ be a weight satisfying \eqref{beta-cond}:
\begin{equation*} 
\beta \in A_{1+\frac{1}{n_0}} \quad \text{and} \quad [\beta]_{A_{1+\frac{1}{n_0}}} \leq M_0  \quad \text{for}\quad  n_0 =\max\{n-1,1\}.
\end{equation*}
For given $z_0 = (x_0, t_0) \in \R^{n}\times \R$ and for $r>0$, the \emph{weighted parabolic cylinder} of radius $r$ centered at $z_0$ is defined by
\begin{equation} \label{cylinder-def}
Q_{r, \beta}(z_0) =B_r(x_0)  \times \Gamma_{\beta}(z_0, r), \quad \text{where}   \quad \Gamma_{\beta} (z_0, r) =(t_0 -r^2\E_\beta(x_0, r), t_0]
\end{equation}
and in which
\begin{equation} \label{E-beta-def}
\mathcal{E}_\beta(x_0, r) = \left(\fint_{B_r(x_0)} \beta(x)^{-n_0} dx \right)^{\frac{1}{n_0}}, \quad n_0 = \max\{n-1, 1\}.
\end{equation}
It follows directly from \eqref{E-beta-def} that
\begin{equation}\label{heights}
r^2 \E_{\beta}(x_0, r) =\left\{
\begin{array}{cl}
   \frac{1}{2} r\, \bar{\beta}(B_r(x_0)) & \text{if}\quad n = 1,\\[6pt]
 \sigma_n^{-\frac{1}{n-1}} r^{\frac{n-2}{n-1}} \left[\bar{\beta}(B_r(x_0))\right]^{\frac{1}{n-1}} & \text{if}\quad n\geq 2.
\end{array}\right.
\end{equation}
Here, $\sigma_n$ denotes the Lebesgue measure of the unit ball $B_1\subset \R^n$, and
\begin{equation*}
\bar{\beta}(x) = [\beta(x)]^{-n_0}, \quad x \in \mathbb{R}^n.
\end{equation*}
Then, the following inclusion property holds
\[
Q_{r_1,\, \beta}(z_0) \subset Q_{r_2,\, \beta}(z_0), \quad \text{for\ } \, 0<r_1 \leq r_2, \  \forall\, z_0 \in \mathbb{R}^{n+1}.
\]

\smallskip
We also denote the upper-half ball in $\mathbb{R}^n$ by
\[
B_r^+(x_0) = \left\{x = (x', x_n) \in \mathbb{R}^{n-1} \times (0, \infty): x \in B_r(x_0)\right\},
\]
and 
its flat boundary portion by
\begin{equation} \label{flat-bdr}
T_r(x_0) =  B_r(x_0) \cap \big(\mathbb{R}^{n-1} \times \{0\}\big).
\end{equation}
In addition, the upper-half parabolic weighted cylinder of radius $r$ centered at $z_0=(x_0, t_0)$ is denoted by
\begin{equation*}
Q_{r, \beta}^+(z_0) = B_r^+(x_0) \times \Gamma_{\beta}(z_0, r).
\end{equation*}
Note also that when $z_0 =(0, 0)$ and $x_0 = 0$, we omit $z_0$ and $x_0$ in the notation. For example, we write
\[
 Q_{r, \beta} = Q_{r, \beta}(0,0), \  B_r = B_r(0),\  T_r = T_r(0), \ \E_\beta(r) = \E_{\beta}(0, r), \  \Gamma_{\beta}(r) = \Gamma_{\beta}(0,0, r).
\]
The same definitions can be defined for $Q_{r, \beta}^+$ and $B_r^+$.

\smallskip
It turned out that the weighted parabolic cyllinders can be derived from a quasi-distance function.  Indeed,  for a given $z_0=(x_0,t_0)\in\R^{n}\times \R$, we let $h_{x_0}: [0,\infty)\rightarrow [0,\infty)$ be defined by
\begin{equation*}
h_{x_0}(r)=\left\{
 \begin{array}{cl}
 0 \quad &\text{if}\quad r=0,\\[4pt]
r^2\E_{\beta}(x_0, r) \quad &\text{if}\quad r\in(0,\infty),\\
 \end{array}\right.
\end{equation*}
where $\E_{\beta, x_0}$ is given in \eqref{E-beta-def}.  From  \eqref{heights}, it follows that $h_{x_0}(r)$ is strictly increasing in $r$. Hence, its inverse function, $h^{-1}_{x_0}$, exists.  

\smallskip
Now, for any $z=(x,t)\in \R^{n} \times \R$ with $t\leq t_0$, we denote
\begin{equation*}
\begin{aligned}
\tilde \rho_{\beta}(z,z_0)
=\inf\big\{r>0: z\in Q_{r,\beta}(z_0)\big\}=\max\big\{|x-x_0|,\ h_{x_0}^{-1}(t_0-t)\big\}.
\end{aligned}
\end{equation*}
Then, for  any $z=(x,t)\in \R^{n}\times \R$, let
\begin{equation}\label{rho def}
\rho_{\beta}(z, z_0)=
\left\{
\begin{array}{rl}
\tilde \rho_{\beta}(z, z_0) \quad &\text{if}\quad t\leq t_0,\\[4pt]
\tilde \rho_{\beta}(z_0, z)\quad &\text{if}\quad t>t_0.
\end{array}\right.
\end{equation}
From \eqref{rho def}, it follows that
\begin{equation*}
|x-x_0|\leq\rho_{\beta}(z, z_0),\quad \text{for all}\quad z=(x,t),\ z_0=(x_0,t_0)\in \R^{n} \times \R.
\end{equation*}
Moreover, for any $z_0 = (x_0, t_0) \in \R^{n}\times \R$ and $r>0$, it holds that
\begin{equation*}
Q_{r,\beta}(z_0) = \{z = (x,t) \in \R^{n}\times \R: \quad \rho_{\beta} (z_0, z) <r \ \text{ and }\  t  \leq t_0\}.
\end{equation*}
\smallskip
From the discussion, it follows that $\rho_\beta$ is a quasi-distance function on $\R^{n}\times \R$. We state this in the following remark, and skip the proof as it is similar to that of \cite[Lemma 2.8]{FP-1}.
\begin{remark}\label{quasi-metric lemma} Suppose that \eqref{beta-cond} holds. Then, the function $\rho_\beta$ defined in \eqref{rho def}  is a quasi-metric. In particular, there exists a constant $\Lambda= \Lambda (n, M_0) \geq 1$ such that
\begin{equation*}
\rho_{\beta}(z_0,\bar{z})\leq \Lambda \big[\rho_{\beta}(z_0,z_1)+\rho_{\beta}(z_1,\bar{z})\big], \quad \forall z_0,\ z_1,\ \bar{z} \ \in \ \R^{n} \times \R.
\end{equation*}
\end{remark}

\subsection{Invariance Properties} \label{Fn-spaces} We discuss some dilations and scalings that are invariant for the class of equations \eqref{main-eqn}. For a given matrix $\bA$ satisfying \eqref{ellip-cond} with its corresponding $\beta$ satisfying \eqref{beta-cond}, let us recall the definitions of weak solutions.
\begin{definition} \label{interior-weak-def} Let $p \in (1, \infty)$, $r>0$, $z_0 = (x_0, t_0) \in \mathbb{R}^{n} \times \mathbb{R}$, and $Q = Q_{r, \beta}(z_0)$. A function $u \in \W^{1,p}(Q, \beta)$ is said to be a weak solution of
\begin{equation} \label{scaling-eqn}
 u_t - \textup{div}(\bA(x,t) \nabla u) = \textup{div}(\beta(x) F(x,t))\quad \text{in} \quad Q,
\end{equation}
if  
\begin{align*}
& -\int_{Q} u(x,t) \partial_t\varphi(x,t) dx dt + \int_{Q} \wei{\bA(x,t) \nabla u(x,t), \nabla \varphi(x,t)} dx dt \\
& = -\int_{Q} \wei{F(x,t), \nabla \varphi(x,t)} \beta(x) dx dt, \quad \forall \ \varphi \in C_0^\infty(Q).
\end{align*}
\end{definition}

\smallskip
To study local boundary estimates for solutions to \eqref{main-eqn}, we also need the following notion of weak solutions.
\begin{definition} \label{boundary-weak-def} Let $p \in (1, \infty)$, $r>0$, $z_0 = (x_0, t_0) \in \mathbb{R}^{n} \times \mathbb{R}$, and $Q^+ = Q_{r, \beta}^+(z_0)$. A function $u \in \hW^{1,p}(Q^+, \beta)$ is said to be a weak solution of
\begin{equation*} 
\left\{
\begin{aligned}
u_t - \textup{div}(\bA(x,t) \nabla u) & = \textup{div}(\beta(x)F(x,t)) \quad &&\text{in} \quad Q^+,\\[4pt]
u & =  0  \quad &&\text{on} \quad T_{r}(x_0) \times \Gamma_{\beta, z_0}(r),
\end{aligned} \right.
\end{equation*}
if 
\begin{align*}
& -\int_{Q^+} u(x,t) \partial_t\varphi(x,t) dx dt + \int_{Q^+} \wei{\bA(x,t) \nabla u(x,t), \nabla \varphi(x,t)} dx dt \\
& = -\int_{Q^+} \wei{F(x,t), \nabla \varphi(x,t)} \beta(x) dx dt, \quad \forall \ \varphi \in C_0^\infty(Q^+).
\end{align*}
\end{definition}
\smallskip
Besides the standard heat dilation $(x,t) \mapsto (rx, r^2t)$, we introduce two special dilation properties for \eqref{main-eqn}. Let us consider the equation \eqref{scaling-eqn}. Without lost of generality, we may assume without loss of generality that $z_0 = (0,0)$.
As defined in \eqref{E-beta-def}, we note that
\[ \E_{\beta} (r) = \E_{\beta, 0}(r) =  \left(\fint_{B_r} \beta(x)^{-n_0} \right)^{\frac{1}{n_0}}, \quad n_0 = \max\{n-1, 1\}.
\]
For $(x,t) \in Q_1=B_1\times (-1,0)$, let
\begin{equation} \label{scale-1}
\left\{
\begin{aligned}
\tilde{u}(x,t) &= u(rx, r^2\E_{\beta}(r) t),
\quad &&\tilde{\beta}(x) = \E_{\beta}(r) \beta(rx),\\[4pt]
\tilde{\bA} (x,t)&= \E_{\beta}(r) \bA(rx, r^2\E_{\beta}(r) t), \quad
&&\tilde{F}(x,t) =rF(rx, r^2\E_{\beta}(r) t).
\end{aligned} \right.
\end{equation}
The following only in the time variable dilation is special for the class of equations in \eqref{main-eqn}
\begin{equation} \label{scale-2}
\left\{
\begin{aligned}
\hat{u}(x,t) &= u(x, \E_{\beta}(r) t),
\quad &&\hat{\beta}(x) = \E_{\beta}(r) \beta(x),\\[4pt]
\hat{\bA} (x,t)&= \E_{\beta}(r)\bA(x, \E_{\beta}(r) t), \quad
&&\hat{F}(x,t) =F(x, \E_{\beta}(r) t).
\end{aligned}  \right.
\end{equation}
From \eqref{scale-1} and \eqref{scale-2}, it follows that
\[
\E_{\tilde \beta}(1)=1 \quad \text{and} \quad \E_{\hat{\beta}}(r) =1 \quad \text{so that}\quad Q_{1,\tilde \beta} = Q_1 \quad \text{and}  \quad Q_{r,\hat{\beta}} = Q_r.
\]
The following lemma will be used in the paper. Its proof is straightforward, and thus omitted.
\begin{lemma} \label{scaling-lemma} Suppose that $u \in \W^{1,p}(Q_{r,\beta}, \beta)$ is a weak solution of \eqref{scaling-eqn} in $Q_{r,\beta}$ with some $p \in (1, \infty)$. Then $\tilde u \in \W^{1,p}(Q_{1}, \tilde{\beta})$ is a weak solution of
\begin{equation*} 
\tilde{u}_t - \textup{div}(\tilde{\bA}(x,t) \nabla \tilde{u}) = \textup{div}(\tilde{\beta}(x)  \tilde{F}(x,t)) \quad \text{in} \quad Q_1,
\end{equation*}
where $\tilde{u}$, $\tilde{\beta}$, $\tilde{\bA}$, and $\tilde{F}$ are defined in \eqref{scale-1}. Similar assertions also hold for weak solutions  defined in Definition~\ref{boundary-weak-def}, and also for $\hat{u}$. 
\end{lemma}

\section{Lipschitz estimates for simple equations} \label{Lips-est-sec}
Let $\bar{\bf{A}}(t)$ be a measurable $n\times n$ symmetric matrix function defined on the open interval $\Gamma_{\beta }(2r)$ with some $r>0$. We assume that $\bar{\bf{A}}(t)$ satisfies the ellipticity and boundedness conditions: there exists a constant $\nu \in (0,1)$ such that
\begin{equation} \label{ellip-constant}
\nu (\beta)_{B_{2r}}|\xi|^2 \leq \wei{\bar{\bA}(t) \xi,\, \xi} \quad \text{and} \quad |\bar{\bA}(t)| \leq \nu^{-1}(\beta)_{B_{2r}}, \qquad \forall \ \xi \in \mathbb{R}^n
\end{equation}
with some $r>0$, $\Gamma_{\beta }(2r) = \Gamma_{\beta }((0,0), 2r)$ defined in \eqref{cylinder-def}, and 
\[(\beta)_{B_{2r}} = \fint_{B_{2r}} \beta(x) dx = \frac{1}{|B_{2r}|} \int_{B_{2r}} \beta(x) dx
\]
with a given $\beta$ satisfying \eqref{beta-cond}. We begin with the following interior Lipschitz estimates.
\begin{lemma} \label{Lipschitz-constant} For every $\nu \in (0,1), p_0 \in (1, 2)$, and $M_0 \geq 1$, there exists a constant $N = N(n, \nu, p_0, M_0)>0$ such that the following assertion holds.  Let $\beta$ be a weight satisfying \eqref{beta-cond}, and let $\bar{\bA}(t)$ be a matrix satisfying \eqref{ellip-constant} on $\Gamma_{\beta }(2r)$ with some $r>0$. Suppose that $v \in \W^{1,p_0}(Q_{2r, \beta})$ is a weak solution of
\begin{equation*} 
v_t - \textup{div}(\bar{\bA}(t)\nabla v) = 0 \quad \text{in} \quad Q_{2r, \beta}.
\end{equation*}
Then  $v \in \W^{1,2}_{\textup{loc}}(Q_{2r, \beta})$. Moreover, $\nabla v \in L^\infty_{\textup{loc}}(Q_{2r, \beta})$ and
\begin{equation} \label{Lip-inter-est-1013}
 \|\nabla v\|_{L^\infty(Q_{r, \beta})} \leq  N\Big(\fint_{Q_{2r, \beta}} |\nabla v(x,t)|^{p_0}\, dxdt  \Big)^{1/p_0}.
\end{equation}
\end{lemma} 
\begin{proof} By the dilation $(x,t) \mapsto (rx, r^2 \E_{\beta}(2r)t)$ and Lemma \ref{scaling-lemma}, it is sufficient to prove the lemma when
\[
r =1 \quad \text{and} \quad \E_{\beta}(2) = \left(\fint_{B_2} \beta(x)^{-n_0} \right)^{\frac{1}{n_0}} =1.
\]
In this setting, $Q_{2,\beta} = Q_2$, and  $v \in \W^{1,p_0}(Q_2)$ is a weak solution of 
\[
v_t - \textup{div}(\bar{\bA} (t) \nabla v)  = 0 \quad \text{in} \quad Q_{2}.
\]
In addition, we have
\[
1\leq \left(\fint_{B_2} \beta(x) dx \right) \left(\fint_{B_2} \beta(x)^{-n_0} dx \right)^{\frac{1}{n_0}}  =\fint_{B_2} \beta(x) dx \leq M_0.
\]
Therefore, the conditions in \eqref{ellip-constant} are reduced to
\begin{equation} 
\nu |\xi|^2 \leq \wei{\bar{\bA}(t)\xi, \xi} \quad \text{and} \quad |\bar{\bA}(t)| \leq \nu^{-1} M_0, \quad \forall\ t \in (-4, 0), \quad \forall \ \xi \in \mathbb{R}^n.
\end{equation}
To prove the first assertion in the lemma, it suffices to show $\nabla v \in L^2_{\text{loc}}(Q_2)$. To this end, we employ the duality argument introduced in \cite{Brezis} for the elliptic case. Let $d_0=\max\{n+2,4\}$ and define $p_1$ by
\begin{equation} \label{p-1-def-0710}
\frac{1}{p_1}=\max\left\{\frac12,\frac{1}{p_0}-\frac{1}{d_0}\right\}.
\end{equation}
By the parabolic Sobolev embedding theorem (see \cite[Lemma 3.1]{DP} for example) and the PDE of $v$, we see that
\begin{equation} \label{v-p_1-est-03-16}
\|v\|_{L^{p_1}(Q_2)} \leq N \|v\|_{\W^{1,p_0}(Q_2)} \leq N \Big[ \|v\|_{L^{1}(Q_2)} + \|\nabla v\|_{L^{p_0}(Q_2)}\Big].
\end{equation}
We claim that 
\begin{equation} \label{L-r-est-03016}
\|\nabla v\|_{L^{p_1}(Q_r)} \leq N \Big[ \|v\|_{L^{1}(Q_2)} + \|\nabla v\|_{L^{p_0}(Q_2)}\Big],
\end{equation}
for each $r \in (0,2)$ and $N = N(n, \nu,p_0, M_0, r)>0$.

\smallskip
Note that the lemma follows from the classical regularity theory when \eqref{L-r-est-03016} holds with $p_1 \geq 2$. However, if $p_1 \in (1,2)$, we let $p_2$ be defined as in \eqref{p-1-def-0710} using $p_1$ instead of $p_0$. If $p_2 \in (1, 2)$, we repeat the proof with $p_0$ replaced by $p_1$. We keep iterating on nested standard parabolic cylinders to obtain $\nabla v \in L^2_{\text{loc}}(Q_2)$. Hence, it is sufficient to prove \eqref{L-r-est-03016} with the assumption that $p_1 \in (1, 2]$.
 
 \smallskip
For a given $G \in C_0^\infty(Q_2)^n$, let $\phi \in \W^{1, p_1'}(Q_2)$ be the solution of the boundary value problem
\begin{equation} \label{phi-duality}
\left\{
\begin{array}{cccl}
-\phi_t -\text{div}(\bar{\bA}(t)^* \nabla \phi) & = & \text{div}(G) & \quad \text{in} \quad Q_2,\\
\phi & = &0 & \quad \text{on} \quad (\partial B_2\times(-4,0))\cup (B_2\times\{0\}).
\end{array}\right.
\end{equation}
Here, $p_1' \geq 2$ is the number satisfying $\frac{1}{p_1} + \frac{1}{p_1'} =1$. Note that by the standard regularity estimates and the Sobolev embedding (see \cite[Lemma 3.1]{DP} for example), we have
\begin{equation} \label{phi-est-03016}
\|\phi\|_{\W^{1, p_1'}(Q_2)} \leq N \|G\|_{L^{p_1'}(Q_2)}.
\end{equation}
Indeed, as $\bar{\bA}$ is only a function depending only in time variable, by using the different quotient if needed to the equation \eqref{phi-duality}, we can conclude that $\phi \in L^{p_1'}((-4, 0), W^{k, p_1'}_{\textup{loc}} (B_2))$ for all $k \in \mathbb{N}$. As such, $\phi$ is a strong solution to \eqref{phi-duality} and $\phi_t \in L^{p_1'}_{\textup{loc}}(Q_2)$.  

\smallskip
Next, let $q=p_0'$ where $1/p_0+1/p_0'=1$. It follows from \eqref{p-1-def-0710} that $1/q\geq 1/p_1'-1/d_0$. Then, by the H\"{o}lder inequality, the parabolic Sobolev embedding (see \cite[Lemma 3.1]{DP} and for example), and \eqref{phi-est-03016}, we infer that
\begin{equation} \label{q-p-zero-relation-0710}
\|\phi\|_{L^{p_0'}(Q_2)} = \|\phi\|_{L^q(Q_2)} \leq N \|G\|_{L^{p_1'}(Q_2)}.
\end{equation}

\smallskip
Now, for each $\eta \in C_0^\infty(Q_2)$, we use $v \eta$ as a test function to the equation \eqref{phi-duality} to obtain
\begin{align*}
& \int_{Q_2} \eta(x,t) G(x,t) \cdot \nabla v(x,t) dxdt  = - \int_{Q_2} v(x,t) G(x,t) \cdot \nabla \eta(x,t) dxdt \\
& \qquad + \int_{Q_2} \Big[v \phi_t(x,t) \eta(x,t) - \wei{\bar{\bA}(t) \nabla v, \nabla \phi} \eta   -  v\wei{\bar{\bA}(t) \nabla \eta, \nabla \phi}  \Big] dxdt \\
& = - \int_{Q_2} \Big[ v G \cdot \nabla \eta  + v\wei{\bar{\bA}(t) \nabla \eta, \nabla \phi}  + v  \phi \partial_t \eta  - \wei{\bar{\bA}(t) \nabla v, \nabla \eta} \phi \Big]dxdt,
\end{align*}
where in the last step we used the weak formulation for the equation of $v$ with the test function $\phi \eta$. We note that all terms in the integrations are well-defined. In particular, it follows from H\"{o}lder inequality, \eqref{v-p_1-est-03-16}, and \eqref{phi-est-03016} that
\begin{align*}
& \left| \int_{Q_2} \eta(x,t) G(x,t) \cdot \nabla v(x,t) dxdt  \right|   \leq  N \int_{Q_2}\Big[ \Big( |v| |G| + |\nabla v| |\phi| \big) + |v|\big( |\phi| + |\nabla \phi| \big) \Big] dxdt\\
& \leq N \Big[\big(\|G\|_{L^{p_1'}(Q_2)} \|v\|_{L^{p_1}(Q_2)} +  \|\nabla v\|_{L^{p_0}(Q_2)} \|\phi\|_{L^{p_0'}(Q_2)} \big) \\
& \qquad \qquad+  \|v\|_{L^{p_1}(Q_2)} \Big(\|\phi\|_{L^{p_1'}(Q_2)} + \|\nabla \phi\|_{L^{p_1'}(Q_2)} \Big) \Big] \\
& \leq N\|G\|_{L^{p_1'}(Q_2)} \Big[ \|v\|_{L^{1}(Q_2)} + \|\nabla v\|_{L^{p_0}(Q_2)}\Big],\end{align*}
where in the last step we used \eqref{v-p_1-est-03-16} to control $\|v\|_{L^{p_1}(Q_2)}$,  \eqref{q-p-zero-relation-0710} to control $\|\phi\|_{L^{p_0'}(Q_2)}$, and the first assertion in \eqref{phi-est-03016}  to control $\|\phi\|_{L^{p_1'}(Q_2)} + \|\nabla \phi\|_{L^{p_1'}(Q_2)}$. Hence, for every $r \in (0,2)$, by taking $\eta =1$ on $Q_{r}$ and $\eta \in C_0^\infty(Q_2)$, we can find $N = N(n, \nu,p_0, M_0, r)>0$ such that
\[
\|\nabla v\|_{L^{p_1}(Q_{r})} \leq N \Big[\|v\|_{L^{1}(Q_2)} + \|\nabla v\|_{L^{p_0}(Q_2)} \Big].
\]
The assertion \eqref{L-r-est-03016} is proved. In particular, after the iteration as we discussed, we have
\[
\|\nabla v\|_{L^{2}(Q_{r})} \leq N \Big[\|v\|_{L^{1}(Q_2)} + \|\nabla v\|_{L^{p_0}(Q_2)} \Big]
\]
for all $r \in (0,2)$ and $N = N(n, \nu, p_0, M_0, r)>0$. The assertion $v \in \W^{1,2}_{\text{loc}}(Q_2)$ in the lemma is proved.

\smallskip
It remains to prove \eqref{Lip-inter-est-1013}. Note that as $v \in \W^{1,2}_{\text{loc}}(Q_2)$, it follows from the standard regularity estimate that $\nabla v \in L^\infty_{\textup{loc}}(Q_2)$. By Lemma \ref{property}-(iii) and \eqref{heights}, we can find $R=R(n,M_0)\in(1,2)$ such that $Q_{1,\beta}\subset Q_R\subset Q_2$. For $R\leq s<r<2$, we have
\begin{align*}
\|\nabla v\|_{L^\infty(Q_s)} &\leq \frac{N}{(r-s)^{(n+2)/2}}\|\nabla v\|_{L^2(Q_r)}\\
&\leq \frac12\|\nabla v\|_{L^\infty(Q_r)}+\frac{N}{(r-s)^{(n+2)/p_0}}\|\nabla v\|_{L^{p_0}(Q_2)}.
\end{align*}
Then, it follows from the standard iteration (see \cite[p. 75]{Lin}, for example) that
\[
\|\nabla v\|_{L^\infty(Q_{1, \beta})} \leq N \left(\fint_{Q_{2, \beta}} |\nabla v(x,t)|^{p_0} dx dt \right)^{\frac{1}{p_0}}.
\]
The proof of the lemma is then completed.
\end{proof}
Next, for the reader's convenience, we refer to the definitions of $T_{2r}(x_0)$ in \eqref{flat-bdr} and  $\Gamma_{\beta, z_0}(2r)$ in \eqref{cylinder-def}. We also have the following boundary Lipschitz regularity estimates for the class of equations with frozen coefficients. 
\begin{lemma} \label{Bdr-Lipschitz-constant} For every $\nu \in (0,1), p_0 \in (1, \infty)$, and $M_0 \geq 1$, there is a constant $N = N(n, \nu, p_0, M_0)>0$ such that the following assertion holds. Let  $\beta(x)$ be a weight satisfying \eqref{beta-cond}, and $\bar{\bA}(t)$ be a symmetric matrix satisfying \eqref{ellip-constant} on $\Gamma_{\beta}(2r)$ with some $r>0$. Suppose that $v \in \hW^{1,p_0}(Q_{2r, \beta}^+)$ is a weak solution of
\begin{equation*}
\left\{
\begin{aligned}
v_t - \textup{div}(\bar{\bA}(t)\nabla v) & =0  \quad &&\text{in} \quad Q_{2r, \beta}^+,\\[4pt]
v & = 0\quad &&\text{on} \quad T_{2r} \times \Gamma_{\beta}(2r)
\end{aligned} \right.
\end{equation*}
Then $v \in \W^{1,2}_{\textup{loc}}(Q_{2r,\beta}^+)$. Moreover, $\nabla v \in L^\infty_{\textup{loc}}(Q_{2r, \beta}^+)$ and
\[
\|\nabla v\|_{L^\infty(Q_{r, \beta}^+)} \leq  N\left(\fint_{Q_{2r, \beta}^+} |\nabla v(x,t)|^{p_0}dxdt  \right)^{1/p_0}.
\]
\end{lemma}
\begin{proof} By applying the dilation as in the proof of Lemma \ref{Lipschitz-constant}, we can assume that $r=1$ and $\E_\beta(2)=1$. Then, it follows that
\[
\nu |\xi|^2 \leq \wei{\bar{\bA}(t)\xi, \xi} \quad \text{and}  \quad |\bar{\bA}(t)| \leq \nu^{-1} M_0, \quad \forall \ t \in (-4, 0), \quad \forall \xi \in \mathbb{R}^n.
\]
By Lemma \ref{property}-(iii) and \eqref{heights}, we can find $R=R(n,M_0)\in(1,2)$ such that $Q^+_{1,\beta}\subset Q_R^+\subset Q_2^+$. Applying the local boundary regularity estimates in \cite[Remark 9.1, Theorems 2.5 and 2.6]{Dong-Kim-2011} and the Poincar\'{e} inequality for $v$ (Lemma \ref{imbed-wei-0309}-(ii)), we obtain the assertions of the lemma. The proof is completed.
\end{proof}
\section{Local interior estimates and proof of Theorem \ref{inter-theorem}}\label{interior section}
This section provides key estimates, and then proves Theorem \ref{inter-theorem}. For readers convenience, let us recall that for $r>0$, we write
\begin{equation} \label{Theta-beta-def}
\Theta_{\bA, r} (x,t) = |\bA(x,t) - (\bA)_{B_{r}}(t)|\beta(x)^{-1} \quad \text{where} \quad (\bA)_{B_{r}}(t) = \frac{1}{|B_r|} \int_{B_r} \bA(x,t) dx.
\end{equation}
\subsection{Interior Caccioppoli Estimates and Poincar\'{e} Estimates} \label{inter-energy} We study the problem
\begin{equation} \label{eqn-Q_R}
u_t - \textup{div}(\bA(x,t) \nabla u)  = \textup{div}(\beta(x) F(x,t))  \quad \text{in} \quad Q_{4,\beta}.
\end{equation}
Recall that for a given weak solution $u \in \W^{1,2}(Q_{4, \beta}, \beta)$, and for $\rho \in (0,4)$, the average of $u$ in the ball $B_\rho$ with respect to the weight $\eta_{0, \rho}^2$ defined in \eqref{average-eta-weight} as following
\[
[u]_{0, \rho}(t) =\fint_{B_\rho} u(x, t) \eta_{0, \rho}(x)^2 dx.
\]
We begin with the following lemma on Caccioppoli type estimates for \eqref{eqn-Q_R} in the spirit of \cite{GS}.
\begin{lemma} \label{Caccio-1} For every $\nu \in (0,1)$ and  $M_0 \geq 1$, there exists a constant $N = N(n, \nu)>0$ such that the following holds. Suppose that $u \in \W^{1,2}_{\textup{loc}}(Q_{4,\beta}, \beta)$ is a weak solution of \eqref{eqn-Q_R} in which $\bA$ satisfies \eqref{ellip-cond} in $Q_{4,\beta}$ and its corresponding $\beta$ satisfies \eqref{beta-cond}, and $F \in L^2_{\textup{loc}}(Q_{4,\beta}, \beta)^n$. Then for  $\hat{u}(x,t) = u(x,t) - [u]_{0, \rho}(t)$ with some $\rho \in (0,4)$, it holds that
\begin{align*}
& \int_{B_4} \hat{u}(x,t_1)^2 \eta_{0, \rho}(x)^2 dx + \nu \int_{Q_{4, \beta}} |\nabla \hat{u}(x,t)|^2 \eta_{0, \rho}(x)^2\chi(t)^2 \beta(x) dx dt\\
&\leq N \int_{Q_{4,\beta}} \hat{u}(x,t)^2\Big [|\nabla \eta_{0, \rho}(x)|^2 \beta(x) + |  \partial_t \chi(t)| \Big] \chi(t)dx dt \\
& \qquad +N  \int_{Q_{4,\beta}} |F(x,t)|^2 \eta_{0, \rho}(x)^2\chi(t)^2  \beta(x) dx dt,
\end{align*}
for a.e. $t_1 \in \Gamma_{\beta}(4) = (-16 \E_{\beta}(4), 0)$, and for every non-negative cut-off function $\chi \in C^\infty(\mathbb{R})$ satisfying $\chi =0$ near $-16 \E_{\beta}(4)$, $\chi(t_1) =1$, and $0 \leq \chi \leq 1$.
\end{lemma}
\begin{proof} From Lemma \ref{scaling-lemma}, we can assume without loss of generality that $\E_{\beta}(4) =1$.  Note that it follows from the definition of $\W^{1,2}(Q_{4}, \beta)$ and Lemma \ref{imbed-wei-0309} that $u(x,t) - [u]_{0, \rho}(t) \in L^2(Q_{4,\beta})$. Therefore, the right hand side in the assertion estimate of the lemma is well-defined.  Let us denote $\varphi(x,t) = \chi(t)\eta_{0, \rho}(x)$. We note from definition of $[u]_{0, \rho}(t)$ in \eqref{average-eta-weight} that
\begin{equation} \label{cancel-eta-average}
\int_{B_4} \hat{u}(x,t)  \varphi(x,t)^2 dx = \chi(t)^2\int_{B_4} \hat{u}(x,t)   \eta_{0, \rho}(x)^2 dx =0.
\end{equation}
Then, it follows that
\begin{align} \notag
& \frac{1}{2}\int_{B_4} |u(x,t)|^2\eta_{0, \rho}(x)^2 dx =  \frac{1}{2} \int_{B_4} |\hat{u}(x,t) + [u]_{0, \rho}(t)|^2\eta_{0, \rho}(x)^2 dx\\ \label{average-indentity-eta}
&=  \frac{1}{2} \int_{B_4} |\hat{u}(x,t)|^2 \eta_{0, \rho}(x)^2 dx 
+  \frac{1}{2}[u]_{0, \rho}(t)^2\int_{B_4} \eta_{0, \rho}(x)^2 dx,
\end{align}
for all a.e. $t \in (-16, 0).$ On the other hand, as in \cite{GS}, it follows from the weak form of the PDE of $u$ and the definition of $[u]_{0, \rho}(t)$ that
\begin{align} \notag
 \partial_t [u]_{0, \rho}(t)   = -\frac{1}{|B_\rho|}\int_{B_\rho} &\Big[ \wei{\bA(x,t) \nabla u(x,t), \nabla [\eta_{0, \rho}(x)^2]}  \\ \label{eta-diff-0311}
& \qquad + \wei{F(x,t), \nabla [\eta_{0, \rho}(x)^2]} \beta(x) \Big] dx.
\end{align}
Then, by using \eqref{cancel-eta-average}, we infer that
\begin{align*}
& \int_{-16}^{t_1} \int_{B_4} u(x,t)^2 \eta_{0,\rho}(x)^2 \partial_t[\chi(t)^2] dxdt \\
& = \int_{-16}^{t_1} \int_{B_4} [\hat{u}(x,t) + [u]_{0,\rho}(t)]^2 \eta_{0,\rho}(x)^2 \partial_t[\chi(t)^2] dxdt \\
& = \int_{-16}^{t_1} \int_{B_4} \hat{u}(x,t)^2 \eta_{0,\rho}(x)^2 \partial_t[\chi(t)^2] dxdt + \int_{-16}^{t_1} \int_{B_4} [u]_{0, \rho}(t)^2 \eta_{0,\rho}(x)^2 \partial_t[\chi(t)^2] dxdt. 
\end{align*}
From this, and by applying the integration by parts with respect to the time integration on the last term on the right hand side and the assumption that $\chi(-16) =0$, we obtain
\begin{align} \notag
&\frac{1}{2}\int_{-16}^{t_1} \int_{B_4} u(x,t)^2 \eta_{0,\rho}(x)^2 \partial_t[\chi(t)^2] dxdt = \frac{1}{2}\int_{-16}^{t_1} \int_{B_4} \hat{u}(x,t)^2 \eta_{0,\rho}(x)^2 \partial_t[\chi(t)^2] dxdt  \\ \notag
& \ + \frac{1}{2} [u]_{0, \rho}(t_1)^2 \chi(t_1)^2 \int_{B_4} \eta_{0, \rho}(x)^2 dx  -\int_{-16}^{t_1} \chi(t)^2 [u]_{0, \rho}(t) \partial_t [u]_{0, \rho}(t) \int_{B_4} \eta_{0, \rho}(x)^2 dx  \\  \notag 
& = \frac{1}{2}\int_{-16}^{t_1} \int_{B_4} \hat{u}(x,t)^2 \eta_{0,\rho}(x)^2 \partial_t[\chi(t)^2] dxdt  + \frac{1}{2} [u]_{0, \rho}(t_1)^2 \chi(t_1)^2 \int_{B_4} \eta_{0, \rho}(x)^2 dx \\  \notag 
& \qquad + \int_{-16}^{t_1}\int_{B_4} \Big[ \wei{\bA(x,t) \nabla u(x,t), \nabla [\eta_{0, \rho}(x)^2]}  +  \\ \label{int-eta-time-0311}
& \qquad  \qquad \qquad \qquad  + \wei{F(x,t), \nabla [\eta_{0, \rho}(x)^2]} \beta(x) \Big] \chi(t)^2 [u]_{0, \rho}(t)  dxdt,
\end{align}
where we used \eqref{eta-diff-0311} in our last step.

\smallskip
Now, with $\varphi(x,t) = \chi(t)\eta_{0, \rho}(x)$, and by using Steklov's average if needed (see \cite[p. 18]{DiB} for example), we can formally test the equation \eqref{eqn-Q_R} with $u\varphi^2$ to obtain
\begin{align*}
& \frac{1}{2}\int_{B_4} u(x,t_1)^2 \varphi(x,t_1)^2 dx + \int_{-16}^{t_1}\int_{B_{4}} \wei{\bA(x,t) \nabla u(x,t), \nabla u(x,t)} \varphi(x,t)^2 dx dt\\
& = -\int_{-16}^{t_1}\chi(t)^2\int_{B_4}\Big[ \wei{\bA \nabla u, \nabla [\eta_{0, \rho}(x)^2]}  +  \wei{F, \nabla [\eta_{0, \rho}(x)^2]} \beta(x)  \Big] u(x,t)dx dt \\
& \qquad + \frac{1}{2}\int_{-16}^{t_1}\partial_t[ \chi(t)^2]\int_{B_4} u(x,t)^2 \eta_{0, \rho}(x)^2  dx dt - \int_{-16}^{t_1}\int_{B_4}  \varphi^2 \wei{F, \nabla u} \beta(x) dxdt.
\end{align*}
From this last estimate, we use \eqref{average-indentity-eta} and \eqref{int-eta-time-0311} to treat the first term on the left hand side and the third term on the right hand side of the last estimate. By doing this, and then rearranging terms, we get
\begin{align*}
& \frac{1}{2}\int_{B_4} \hat{u}(x,t_1)^2 \varphi(x,t_1)^2 dx + \int_{-16}^{t_1}\int_{B_4} \wei{\bA(x,t) \nabla \hat{u}(x,t), \nabla \hat{u}(x,t)} \varphi(x,t)^2 dx dt\\
& = -\int_{-16}^{t_1}\int_{B_4}\Big[ \wei{\bA \nabla u, \nabla [\eta_{0, \rho}(x)^2]}  +  \wei{F, \nabla [\eta_{0, \rho}(x)^2]} \beta(x)  \Big] \hat{u}(x,t)\chi(t)^2 dx dt \\
& \qquad + \frac{1}{2}\int_{-16}^{t_1} \int_{B_4} \hat{u}(x,t)^2 \eta_{0,\rho}(x)^2 \partial_t[\chi(t)^2] dxdt  -   \int_{-16}^{t_1}\int_{B_4}  \varphi^2 \wei{F, \nabla \hat{u}} \beta(x) dxdt.
\end{align*}
From this, the conditions on $\bA$ in \eqref{ellip-cond}, we can follow the standard energy estimates, using H\"{o}lder's inequality and Young's inequality to derive the estimate
\begin{align*}
& \int_{B_4} \hat{u}(x,t_1)^2 \eta_{0, \rho}(x)^2 \chi(t_1)^2 dx + \nu \int_{-16}^{t_1}\int_{B_4} |\nabla \hat{u}(x,t)|^2 \eta_{0, \rho}(x)^2\chi(t)^2 \beta(x) dx dt\\
&\leq N \int_{Q_{4,\beta}} \hat{u}(x,t)^2\Big [|\nabla \eta_{0, \rho}(x)|^2 \beta(x) + |  \partial_t \chi(t)| \Big] \chi(t)dx dt \\
& \qquad +N  \int_{Q_{4,\beta}} |F(x,t)|^2 \eta_{0, \rho}(x)^2\chi(t)^2  \beta(x) dx dt.
\end{align*}
for $N = N(n, \nu)>0$ and for a.e. $t_1 \in (-16, 0)$.  Using $\chi(t_1)=1$, we bound the first term on the left hand side. Dropping the first term and letting the upper limit of the time integration tend to $0^-$, we bound the second term in the assertion of the lemma. Combining these estimates, we obtain the assertion of the lemma. The lemma is then proved.
\end{proof}

\smallskip
From now on, for a given number $M_0 \geq 1$, let $\gamma = \gamma(n, M_0)$ be the sufficently small number defined in Lemma \ref{L-q-2-wei} in which $q = 1+\frac{1}{n_0}$ for $n_0 = \max\{n-1, 1\}$. Then, let us define
\begin{equation} \label{p-zero-def}
p_0 = \min \Big\{ \frac{2(1+\gamma)}{2+ \gamma}, \frac{2}{1+\frac{1}{n_0} -\gamma} \Big\} \in (1, 2).
\end{equation}
Our next lemma is on the Poincar\'{e} type inequality for the weak solution of the equation \eqref{eqn-Q_R}.
\begin{lemma} \label{L-2-inter-u-est} For every $\nu \in (0,1)$ and $M_0 \geq 1$, there is $N = N(n, \nu, M_0)>0$ such that the following assertion holds. Suppose $u \in \W^{1,2}(Q_{4, \beta}, \beta)$ is a weak solution of \eqref{eqn-Q_R} in which the matrix $\bA$ satisfies \eqref{ellip-cond} in $Q_{4,\beta}$ and its corresponding $\beta$ satisfies \eqref{beta-cond}, and $F \in L^2(Q_{4,\beta}, \beta)^n$. Then
\begin{equation} \label{L-2-inter-u-est-01-07}
\begin{split}
& \Big( \fint_{Q_{4, \beta}} | u (x,t) - (u)_{Q_{4,\beta}}|^{p_0} dx dt \Big)^{\frac{1}{p_0}}  \\
 & \leq N \Big(\fint_{Q_{4, \beta}}|\nabla u(x,t)|^2\beta(dx dt) \Big)^{\frac{1}{2}} + N  \Big(\fint_{Q_{4, \beta}} |F(x,t)|^2 \beta(dx dt) \Big)^{\frac{1}{2}}.
 \end{split}
\end{equation}
\end{lemma}
\begin{proof} Note that we only need to prove the assertion \eqref{L-2-inter-u-est-01-07} for the class of non-constant solutions. Also, by using Lemma \ref{scaling-lemma}, it is sufficient to consider the case that
\[
\fint_{B_4} \beta(x)^{-n_0} dx =1 \quad \text{where} \quad n_0 =\max\{1, n-1\}.
\]
In this case, we have $Q_{4, \beta} = Q_4$, and
\begin{equation*} 
1 \leq \Big(\fint_{B_4} \beta(x) dx \Big) \Big(\fint_{B_4} \beta(x)^{-n_0}dx \Big)^{\frac{1}{n_0}} =  \fint_{B_4} \beta (x) dx \leq M_0.
\end{equation*}
We now prove \eqref{L-2-inter-u-est-01-07} by a contradiction argument. Assume that the assertion \eqref{L-2-inter-u-est-01-07} is not true. Then, for every $k \in \mathbb{N}$, there are $\bA_k$ satisfying \eqref{ellip-cond} in $Q_{4}$ with its corresponding $\beta_k$ satisfying 
\begin{equation} \label{beta-k-all-contradiction-0702}
[\beta_k]_{A_{1+\frac{1}{n_0}}} \leq M_0, \quad \fint_{B_4} \beta_k(x)^{-n_0} dx =1, \quad \text{and} \quad  1 \leq \fint_{B_4} \beta_k (x) dx \leq M_0.
\end{equation}
In addition, there are a sequence of vector fields $\{F_k\} \subset  L^2(Q_4, \beta_k)^n$ and a sequence of weak non-constant solutions $\{u_k\} \subset \W^{1,2}(Q_4, \beta_k)$ such that for each $k \in \mathbb{N}$, $u_k$ weakly solves
\begin{equation} \label{u-k-weak-07-02}
\partial_t u_k - \text{div}[\bA_k(x,t) \nabla u_k] = \text{div}[\beta_k(x) F_k(x,t)] \quad \text{in} \quad Q_4
\end{equation}
and it satisfies
\[
 \|u_k - (u_k)_{Q_4}\|_{L^{p_0}(Q_4)} \geq k \Big[\|\nabla u_k\|_{L^2(Q_4, \beta_k)} + \|F_k\|_{L^2(Q_4, \beta_k)}  \Big].
\]
By dividing the PDE of $u_k$ by the non-zero number $\|u_k - (u_k)_{Q_4}\|_{L^{p_0}(Q_{4})}$, we can assume without loss of generality that 
\begin{equation} \label{u-k-L2-est-0701}
\|u_k - (u_k)_{Q_4}\|_{L^{p_0}(Q_4)} =1, \quad \forall \ k \in \mathbb{N}.
\end{equation}
It then follows that
\begin{equation} \label{grad-u-k-F-k-est-0702}
\|\nabla u_k\|_{L^2(Q_4, \beta_k)} + \|F_k\|_{L^2(Q_4, \beta_k)} \leq \frac{1}{k}, \quad \forall \ k \ \in \mathbb{N}.
\end{equation}
Now, we apply Lemma \ref{L-q-2-wei}-(i) and  \eqref{beta-k-all-contradiction-0702} to find a sufficiently small constant $\gamma = \gamma(n, M_0) \in (0,1)$ such that
\begin{equation} \label{grad-u-k-F-k-est-07-01}
\|\nabla u_k\|_{L^2((-16, 0), L^{\bar{p}_0}(B_4))} + \|F_k\|_{L^2((-16, 0), L^{\bar{p}_0}(B_4))} \leq \frac{N}{k}, \quad \forall \ k \in \mathbb{N},
\end{equation}
where $\bar{p}_0 = \frac{2}{1+\frac{1}{n_0} -\gamma} \in (1,2)$. In addition, for $q_0 = \frac{2(1+\gamma)}{2+\gamma} \in (1,2)$, it follows from Lemma \ref{L-q-2-wei}-(ii) and \eqref{beta-k-all-contradiction-0702} that there is $N = N(n, M_0) >0$ so that
\[
\|\nabla u_k \beta_k\|_{L^2((-16, 0), L^{q_0}(B_4))} + \|F_k \beta_k\|_{L^2((-16, 0), L^{q_0}(B_4))}
\leq \frac{N}{k}, \quad \forall \ k \in \mathbb{N}.
\]
From this estimate, the PDE of $u_k$, \eqref{u-k-L2-est-0701}, \eqref{grad-u-k-F-k-est-07-01}, the Poincar\'{e} inequality in Lemma \ref{imbed-wei-0309}-(i), and the equation for $[u_k]_{0,1}(t)$ as in \eqref{eta-diff-0311}, we infer that there is $N = N(n, \nu, M_0)>0$ such that
\[
\|\bar{u}_k\|_{L^{2}((-16, 0), W^{1, \bar{p}_0}(B_4))}  \leq N, \quad \|\partial_t \bar{u}_k\|_{L^{2}((-16,0), W^{-1, q_0}(B_4))} \leq N, \quad \forall \ k \in \mathbb{N}
\]
for $\bar{u}_k = u_k - (u_k)_{Q_{4}}$. Then, by applying the Aubin-Lions theorem, see \cite{Simon}, and passing through a subsequence if needed, we can find $u \in \W^{1, p_0}(Q_{4})$ so that
\begin{equation} \label{uk-convergence-0702}
\begin{array}{lll}
 \bar{u}_k \rightarrow u &\quad \text{strongly in} & \quad L^{p_0}((-16, 0), L^{p_0}(B_4)),\\
\nabla \bar{u}_k \rightharpoonup \nabla u  & \quad \text{weakly in} & \quad  L^{p_0}(Q_4),\\
\partial_t \bar{u}_k \rightharpoonup \partial_t u  & \quad \text{weakly in} & \quad  L^{p_0}((-16, 0), W^{-1,p_0}(B_4)),
\end{array}
\end{equation}
as $k \rightarrow \infty$, where $p_0 = \min\{\bar{p}_0, q_0\} \in (1, 2)$ as in \eqref{p-zero-def}. 

\smallskip
Next, for each $\varphi \in C^\infty_0(Q_4)$, as $\bar{u}_k$ is also a weak solution of the same equation \eqref{u-k-weak-07-02} of $u_k$, we have
\[
-\int_{Q_4} \bar{u}_k (x,t)\partial_t \varphi(x,t) dx dt + \int_{Q_4} \wei{\bA_k(x,t) \nabla u_k, \nabla \varphi} dx dt = - \int_{Q_4} \wei{F_k, \nabla \varphi} \beta_k(x) dx dt, 
\]
Then, by using \eqref{grad-u-k-F-k-est-0702}, the convergences in \eqref{uk-convergence-0702}, and by passing through the limit as $k \rightarrow \infty$ of the last identity, we see that $u$ is a weak solution of the equation
\[
\partial_t u =0 \quad \text{in} \quad Q_4.
\]
On the other hand, it follows from \eqref{grad-u-k-F-k-est-07-01}, the definition of $\bar{u}_k$, and the first convergence in \eqref{uk-convergence-0702} that
\[
\nabla u =0 \quad \text{in} \quad Q_4, \quad \text{and} \quad (u)_{Q_4} =0.
\]
Therefore, $u \equiv 0$ in $Q_4$. This contradicts to the fact that 
\[ 1 = \lim_{k \rightarrow \infty}\|\bar{u}_k\|_{L^{p_0}(Q_4)} = \|u\|_{L^{p_0}(Q_4)}.
\]
The proof of the lemma is then completed.
\end{proof}
\subsection{Interior Approximation Estimates and Proof of Theorem \ref{inter-theorem}} \label{inter-appr-est}  We use the freezing-coefficients technique to locally approximate solutions of \eqref{eqn-Q_R} by solutions of the frozen coefficient equation:
\begin{equation} \label{v-Qr-sol}
v_t - \textup{div} ((\bA)_{B_{4}}(t) \nabla v)  = 0 \quad \text{in} \quad Q_{4, \beta}.
\end{equation}
Due to Lemma \ref{Lipschitz-constant}, we see that each weak solution $v \in \W^{1,p_0}(Q_{4,\beta})$ of \eqref{v-Qr-sol} is sufficiently smooth. Our goal is to show that $\nabla u$ is sufficiently close to $\nabla v$ in $L^2(Q_{2,\beta}, \beta)$ when $\bA(x,t)$ is sufficiently close to $(\bA)_{B_4}(t)$. 

\smallskip
We  begin with the following important lemma on the Caccioppoli type estimates for the difference of two solutions of \eqref{eqn-Q_R} and \eqref{v-Qr-sol}.
\begin{lemma} \label{compare-lemma-1} Let $\nu \in (0,1)$, $M_0 \geq 1$ and $F \in L^2_{\textup{loc}}(Q_{4,\beta},\beta)^n$. Suppose that the matrix $\bA$ satisfies \eqref{ellip-cond} in $Q_{4, \beta}$ and its corresponding $\beta$ satisfies \eqref{beta-cond}. Suppose also that $u \in \W^{1,2}(Q_{4,\beta}, \beta)$ is a weak solution of \eqref{eqn-Q_R},  and $v \in \W^{1,p_0}_{\textup{loc}}(Q_{4,\beta})$ is a weak solution of \eqref{v-Qr-sol} with some $p_0 \in (1, 2)$.  Then, for  $w = u -v$, and $\hat{w}(x,t) = w(x,t) - [w]_{0, \rho}(t)$ with some $\rho \in (0,4]$, it holds that
\begin{align*}
&\int_{Q_{4, \beta}} |\nabla w|^2 \varphi(x,t)^2 \beta(x)\, dx dt \leq   N\int_{Q_{4,\beta}}  \hat{w}^2 \big( |\varphi \varphi_t| +|\nabla \varphi|^2 \beta(x) \big) dxdt\\
& + N \int_{Q_{4, \beta}} |F(x,t)|^2 \beta(x)\varphi(x,t)^2 dx dt +   N \|\varphi \nabla v\|_{L^\infty(Q_{4, \beta})}^2 \int_{Q_{4, \beta}} \Theta_{\bA, 4}(x,t)^2  \beta(x)dx dt,
\end{align*}
for every $\varphi(x,t) = \eta_{0, \rho}(x) \chi(t)$ with $\chi \in C^\infty(\Gamma_{\beta}(4))$ is the standard non-negative cut-off function vanishing near $- 16 \E_{\beta}(4)$, where $N = N(n, \nu)>0$.
\end{lemma}
\begin{proof} From Lemma \ref{Lipschitz-constant}, it follows that $v,\nabla v \in L^\infty_{\text{loc}}(Q_{4,\beta})$ and therefore $v,\nabla v \in L^2_{\text{loc}}(Q_{4,\beta}, \beta)$. Consequently, if $u \in \W^{1,2}(Q_{4, \beta}, \beta)$ is a weak solution to \eqref{eqn-Q_R} in $Q_{4, \beta}$, it follows from Lemma \ref{L-q-2-wei}-(i) that $w = u - v \in \W^{1,2}_{\text{loc}}(Q_{4, \beta}, \beta)$ is a weak solution of
\begin{align*}
w_t  - \textup{div}( \bA \nabla w) & = \textup{div}(\beta(x) G(x,t)) \quad \text{in} \quad Q_{4,\beta}.
\end{align*}
where
\[
G(x,t) = F(x,t) + \beta(x)^{-1}(\bA(x,t) - (\bA)_{B_{4}}(t))\nabla v(x,t) \in L^2_{\text{loc}}(Q_{4,\beta}, \beta).
\]
Note that from the triangle inequality, we have
\begin{align*}
& \int_{Q_{4,\beta}} |G(x,t)|^2 \varphi(x,t)^2  \beta(x) dx dt  \leq 2\int_{Q_{4,\beta}} |F(x,t)|^2 \varphi(x,t)^2  \beta(x) dx dt \\
&\qquad \qquad + 2\|\varphi \nabla v \|_{L^\infty(Q_{4, \beta})}^2 \int_{Q_{4,\beta}} |\Theta_{\bA, 4}(x,t)|^2 \beta(x) dx dt.
\end{align*}
Then, the assertion of the lemma follows  directly from Lemma \ref{Caccio-1}.
\end{proof}

\smallskip
The following lemma provides an important result on the closeness in $L^2(Q_{4,\beta}, \beta)$ of weak solutions of \eqref{eqn-Q_R} and \eqref{v-Qr-sol}. The proof of the lemma is the most technical one in the session.
\begin{lemma} \label{L2-comparision} For every $\nu \in (0,1)$, $M_0 \geq 1$, and $\bar{\epsilon} \in (0, 1)$, there exists a sufficiently small constant $\bar{\delta} =\bar{\delta}(n, \nu, M_0, \bar{\epsilon})>0$ such that the following assertions hold. Suppose that $\bA$ satisfies \eqref{ellip-cond} in $Q_{4, \beta}$ and its  corresponding $\beta$ satisfies \eqref{beta-cond}. Suppose also that
\begin{align*}
& \fint_{Q_{4,\beta}} \big[ \Theta_{\bA, 4}(x,t)^2  + |F(x,t)|^2\big]\beta(dxdt)  \leq \bar{\delta}^2.
\end{align*}
Then, for every weak solution $u \in \W^{1,2}(Q_{4,\beta},\beta)$ of \eqref{eqn-Q_R} satisfying
\[
 \fint_{Q_{4,\beta}} |\nabla u(x,t)|^2 \beta(dxdt) \leq 1,
\]
there exists a weak solution $v \in \W^{1,2}_{\textup{loc}}(Q_{4,\beta})$ of \eqref{v-Qr-sol} in $Q_{4,\beta}$ such that
\begin{equation} \label{L-2-w-small}
\left(\fint_{Q_{7/2, \beta}} |w(x,t) - [w]_{0,5/2}(t)|^2 \beta(dx dt)\right)^{1/2}  \leq \bar{\epsilon},
\end{equation}
for $w(x,t) = u(x,t) - (u)_{Q_4, \beta}- v(x,t)$. Moreover, 
\begin{equation} \label{L-2-mu}
\fint_{Q_{3,\beta}} |\nabla v(x,t)|^2 dx dt \leq N,
\end{equation}
with some constant $N = N(n, \nu, M_0)>0$.
\end{lemma}
\begin{proof}  By applying the time dilation as in \eqref{scale-2}, and replacing $u$, $\beta$, $\bA$, and $F$ with their corresponding dilated ones, it follows from Lemma \ref{scaling-lemma} that we  can assume without loss of generality that
\begin{equation} \label{scaling-beta-k}
\E_{\beta}(4) =\left( \fint_{B_4} \beta(x)^{-n_0} dx\right)^{\frac{1}{n_0}} =1.
\end{equation}
In this case $Q_{4,\beta} =Q_4$. It is worth noting that this procedure is admissible as the conditions in the lemma are invariant under the time dilation \eqref{scale-2}. 

\smallskip
We prove the assertion \eqref{L-2-w-small} via a contradiction argument. Assume that the assertion is not true, then there exist $\epsilon_0>0$, a sequence of coefficient matrices $\{\bA_k\}_k$ satisfying \eqref{ellip-cond} in $Q_{4}$ with its corresponding  sequence of weights $\{\beta_k\}_k$ satisfying \eqref{beta-cond} and \eqref{scaling-beta-k} for each $k\in \N$, i.e., $\beta_k \in A_{1+\frac{1}{n_0}}$ and for $n_0 = \max\{n-1, 1\}$
\begin{equation} \label{nabla-u-k}
[\beta_k]_{A_{1+\frac{1}{n_0}}} \leq M_0 \quad \text{and} \quad \left(\fint_{B_4} \beta_k(x)^{-n_0} dx \right)^{\frac{1}{n_0}} =1.
\end{equation} 
Moreover, there is a sequence $\{F_k\}_k \subset L^2(Q_4, \beta_k)^n$ such that
\begin{align} \label{a-mu-k}
\fint_{Q_{4}}\Big[ \Theta_{\bA_k, 4}(x,t)^2 + |F_k(x,t)|^2 \Big]\beta_k(dx dt) \leq \frac{1}{k^2},
\end{align}
and there exists a sequence $\{u_k\}_k \subset \W^{1,2}(Q_4, \beta_k)$ satisfying
\begin{align}\label{gradient small}
\fint_{Q_{4}} |\nabla u_k(x,t)|^2\beta_k(dxdt) \leq 1, \quad \forall \ k \in \mathbb{N},
\end{align}
where each $u_k \in \W^{1,2}(Q_4, \beta_k)$ is a weak solution of 
\[
\partial_tu_k-\textup{div}(\bA_k(x,t)\nabla u_k)=\textup{div}(\beta_k(x) F_k(x,t))\quad \text{in}\quad Q_4.
\]
However, for every $k \in \mathbb{N}$, and for every weak solution $v \in \W^{1,2}_{\text{loc}}(Q_{4})$  satisfying \eqref{L-2-mu} with a fixed sufficiently large $N=N(n,\nu,M_0)$ of the equation
\[
v_t - \textup{div} [(\bA_k)_{B_{4}}(t) \nabla v] = 0 \quad \text{in} \quad Q_{4},
\]
it follows that
\begin{equation} \label{w-k-contra}
\fint_{Q_{7/2, \beta_k}} |u_k(x,t) - (u_k)_{Q_{4}}- v(x,t) - [u_k - (u_k)_{Q_{4}}- v]_{0, 5/2}(t)|^2\beta_k(dxdt) \geq \epsilon_0.
\end{equation}

\smallskip
Now, let us provide a few observation. Note that from the H\"{o}lder inequality and \eqref{nabla-u-k} that
\begin{equation} \label{beta_k-measure}
1 \leq \left(\fint_{B_4} \beta_k(x) dx \right) \left(\fint_{B_4} \beta_k(x)^{-n_0} dx\right)^{\frac{1}{n_0}} =  \fint_{B_4} \beta_k(x) dx \leq M_0, \quad \forall \ k \in \mathbb{N}.
\end{equation}
Also, as in Lemma \ref{property}-(iii), there is $\zeta = \zeta(n, M_0) \in (0,1)$ such that
\[
\bar{\beta}_k(B_{7/2}) \leq \zeta \bar{\beta}_k(B_{4}),
\]
where $\bar{\beta}_k(x) = \beta_{k}(x)^{-n_0}$. Therefore, there is $\kappa_0 = \kappa_0(n, M_0) \in (0,1)$ such that
\begin{equation} \label{beta-k-cylinder-0314}
Q_{7/2, \beta_k} \subset Q_{4\kappa_0} \subset Q_{4}, \quad \forall \ k \in \mathbb{N}.
\end{equation}

\smallskip
Next, it follows from the assumption \eqref{ellip-cond} and \eqref{beta_k-measure} that for all $t \in (-16, 0)$ and for all $\xi \in \mathbb{R}^n$,
\begin{equation} \label{bar-A-scaled-ellip}
\nu |\xi|^2 \leq \wei{(\bA_k)_{B_4}(t) \xi, \xi} \quad \text{and} \quad |(\bA_k)_{B_4}(t)| \leq \nu^{-1} M_0, \quad \forall \ k \in \mathbb{N}.
\end{equation}
Hence, there exist a subsequence that we still denote by $\{(\bA_k)_{B_4}(t)\}_k$ and a matrix $\bar{\bA}(t)\in L^{\infty}((-16,0), \M^{n\times n})$  satisfying \eqref{bar-A-scaled-ellip} such that the sequence $\{(\bA_k)_{B_4}(t)\}_k$ converges weak-$^*$ in $L^{\infty}((-16,0), \M^{n\times n})$ to $\bar{\bA}(t)$ as $k \rightarrow \infty$. Equivalently, 
\begin{equation}\label{weak star convergence}
\lim_{k\rightarrow \infty}\int_{-16}^{0} [(\bA_k)_{B_4}(t)-\bar{\bA}(t)]\cdot \bB(t) dt=0,
\end{equation}
for all $\bB(t)\in L^{1}((-16,0), \M^{n\times n})$. 

\smallskip
In addition,  it follows from Lemma \ref{L-2-inter-u-est},  \eqref{a-mu-k}, and \eqref{gradient small} that
\[
\|u_k - (u_k)_{Q_{4}}\|_{L^{p_0}(Q_4)} \leq N,
\]
where $N = N(n, \nu, M_0)>0$, and $p_0 \in (1, 2)$ is defined in \eqref{p-zero-def}. From this and as in the proof of Lemma \ref{L-2-inter-u-est} that uses Lemma \ref{L-q-2-wei}, there is $N = N(n, \nu, M_0)>0$ such that
\begin{equation} \label{u-k-bounded-p-0}
 \|u_k - (u_k)_{Q_4}\|_{\W^{1,p_0}(Q_4)} \leq N, \quad \forall \ k \in \mathbb{N}.
\end{equation}
Moreover, as in the proof of Lemma \ref{L-2-inter-u-est}, applying Lemma \ref{L-q-2-wei}, we have $\nabla u_k$ bounded in $L^2((-16,0),L^{p_0}(B_4))$ and $\partial_t u_k$ bounded in $L^2((-16,0),W^{-1,p_0}(B_4))$. Applying the Poincar\'{e} inequality (Lemma \ref{imbed-wei-0309}-(i)),  and using the equation for $[u_k]_{0,1}(t)$ as in \eqref{eta-diff-0311}, we obtain
\[
\|u_k-(u_k)_{Q_4}\|_{L^2((-16,0),W^{1,p_0}(B_4))}+\|\partial_t u_k\|_{L^2((-16,0),W^{-1,p_0}(B_4))}\leq N.
\]
Then, from Aubin-Lions theorem (see \cite{Simon}) and by passing a subsequence, we can find $u_{0} \in \W^{1,p_0}(Q_4)$ such that
\begin{equation} \label{uk-convergence}
\begin{array}{cclll}
 u_k - (u_k)_{Q_4} & \rightarrow & u_0 &\quad \text{strongly in} & \quad  L^2((-16,0),L^{p_0}(B_4)),\\
\nabla u_k & \rightharpoonup & \nabla u_0  & \quad \text{weakly in} & \quad  L^{p_0}(Q_4),\\
\partial_t u_k & \rightharpoonup & \partial_t u_0  & \quad \text{weakly in} & \quad  L^{p_0}((-16, 0), W^{-1,p_0}(B_4)),
\end{array}
\end{equation}
as $k \rightarrow \infty$.

\smallskip
Next, for a fixed $\varphi \in C^\infty_0(Q_4)$, and for every $k \in \mathbb{N}$, due to the weak formulation for the equation of $u_k$, we have
\begin{equation} \label{weak-form-u-k}
\begin{split}
& \int_{Q_4} [u_k(x,t) - (u_k)_{Q_4}] \partial_t \varphi (x,t) dxdt \\
& = \int_{Q_4} \wei{\bA_k \nabla u_k(x,t), \nabla \varphi (x,t)} dxdt + \int_{Q_4} \wei{F_k(x,t), \nabla \varphi(x,t)}  \beta_k(x) dxdt.
\end{split}
\end{equation}
Our goal now is to pass the limit as $k \rightarrow \infty$ on each term in \eqref{weak-form-u-k} to derive the equation for $u_0$. Note that from the first assertion in \eqref{uk-convergence} and from \eqref{a-mu-k}, we get
\begin{equation} \label{two-main-07-02}
\begin{split}
& \lim_{k\rightarrow \infty}\int_{Q_4} [u_k(x,t) - (u_k)_{Q_4}] \partial_t \varphi(x,t)  dxdt  = \int_{Q_4} u_0(x,t) \partial_t \varphi (x,t) dxdt  \quad \text{and} \\
&\lim_{k\rightarrow \infty} \int_{Q_4} \wei{F_k(x,t), \nabla \varphi(x,t)}  \beta_k(x) dxdt =0.
\end{split}
\end{equation}
Also, we write the first term on the right hand side of \eqref{weak-form-u-k} as
\begin{align}  \notag
& \int_{Q_4} \wei{\bA_k(x,t) \nabla u_k(x,t), \nabla \varphi(x,t)}  dxdt\\ \label{main-k-term-07-02}
& = \int_{Q_4} \wei{[\bA_k(x,t)  - (\bA_k)_{B_4}]\nabla u_k(x,t), \nabla \varphi(x,t)}  dxdt \\ \notag
&\qquad+ \int_{Q_4} \wei{[(\bA_k)_{B_4}(t) - \bar{\bA}(t)] \nabla u_k(x,t), \nabla \varphi(x,t)} dxdt\\ \notag
& \qquad + \int_{Q_4} \wei{\bar{\bA}(t) \nabla u_k(x,t), \nabla \varphi(x,t)} dxdt.
\end{align}
We then consider terms by terms on the right hand side of \eqref{main-k-term-07-02}. First of all, we note that
\begin{align*}
& \left| \int_{Q_4} \wei{[\bA_k(x,t)  - (\bA_k)_{B_4}]\nabla u_k(x,t), \nabla \varphi(x,t)}  dxdt \right| \\
& \leq \|\nabla \varphi \|_{L^\infty}  \int_{Q_4} |\bA_k(x,t)  - (\bA_k)_{B_4}| | \nabla u_k(x,t)| dxdt\\
& \leq N \left(\int_{Q_4} |\bA_k(x,t)  - (\bA_k)_{B_4}|^2 \beta_k(x)^{-1} dxdt \right)^{1/2} \left(\int_{Q_4} |\nabla u_k(x,t)|^2 \beta_k(x)  dx dt \right)^{1/2}
\end{align*}
From this, \eqref{a-mu-k}, \eqref{gradient small}, and \eqref{beta_k-measure}, we deduce that
\[
\lim_{k\rightarrow \infty} \int_{Q_4} \wei{\big[\bA_k(x,t)  - (\bA_k)_{B_4}\big] \nabla u_k(x,t), \nabla \varphi(x,t)}dxdt =0.
\]
Next, it follows from the integrations by parts that
\begin{align*}
& \int_{Q_4} \wei{[(\bA_k)_{B_4}(t) - \bar{\bA}(t)] \nabla u_k(x,t), \nabla \varphi(x,t)}  dxdt\\
& =-\int_{Q_4} [(\bA_k)_{B_4}(t) - \bar{\bA}(t)] \cdot D^2 \varphi(x,t) [u_k(x,t) - (u_k)_{Q_4}]dxdt\\
&  = - \int_{Q_4} [(\bA_k)_{B_4}(t) - \bar{\bA}(t)] \cdot D^2\varphi(x,t)  u_0(x,t) dxdt \\
& \quad + \int_{Q_4} [(\bA_k)_{B_4}(t) - \bar{\bA}(t)] \cdot D^2 \varphi(x,t) [u_0(x,t) - u_k(x,t) + (u_k)_{Q_4}] dxdt.
\end{align*}
From this,  and by using \eqref{weak star convergence}, and the first convergence in \eqref{uk-convergence} with \eqref{bar-A-scaled-ellip}, we infer that
\[
\lim_{k\rightarrow \infty}\int_{Q_4} \wei{[(\bA_k)_{B_4}(t) - \bar{\bA}(t)] \nabla u_k(x,t), \nabla \varphi(x,t)} dxdt =0.
\]
Lastly, it follows from the second convergence in \eqref{uk-convergence} and the boundedness of $\bar{\bA}(t)$ and $\nabla \varphi$ that
\[
\lim_{k \rightarrow \infty} \int_{Q_4} \wei{\bar{\bA}(t) \nabla u_k(x,t), \nabla \varphi(x,t)} dxdt = \int_{Q_4} \wei{\bar{\bA}(t) \nabla u_0(x,t), \nabla \varphi(x,t)} dxdt.
\]
Summing up, from the limits we just did for the terms in the right hand side of \eqref{main-k-term-07-02}, we see that
\begin{equation} \label{last-main-07-02}
\lim_{k \rightarrow \infty} \int_{Q_4} \wei{\bA_k(x,t) \nabla u_k(x,t), \nabla \varphi(x,t)}  dxdt = \int_{Q_4} \wei{\bar{\bA}(t) \nabla u_0(x,t), \nabla \varphi(x,t)} dxdt.
\end{equation}

\smallskip
Now, from \eqref{two-main-07-02} and \eqref{last-main-07-02}, we pass the limit as $k \rightarrow \infty$ both sides of \eqref{weak-form-u-k} to obtain
\[
\int_{Q_4} u_0 \partial_t \varphi(x,t)  dxdt = \int_{Q_4} \wei{\bar{\bA}(t) \nabla u_0(x,t), \nabla \varphi(x,t) } dxdt
\]
This implies that $u_0 \in \W^{1, p_0}(Q_4)$ is a weak solution to the equation
\[
\partial_t u_0 -\text{div}(\bar{\bA}(t) \nabla u_0(x,t))= 0 \quad \text{in} \quad Q_4.
\]
We note that as $u_0 \in \W^{1,p_0}(Q_4)$, it follows from Lemma \ref{Lipschitz-constant} that $u_0 \in \W^{1,2}_{\text{loc}}(Q_{4})$. Moreover, it follows from \eqref{u-k-bounded-p-0}  and the first convergence in \eqref{uk-convergence} that
\[
\|u_0\|_{L^{p_0}(Q_4)} + \|\nabla u_0\|_{L^{p_0}(Q_4)}\leq N
\]
for some $N = N(n, \nu, M_0)>0$.  Then, by applying Lemma \ref{Lipschitz-constant} and using \eqref{beta-k-cylinder-0314}, we infer that
\begin{equation} \label{u-zero-bound-0314}
\|\nabla u_0\|_{L^\infty(Q_{7/2, \beta_k})}  \leq \|\nabla u_0\|_{L^\infty(Q_{4 \kappa_0})}\leq N, \quad \forall \ k \in \mathbb{N},
\end{equation}
where $\kappa_0 \in (0,1)$ is defined in \eqref{beta-k-cylinder-0314}.

\smallskip
Next for each $k \in \mathbb{N}$, let $g_k \in \W^{1,p_0}(Q_4)$ be the weak solution of the equation
\[
\left\{
\begin{array}{cccl}
\partial_t g_k -\text{div}((\bA_k)_{B_4}(t) \nabla g_k) & = & -\text{div}\big([(\bA_k)_{B_4}(t) - \bar{\bA}(t)] \nabla u_0\big) & \quad \text{in} \quad Q_4, \\
g_k & = & 0 & \quad \text{on} \quad \partial' Q_4.
\end{array} \right.
\]
We note that as $ -[(\bA_k)_{B_4}(t) - \bar{\bA}(t)] \nabla u_0 \in L^{p_0}(Q_4)$ and the leading coefficient $(\bA_k)_{B_4}(t)$ is uniformly elliptic and bounded as in \eqref{bar-A-scaled-ellip}, and it only depends on the time variable, the existence and uniqueness of $g_k \in \W^{1,p_0}(Q_4)$ is well-known (see \cite{KK1, KK2, DP}, and \cite[Theorem 2.10 and Section 11.1]{Dong-Kim-2011}). Note also that as $\nabla u_0 \in L^2_{\text{loc}}(Q_4)$, it follows that $g_k \in \W^{1,2}_{\text{loc}}(Q_4)$. In addition, by the well-known regularity theory (see \cite{DP}), for every $\bar{q} \in (1, \infty)$, there are $N_1 = N_1(n, \nu, M_0)>0$ and $N_2 = N_2(n, \nu, M_0, \bar{q})>0$
\begin{equation} \label{h-k-est-07-1}
\begin{split}
& \|g_k\|_{\W^{1,p_0}(Q_4)} \leq N \|[(\bA_k)_{B_4}(t) - \bar{\bA}(t)] \nabla u_0\|_{L^{p_0}(Q_4)} \leq N_1, \\
& \|\nabla g_k \|_{L^{\bar{q}}(Q_{4\kappa_0})} \leq N\Big[\|g_k\|_{L^{p_0}(Q_4)} + \|\nabla u_0\|_{L^{\infty}(Q_{2(1+\kappa_0)})} \Big] \leq N_2,
\end{split}
\end{equation}
for all $k \in \mathbb{N}$, where $\kappa_0$ is defined in \eqref{beta-k-cylinder-0314} and we also used  Lemma \ref{Lipschitz-constant} in the last estimate. Due to the first assertion in \eqref{h-k-est-07-1}, the Aubin-Lions theorem (see \cite{Simon}), and by passing through a subsequence, we can find $g_0 \in \W^{1,p_0}(Q_4)$ such that
\begin{equation} \label{gk-convergence}
\begin{array}{lll}
 g_k \rightarrow g_0 &\quad \text{strongly in} & \quad L^{p_0}(Q_4),\\
\nabla g_k \rightharpoonup \nabla g_0  & \quad \text{weakly in} & \quad  L^{p_0}(Q_4),\\
\partial_t g_k \rightharpoonup \partial_t g_0  & \quad \text{weakly in} & \quad  L^{p_0}((-16, 0), W^{-1,p_0}(B_4)),
\end{array}
\end{equation}
as $k \rightarrow \infty$. By using \eqref{weak star convergence}, the weak formulation of $g_k$, \eqref{gk-convergence}, and then passing through the limit as $k \rightarrow \infty$, we see that $g_0 \in \W^{1,p_0}(Q_4)$ is the weak solution of the equation
\[
\left\{
\begin{array}{cccl}
\partial_t g_0 -\text{div}(\bar{\bA}(t) \nabla g_0) & = & 0 & \quad \text{in} \quad Q_4, \\
g_0 & = & 0 & \quad \text{on} \quad \partial' Q_4.
\end{array} \right.
\]
From this, and the uniqueness of the weak solution to this equation, we see that $g_0 =0$ in $Q_4$. Consequently, it follows from the first assertion in \eqref{gk-convergence} that
\begin{equation} \label{g-k-zero-con}
\lim_{k \rightarrow \infty}\|g_k\|_{L^{p_0}(Q_4)} =0.
\end{equation}
Let $v_k = u_0 - g_k \in \W^{1, p_0}(Q_4) \cap \W^{1,2}_{\text{loc}}(Q_4)$, we see that $v_k$ is a weak solution of the equation
\[
\partial_t v_k - \text{div}((\bA_k)_{B_4} \nabla v_k) = 0 \quad \text{in} \quad Q_4.
\]
Combining \eqref{u-zero-bound-0314} and \eqref{h-k-est-07-1} with $\bar q=2$, and using Lemma \ref{property}-(i), we infer that $v_k$ satisfies \eqref{L-2-mu}. Applying the classical local boundedness estimate to $u_0$ and $v_k$, we infer that $\|g_k\|_{L^\infty(Q_{4\kappa_0})}\leq N$. From this and \eqref{g-k-zero-con}, we infer that $g_k\to0$ strongly in $L^2((-16\kappa_0^2,0),L^{p_0}(B_{4\kappa_0}))$.

\smallskip
Now, let $w_k =u_k(x,t) - (u_k)_{Q_{4}} - v_k(x,t)$ and we claim that
\begin{equation} \label{inter-claim}
\lim_{k \rightarrow \infty} \fint_{Q_{7/2, \beta_k}} |w_k(x,t) -[w_k]_{0, 5/2}(t)|^2 \beta_k(dxdt) =0.
\end{equation}
Note that once this claim is proved, we obtain a contradiction to \eqref{w-k-contra} by taking $v=v_k$ in \eqref{w-k-contra} with sufficiently large $k$. From this, the assertion \eqref{L-2-w-small} is proved. 

\smallskip
We now focus on proving the assertion \eqref{inter-claim}. Let us denote $\hat{w}_k(x,t) = w_k(x,t) - [w_k]_{0,5/2}(t)$. By the Sobolev-Poincar\'{e} inequality (Lemma \ref{Sobolev-Poincare}) and the reverse H\"{older} property of the Muckenhoupt weights (Lemma \ref{R-Holder}), there are $\kappa>1$ and $\gamma = \gamma(n, M_0) \in (0,1)$ such that
\begin{align}\notag
& \Big[\fint_{\Gamma_{\beta_k}(7/2)}\Big(\fint_{B_{7/2}} |\hat{w}_k(x,t)|^{2\kappa}\beta_k(dx)\Big)^{1/\kappa}dt\Big]^{1/2}  \\ \notag 
& \leq N \Big (\fint_{Q_{7/2, \beta_k}} |\nabla u_k(x,t)|^{2} \beta_k(dxdt) \Big)^{\frac{1}{2}} + N \Big(\fint_{Q_{7/2, \beta_k}} |\nabla u_0(x,t)|^{2}  \beta_k(dxdt) \Big)^{\frac{1}{2}} \\
& \qquad \quad + N \Big(\fint_{Q_{7/2, \beta_k}} |\nabla g_k(x,t)|^{2}  \beta_k(dxdt) \Big)^{\frac{1}{2}} \\  \notag
& \leq N \Big (\fint_{Q_{4, \beta_k}} |\nabla u_k(x,t)|^{2} \beta_k(dxdt) \Big)^{\frac{1}{2}} +  N  \|\nabla u_0\|_{L^\infty(Q_{7/2, \beta_k})}  \\ \notag
& \qquad + N \Big( \frac{1}{|\Gamma_{\beta_k}(7/2)|\beta_k(B_{7/2})} \Big)^{1/2} \Big( \fint_{B_4} \beta_k(x)^{1+\gamma} dx\Big)^{\frac{1}{2(1+\gamma)}}\|\nabla g_k\|_{L^{\frac{2(\gamma+1)}{\gamma}}(Q_{4\kappa_0})}   \\ \label{w-k-bounded-0318}
& \leq N,
\end{align}
where we have used Lemma \ref{R-Holder}, \eqref{beta_k-measure}, the doubling property of $\beta_k$, \eqref{u-zero-bound-0314}, Lemma \ref{property}-(i), and the second estimate in \eqref{h-k-est-07-1} for $\bar q =\frac{2(\gamma+1)}{\gamma}$. Now, let $\tau = \tau(n, M_0)>0$ be a sufficiently small number to be determined, and let
\[
\theta = \frac{1/2 - 1/(2\kappa)}{1/\tau - 1/(2\kappa)} \in (0,1).
\]
Applying the H\"{o}lder inequality first in the spatial variable and then in the time variable, and using \eqref{w-k-bounded-0318}, we infer that
\begin{align*}
& \Big( \fint_{Q_{7/2, \beta_k}} |\hat{w}_k(x,t)|^{2}\beta_k(dxdt)  \Big)^{\frac{1}{2}}\\
&\leq \Big[\fint_{\Gamma_{\beta_k}(7/2)}\Big(\fint_{B_{7/2}}|\hat w_k|^\tau\beta_k(dx)\Big)^{2/\tau}dt\Big]^{\theta/2}
\Big[\fint_{\Gamma_{\beta_k}(7/2)}\Big(\fint_{B_{7/2}}|\hat w_k|^{2\kappa}\beta_k(dx)\Big)^{1/\kappa}dt\Big]^{(1-\theta)/2}\\
&\leq N^{1-\theta}\Big[\fint_{\Gamma_{\beta_k}(7/2)}\Big(\fint_{B_{7/2}}|\hat w_k|^\tau\beta_k(dx)\Big)^{2/\tau}dt\Big]^{\theta/2}.
\end{align*}
On the other hand, from Lemma \ref{R-Holder} and \eqref{beta_k-measure}, we can find $\gamma_0=\gamma_0(n,M_0)>0$ such that
\[
\Big(\fint_{B_{7/2}}|\hat w_k|^\tau\beta_k(dx)\Big)^{1/\tau}\leq N\Big(\fint_{B_{7/2}}|\hat w_k|^{\tau(1+\gamma_0)}dx\Big)^{1/[\tau(1+\gamma_0)]}.
\]
Then, by choosing $\tau=p_0/(1+\gamma_0)$, we obtain
\begin{align*}
&\Big(\fint_{Q_{7/2,\beta_k}}|\hat w_k|^2\beta_k(dxdt)\Big)^{1/2}\leq N\Big[\fint_{\Gamma_{\beta_k}(7/2)}\Big(\fint_{B_{7/2}}|\hat w_k|^{p_0}dx\Big)^{2/p_0}dt\Big]^{\theta/2}\\
&\longrightarrow0\quad\text{as }k\to\infty,
\end{align*}
where the last assertion is due to the first convergence in \eqref{uk-convergence}, the strong convergence of $g_k$ in $L^2((-16\kappa_0^2,0),L^{p_0}(B_{4\kappa_0}))$, H\"{o}lder's inequality, \eqref{average-eta-weight}, and Lemma \ref{property}-(i). The claim \eqref{inter-claim} is then proved, which also completes the proof of \eqref{L-2-w-small}. As $v_k$ also satisfies \eqref{L-2-mu}, the proof of the lemma is completed.
\end{proof}
We conclude the section with the following proposition, which is the main result of the section.
\begin{proposition} \label{L2-gradient-comparision} For every $\nu\in (0,1)$, $M_0 \geq 1$, and $\epsilon \in (0,1)$, there exist a sufficiently small constant $\delta = \delta(n, \nu, M_0, \epsilon) \in (0,1)$ and a constant $N = N(n, \nu, M_0)>0$ such that the following assertions hold. Suppose that $\bA$ satisfies \eqref{ellip-cond} in $Q_{4, \beta}$ and its corresponding $\beta$ satisfies \eqref{beta-cond}. Suppose also that
\begin{equation} \label{small-data-interior}
 \fint_{Q_{4,\beta}} \big(\Theta_{\bA, 4}(x,t)^2 + |F(x,t)|^2\big)\beta(dx dt) \leq \delta^2.
\end{equation}
Then, for every weak solution $u \in \W^{1,2}(Q_{4,\beta}, \beta)$ of \eqref{eqn-Q_R} satisfying
\[
\fint_{Q_{4,\beta}} |\nabla u(x,t)|^2\beta(dxdt) \leq 1,
\]
there exists a weak solution $v \in \W^{1,2}_{\textup{loc}}(Q_{4,\beta})$ of \eqref{v-Qr-sol} such that
\begin{equation} \label{L-2-gradient-w-small}
\fint_{Q_{2, \beta}} |\nabla u(x,t)-\nabla v(x,t)|^2\beta( dx dt)  \leq \epsilon^2 \ \text{and} \ \|\nabla v\|_{L^\infty(Q_{5/2, \beta})}  \leq N.
\end{equation}
\end{proposition}
\begin{proof} For the given $\epsilon>0$, let $\bar{\epsilon} = \bar{\epsilon}(n, \nu, M_0, \epsilon) \in (0,1)$ be a sufficiently small so that
\begin{equation} \label{bar-epsilon-choice}
N_0 \bar{\epsilon}^{1-\theta} \leq \frac{\epsilon}{2}
\end{equation}
where $N_0 = N_0(n, \nu, M_0)>0$ is the number defined in \eqref{nabla-w-inter-est-0330} below, and $\theta = \theta(n, M_0) \in (0,1)$ defined in Lemma \ref{imbed-epsilon-0309}.  Then, with the given number $\nu$ and $M_0$, and this $\bar{\epsilon}$, let $\bar{\delta} = \bar{\delta}(n, \nu, M_0, \bar{\epsilon}) \in (0,1)$ be the small number defined in Lemma \ref{L2-comparision}.  

\smallskip
Now, we define
\begin{equation} \label{delta-choice-inter}
\delta =\min\big\{\bar{\delta}, \frac{\epsilon}{2N_0}\big \}
\end{equation}
and we prove the proposition with this choice of $\delta$. Corresponding to $\nu, M_0$, and $\bar{\epsilon}$, let $v \in \W^{1,2}_{\text{loc}}(Q_{4, \beta})$ be the function defined in Lemma \ref{L2-comparision}, which is a weak solution of \eqref{v-Qr-sol} that satisfies
\begin{equation} \label{w-inter-L2-wei}
\Big (\fint_{Q_{7/2, \beta}} |\hat{w}(x,t)|^2 \beta(dxdt)\Big)^{1/2} \leq \bar{\epsilon} \quad \text{and} \quad \Big (\fint_{Q_{3, \beta}} |\nabla v(x,t)|^2 dx dt \Big)^{\frac{1}{2}} \leq N,
\end{equation}
where $\hat{w}(x,t) = w(x,t) -[w]_{0, 5/2}(t)$ with $w(x,t) = u(x,t) - (u)_{Q_{4,\beta}} - v(x,t)$, and $N = N(n, \nu, M_0)>0$. We then apply Lemma \ref{Lipschitz-constant} and the doubling property of $\beta$ to infer that
\[
\|\nabla v\|_{L^\infty(Q_{5/2, \beta})} \leq N\Big (\fint_{Q_{3, \beta}} |\nabla v(x,t)|^2 dx dt \Big)^{\frac{1}{2}} \leq N
\]
with some $N = N(n, \nu, M_0)>0$. The second assertion in \eqref{L-2-gradient-w-small} is proved. 

\smallskip
It remains to prove the first assertion in \eqref{L-2-gradient-w-small}. To this end, we apply Lemma \ref{imbed-epsilon-0309}-(i), H\"{o}lder's inequality in time, the doubling property of $\beta$, and the triangle inequality to infer that
\begin{align} \notag
& \Big (\fint_{Q_{5/2, \beta}} |\hat{w}(x,t)|^2 dxdt \Big)^{\frac{1}{2}}  \\ \notag
& \leq N\Big (\fint_{Q_{5/2, \beta}} |\hat{w}(x,t)|^2 \beta(dxdt) \Big)^{\frac{1-\theta}{2}} \Big (\fint_{Q_{5/2, \beta}} |\nabla w(x,t)|^2\beta(dxdt) \Big)^{\frac{\theta}{2}} \\ \notag
& \leq \bar{\epsilon}^{1-\theta} N \left[ \Big (\fint_{Q_{5/2, \beta}} |\nabla u(x,t)|^2 \beta(dxdt) \Big)^{\frac{1}{2}}
+ \|\nabla v\|_{L^\infty(Q_{5/2,\beta})} \right]^{\theta} \\ \label{w-inter-L2}
& \leq N \bar{\epsilon}^{1-\theta},
\end{align}
where $N = N(n, \nu, M_0)>0$.

\smallskip
Now, by using Lemma \ref{compare-lemma-1} with $\rho = 5/2$ and a suitable cut-off function $\chi$,  and using the estimates \eqref{small-data-interior}, \eqref{w-inter-L2-wei}, \eqref{w-inter-L2}, and the doubling property of $\beta$, we can find for some $N_0 = N_0(n, \nu, M_0)>0$ so that
\begin{equation} \label{nabla-w-inter-est-0330}
\Big (\fint_{Q_{2, \beta}} |\nabla w(x,t)|^2 \beta(dxdt) \Big)^{\frac{1}{2}}  \leq N_0[\bar{\epsilon}^{1-\theta} + \delta].
\end{equation}
From this, and the choice of $\bar{\epsilon}$ in \eqref{bar-epsilon-choice} and the choice of $\delta$ in \eqref{delta-choice-inter}, we conclude that
\[
\Big (\fint_{Q_{2, \beta}} |\nabla w(x,t)|^2 \beta(dxdt) \Big)^{\frac{1}{2}} \leq \epsilon.
\]
The first assertion in \eqref{L-2-gradient-w-small} is proved, and the proof is completed.
\end{proof}
\begin{proof}[Proof of Theorem \ref{inter-theorem}] Theorem \ref{inter-theorem} follows from Proposition \ref{L2-gradient-comparision} and the level set argument introduced in \cite{CP}. We skip the details as the argument is standard now. Interested readers can find \cite{FP-2} for similar setting with completed proof, for examples.
\end{proof}
\section{Boundary estimates and proof of Theorem \ref{inter-theorem-bdry}} \label{bdr-section}
This section is to prove Theorem \ref{inter-theorem-bdry}. We localize the equation \eqref{eqn-bdr-10-4} using the freezing coefficient technique and apply the level-set method. As in the interior case, it is sufficient to consider the following parabolic problem
\begin{equation} \label{eqn-bdry}
\left\{
\begin{aligned}
u_t - \textup{div}(\bA(x,t) \nabla u)  &= \textup{div}(\beta(x)F(x,t))\quad &&\text{in}\quad Q^+_{4,\beta},\\
u&=0 \quad &&\text{on} \quad T_{4} \times \Gamma_{\beta}(4).
\end{aligned}\right.
\end{equation}
We recall that the space $\hW^{1,p}(Q_{1,\beta}^+, \beta)$ is defined in the paragraph before Theorem \ref{inter-theorem-bdry}. Also, the definition of weak solutions $u \in \hat{\W}^{1,2}(Q^+_{4, \beta}, \beta)$ to \eqref{eqn-bdry} is given in Definition \ref{boundary-weak-def}.  
\subsection{Boundary Poincar\'e Estimates and Caccioppoli Type Estimates}  For a given number $M_0 \geq 1$, let $\gamma = \gamma(n, M_0)$ be the small number defined in Lemma \ref{L-q-2-wei} in which $q = 1+\frac{1}{n_0}$ for $n_0 = \max\{n-1, 1\}$. Then, the same as in \eqref{p-zero-def}, we define
\begin{equation*}
p_0 = \min \Big\{ \frac{2(1+\gamma)}{2+ \gamma}, \frac{2}{1+\frac{1}{n_0} -\gamma} \Big\} \in (1, 2).
\end{equation*}
We begin with the following lemma on Poincar\'e type inequality for weak solutions to \eqref{eqn-bdry}.
\begin{lemma} \label{L-2-bdry-u-est} For every $\nu \in (0,1)$ and $M_0 \geq 1$, there is $N = N(n, \nu, M_0)>0$ such that the following assertion holds. Suppose $u \in \hat{\W}^{1,2}(Q^+_{4, \beta}, \beta)$ is a weak solution of \eqref{eqn-bdry} with a matrix $\bA$ satisfying \eqref{ellip-cond} in $Q_{4,\beta}^+$ and its corresponding $\beta$ satisfying \eqref{beta-cond}, and with some $F \in L^2(Q_{4,\beta}^+, \beta)^n$. Then
\begin{equation*}
\begin{split}
& \Big( \fint_{Q_{4, \beta}^+} | u (x,t)|^{p_0} dx dt \Big)^{\frac{1}{p_0}}   \\
& \leq N \Big(\fint_{Q_{4, \beta}^+}|\nabla u(x,t)|^2 \beta(dx dt) \Big)^{\frac{1}{2}} + \Big(\fint_{Q_{4, \beta}^+}  |F(x,t)|^2  \beta(dx dt) \Big)^{\frac{1}{2}}.
\end{split}
\end{equation*}
\end{lemma}
\begin{proof} We use the contradiction argument as in the proof of Lemma \ref{L-2-inter-u-est}. Instead of using  Lemma \ref{L-q-2-wei} for the whole balls, we use its version with the upper-half balls. We skip the details as it is exactly the same as that of Lemma \ref{L-2-inter-u-est}.
\end{proof}

\smallskip
As aplly the perturbation technique by freezing the coefficients prove Theorem \ref{inter-theorem-bdry}, we  study the following equation with frozen coefficients
\begin{equation} \label{bdry-approx}
\left\{
\begin{aligned}
v_t - \textup{div} ((\bA)_{B^+_{4}}(t) \nabla v)  &= 0 \quad &&\text{in}\quad Q^+_{4,\beta},\\[4pt]
v&=0 \quad &&\text{on} \quad T_{4} \times \Gamma_{\beta}(4).
\end{aligned}\right.
\end{equation}
Let us recall the following notation defined in \eqref{Theta-beta-def} in the boundary setting
\begin{equation} \label{Theta-A-def+}
 \Theta_{\bA, r} (x,t)= |\bA(x,t) - (\bA)_{B_{r}^+}(t)|\beta(x)^{-1} \quad \text{where} \quad (\bA)_{B^+_{r}}(t)= \fint_{B^+_{r}} \bA(x, t) dx.
\end{equation}
Similar to Lemma \ref{compare-lemma-1}, we have the following important lemma on boundary Caccioppoli type estimates for the difference of two solutions of \eqref{eqn-bdry} and \eqref{bdry-approx}.
\begin{lemma} \label{bdry-Caccioppoli}  Suppose that $u \in \hat\W^{1,2}(Q^+_{4,\beta}, \beta)$ is a weak solution of \eqref{eqn-bdry} in which the matrix $\bA$ satisfies \eqref{ellip-cond} in  $Q^+_{4,\beta}$ and its corresponding $\beta$ satisfies \eqref{beta-cond}, and $F \in L^2(Q_{4,\beta}^+, \beta)^n$. Suppose also that  $v \in \hat\W^{1,p_0}_{\textup{loc}}(Q^+_{4,\beta})$ is a weak solution of \eqref{bdry-approx} with $p_0 \in (1, 2)$. Then, for  $w = u -v$,  it holds that
\begin{align*}
& \int_{Q^+_{4, \beta}} |\nabla w(x,t)|^2 \varphi(x,t)^2 \beta(x) dx dt \leq   N \int_{Q^+_{4, \beta}} |F(x,t)|^2 \varphi^2(x,t) \beta(x) dx dt \\
& \quad + N\int_{Q^+_{4, \beta}} w(x,t)^2  \big[|\varphi| |\partial_t \varphi| + |\nabla \varphi|^2 \beta(x) \big]  dxdt  \\
& \qquad +   N \|\varphi \nabla v\|_{L^\infty(Q^+_{4, \beta})}^2 \int_{Q^+_{4, \beta}} \Theta_{\bA, 4}(x,t)^2 \beta(x) dx dt,
\end{align*}
for every $\varphi  \in C^\infty(Q_{4, \beta})$ with $0 \leq \varphi \leq 1$ and vanishing near $\partial' Q_{4,\beta}$, where $N = N(n, \nu)>0$.
\end{lemma}

\smallskip
\begin{proof} To begin with, it follows from Lemma \ref{Bdr-Lipschitz-constant} that $\nabla v  \in L^{\infty}_{\text{loc}}(Q^+_{4,\beta})$. Consequently,  $w = u - v \in  \hat\W^{1,2}_{\text{loc}}(Q^+_{4, \beta}, \beta)$ is a weak solution of
\begin{equation*}
w_t  - \textup{div}( \bA \nabla w)  = \textup{div} (\beta(x) G(x,t))  \quad \text{in} \quad Q^+_{4,\beta},
\end{equation*}
where
\[
G(x,t)= F(x,t) + (\bA (x,t)- (\bA)_{B^+_{4}}(t)) \nabla v(x,t)\beta(x)^{-1}.
\]
with the boundary condition $w=0$ on $T_4\times \Gamma_{\beta}(4)$. Then, by using Steklov's average as in \cite[p. 18]{DiB}, we multiply equation of $w$ with $w\varphi^2$ and use the integration by parts to obtain
\begin{align*}
& \frac{1}{2} \frac{d}{dt} \int_{B_4^{{+}}} w(x,t)^2 \varphi(x,t)^2 dx + \int_{B_4^{{+}}} \wei{\bA(x,t) \nabla w(x,t), \nabla w(x,t)} \varphi(x,t)^2 dx \\
& = \int_{B_4^{{+}}} \big[w(x,t)^2 \varphi(x,t) \varphi_t(x,t) - 2\wei{\bA(x,t) \nabla w(x,t), \nabla \varphi(x,t)} \varphi(x,t) w(x,t) \big]dx \\
& \quad - \int_{B_4^{{+}}} \beta(x)\big[ \wei{G(x,t), \nabla w(x,t) } \varphi(x,t)^2 + 2 \wei{G(x,t), \nabla \varphi(x,t)} \varphi(x,t) w(x,t) \big] dx.
\end{align*}
Then, applying  \eqref{ellip-cond} and Young's inequality, we infer that
\begin{align*}
&  \frac{d}{dt} \int_{B_4^{{+}}} w(x,t)^2 \varphi(x,t)^2 dx + \nu \int_{B_4^{{+}}} |\nabla w(x,t)|^2 \varphi(x,t)^2 \beta(x) dx \\
& \leq N \int_{B_4^{{+}}} w(x,t)^2 \big[|\varphi| |\partial_t \varphi| + |\nabla \varphi|^2 \beta(x) \big] dx + N \int_{B_4^{{+}}} |G(x,t)|^2 |\varphi(x,t)|^2  \beta(x) dx.
\end{align*}
Using \eqref{Theta-A-def+}, we infer that
\begin{align*}
& \int_{B_4^{{+}}} |G(x,t)|^2 |\varphi(x,t)|^2  \beta(x) dx \leq  N\int_{B_4^{{+}}} |F(x,t)|^2 |\varphi(x,t)|^2  \beta(x) dx \\
& \quad + N \|\varphi \nabla v\|_{L^\infty(Q^+_{4, \beta})}^2  \int_{B_4^{{+}}} |\Theta_{\bA, 4}(x,t)|^2 \beta(x) dx.
\end{align*}
From this and the standard method with the time integration, assertion of the lemma follows.
\end{proof}
\subsection{Boundary Approximation Estimates and Proof of Theorem \ref{inter-theorem-bdry}} We begin with the following lemma which is the boundary version of  Lemma \ref{L2-comparision}  that prove the $L^2(Q_{7/2, \beta}, \beta)$ closeness of solutions of the two PDEs \eqref{eqn-bdry} and \eqref{bdry-approx}. 
\begin{lemma} \label{L2-comparision-bdry} For every $\nu \in (0,1)$ , $M_0 \geq 1$, and a small number $\epsilon \in (0, 1)$, there exist a sufficiently small constant $\bar\kappa = \bar \kappa(n, \nu, M_0, \epsilon)$  and a number $N = N(n, \nu, M_0)>0$ such that the following assertions hold. Suppose $\bA$ satisfies \eqref{ellip-cond} in $Q^+_{4, \beta}$ and its corresponding $\beta$ satisfies \eqref{beta-cond}. Suppose also that
\begin{align*}
& \fint_{Q^+_{4,\beta}}\Theta_{\bA, 4}(x,t)^2 \beta(dxdt)  + \fint_{Q^+_{4,\beta}}  |F(x,t)|^2 \beta(dx dt) \leq \bar \kappa^2.
\end{align*}
Then, for every weak solution $u \in \hat\W^{1,2}(Q^+_{4,\beta},\beta)$ of \eqref{eqn-bdry} satisfying
\[
\fint_{Q^+_{4,\beta}} |\nabla u(x,t)|^2\beta(dxdt) \leq 1,
\]
there exists a weak solution $v \in \hat\W^{1,2}_{\textup{loc}}(Q^+_{4,\beta})$ of \eqref{bdry-approx} such that
\begin{equation} \label{L-2-w-small-bdry}
\begin{split}
&\Big(\fint_{Q^+_{7/2, \beta}} |u(x,t)-v(x,t)|^2 \beta(dx dt)\Big)^{1/2}  \leq \epsilon.
\end{split}
\end{equation}
Moreover,
\[
\Big(\fint_{Q^+_{3,\beta}} |\nabla v(x,t)|^2 dx dt\Big)^{\frac{1}{2}} \leq N.
\]
\end{lemma}
\begin{proof}  By applying the dilation in the time variable as in \eqref{scale-2}, we may assume without loss of generality that
\begin{equation} \label{scaling-beta-k-bdry}
\E_{\beta}(4) =\Big( \fint_{B_4} \beta(x)^{-n_0} dx\Big)^{\frac{1}{n_0}} =1 \quad \text{then} \quad Q^+_{4,\beta}=Q^+_4.
\end{equation}
To prove the assertion \eqref{L-2-w-small-bdry}, we use the same contradiction argument as in the proof of Lemma \ref{L2-comparision}.  The argument is similar but there are some few different details which we point out below.

\smallskip
Following the contradiction argument as in the proof of Lemma \ref{L2-comparision}, there exist $\epsilon_0>0$, a sequence of coefficient matrices $\{\bA_k\}_k$ satisfying \eqref{ellip-cond} in $Q_{4}^+$ with its corresponding  sequence of weights $\{\beta_k\}_k$ satisfying \eqref{beta-cond} and \eqref{scaling-beta-k-bdry} for each $k\in \N$, i.e., $\beta_k \in A_{1+\frac{1}{n_0}}$ and for $n_0 = \max\{n-1, 1\}$, $[\beta_k]_{A_{1+\frac{1}{n_0}}} \leq M_0$,
\begin{equation} \label{nabla-u-k-bdry}
[\beta_k]_{A_{1+\frac{1}{n_0}}} \leq M_0, \ \Big(\fint_{B_4} \beta_k(x)^{-n_0} dx \Big)^{\frac{1}{n_0}} =1, \ \text{and} \ 1 \leq \fint_{B_4} \beta_k(x) dx \leq M_0,
\end{equation} 
for all $k \in \mathbb{N}$. Moreover, there is a sequence $\{F_k\}_k \subset L^2(Q_4, \beta_k)^n$ such that
\begin{align} \label{a-mu-k-bdry}
\fint_{Q_{4}^+}\Big[ \Theta_{\bA_k, 4}(x,t)^2 + |F_k(x,t)|^2 \Big]\beta_k(dx dt) \leq \frac{1}{k^2},
\end{align}
and there exists a sequence $\{u_k\}_k \subset \hat{\W}^{1,2}(Q_4^+, \beta_k)$ satisfying
\begin{align}\label{gradient small-bdry}
\fint_{Q_{4}^+} |\nabla u_k(x,t)|^2\beta_k(dxdt) \leq 1, \quad \forall \ k \in \mathbb{N},
\end{align}
where each $u_k \in \hat{\W}^{1,2}(Q_4^+, \beta_k)$ is a weak solution of 
\[
\left\{
\begin{array}{cccl}
\partial_tu_k-\textup{div}(\bA_k(x,t)\nabla u_k) & = &\textup{div}(\beta_k(x) F_k(x,t)) & \quad \text{in}\quad Q_4^+ \\
u_k & = & 0 & \quad \text{on} \quad T_4 \times (-16, 0).
\end{array} \right.
\]
However, for every $k \in \mathbb{N}$, and for every weak solution $v \in \hat{\W}^{1,2}_{\text{loc}}(Q_{4}^+)$ satisfying the second assertion of the lemma with a fixed sufficiently large $N=N(n,\nu,M_0)$ of the equation
\[
\left\{
\begin{array}{cccl}
v_t - \textup{div} [(\bA_k)_{B_{4}^+}(t) \nabla v]  & = &  0 & \quad \text{in} \quad Q_{4}^+ \\
v & = & 0 & \quad \text{on} \quad T_4 \times (-16, 0).
\end{array} \right.
\]
we have
\begin{equation} \label{w-k-contra-bdry}
\fint_{Q^+_{7/2, \beta_k}} |u_k(x,t) - v(x,t)|^2\beta_k(dxdt) \geq \epsilon_0.
\end{equation}

\smallskip
As in \eqref{bar-A-scaled-ellip}, using \eqref{nabla-u-k-bdry} and the $A_2$ condition to compare the averages of $\beta_k$ on $B_4^+$ and $B_4$, we have
\begin{equation} \label{bar-A-scaled-ellip-bdry}
c\nu |\xi|^2 \leq \wei{(\bA_k)_{B_4^+}(t) \xi, \xi} \quad \text{and} \quad |(\bA_k)_{B_4^+}(t)| \leq 2\nu^{-1} M_0, \quad \forall \ k \in \mathbb{N}.
\end{equation}
where $c=c(n,M_0)>0$. Then, let $\bar{\bA}(t)\in L^{\infty}((-16,0), \M^{n\times n})$ be a matrix satisfying \eqref{bar-A-scaled-ellip-bdry} such that \begin{equation}\label{weak star convergence--bdry}
\lim_{k\rightarrow \infty}\int_{-16}^{0} [(\bA_k)_{B_4^+}(t)-\bar{\bA}(t)]\cdot \bB(t) dt=0,
\end{equation}
for all $\bB(t)\in L^{1}((-16,0), \M^{n\times n})$. 

\smallskip
Then as in \eqref{u-k-bounded-p-0}, by applying Lemma \ref{L-2-bdry-u-est},  \eqref{a-mu-k-bdry}, and \eqref{gradient small-bdry} we find $N = N(n, \nu, M_0)>0$ such that
\begin{equation} \label{u-k-bounded-p-0-bdry}
 \|u_k\|_{\hat{\W}^{1,p_0}(Q_4^+)} \leq N, \quad \forall \ k \in \mathbb{N}.
\end{equation}
From \eqref{u-k-bounded-p-0-bdry} and the Aubin-Lions theorem (see \cite{Simon}) and by passing a subsequence, we can find $u_0 \in \W^{1,p_0}(Q_4^+)$ such that
\begin{equation} \label{uk-convergence-bdry}
\begin{array}{lll}
 u_k \rightarrow u_0 &\quad \text{strongly in} & \quad  L^{p_0}(Q_4^+),\\
\nabla u_k \rightharpoonup \nabla u_0  & \quad \text{weakly in} & \quad  L^{p_0}(Q_4^+),\\
\partial_t u_k \rightharpoonup \partial_t u_0  & \quad \text{weakly in} & \quad  L^{p_0}((-16, 0), W^{-1,p_0}(B_4^+)),
\end{array}
\end{equation}
as $k \rightarrow \infty$. From \eqref{uk-convergence-bdry}, and by passing through a limit as $k \rightarrow \infty$ as in the proof of Lemma \ref{L2-comparision}, we infer that $u_0 \in \hat{\W}^{1, p_0}(Q_4^+)$ is a weak solution to the equation
\[
\left\{
\begin{array}{cccl}
\partial_t u_0 -\text{div}(\bar{\bA}(t) \nabla u_0(x,t)) & = & 0 & \quad \text{in} \quad Q_4^+,\\
u_0 & = & 0 & \quad \text{on} \quad T_4 \times (-16, 0).
\end{array} \right.
\]
Also,  by applying Lemma \ref{Bdr-Lipschitz-constant} and using \eqref{beta-k-cylinder-0314} (with $\kappa_0$ chosen as in the proof of Lemma \ref{L2-comparision}), we infer that $u_0 \in \W^{1,2}_{\text{loc}}(Q_4^+)$, and 
\begin{equation} \label{u-zero-bound-0314-bdry}
\|\nabla u_0\|_{L^\infty(Q^+_{7/2, \beta_k})}  \leq \|\nabla u_0\|_{L^\infty(Q^+_{4 \kappa_0})} \leq N\|\nabla u_0\|_{L^{p_0}(Q_4^+)}\leq N, \quad \forall \ k \in \mathbb{N}.
\end{equation}

\smallskip
Next, let us denote $D=\mathbb R^n_+\times(-16,0]$. For each $k\in\mathbb N$, we define
\[
H_k(x,t)=\begin{cases}
[(\bA_k)_{B_4^+}(t)-\bar{\bA}(t)]\nabla u_0(x,t), & (x,t)\in Q_4^+,\\
0, & (x,t)\in D\setminus Q_4^+.
\end{cases}
\]
Let $g_k\in\W^{1,p_0}_*(D)$ be the weak solution of the equation
\begin{equation} \label{g-k-bdry-eqn}
\left\{
\begin{array}{cccl}
\partial_t g_k-\text{div}((\bA_k)_{B_4^+}(t)\nabla g_k)&=&-\text{div}(H_k)&\quad\text{in }D,\\
g_k&=&0&\quad\text{on }\partial' D.
\end{array}\right.
\end{equation}
Note that $H_k\in L^{p_0}(D)^n$. By \eqref{bar-A-scaled-ellip-bdry} and the solvability theory (see \cite[Theorem 2.5 and Remark 2.4]{Dong-Kim-2011}), we obtain the existence and uniqueness of $g_k$ and
\begin{equation}\label{h-k-est-07-1-bar}
\|g_k\|_{\W^{1,p_0}(D)}\leq N\|H_k\|_{L^{p_0}(D)}\leq N_1, \quad \forall\ k\in\mathbb N,
\end{equation}
where $N=N(n,\nu,M_0)>0$ and $N_1=N_1(n,\nu,M_0)>0$.

\smallskip
Let $v_k=u_0-g_k$ in $Q_4^+$. Then $v_k\in\hat{\W}^{1,p_0}(Q_4^+)$ is a weak solution of the equation
\[
\left\{
\begin{array}{cccl}
\partial_t v_k-\text{div}((\bA_k)_{B_4^+}(t)\nabla v_k)&=&0&\quad\text{in }Q_4^+,\\
v_k&=&0&\quad\text{on }T_4\times(-16,0).
\end{array}\right.
\]
Applying Lemma \ref{Bdr-Lipschitz-constant} to $v_k$, and using \eqref{h-k-est-07-1-bar} and \eqref{u-zero-bound-0314-bdry}, we infer that $v_k\in\hat{\W}^{1,2}_{\text{loc}}(Q_4^+)$ and
\[
\|\nabla v_k\|_{L^\infty(Q_{4\kappa_0}^+)}\leq N\|\nabla v_k\|_{L^{p_0}(Q_4^+)}\leq N,
\]
where $N=N(n,\nu,M_0)>0$. As $g_k=u_0-v_k$ in $Q_4^+$, combining this and \eqref{u-zero-bound-0314-bdry}, we infer that $\|\nabla g_k\|_{L^\infty(Q_{4\kappa_0}^+)}\leq N$. Hence, for every $\bar q\in(1,\infty)$,
\begin{equation}\label{h-k-grad-est-bdry}
\|\nabla g_k\|_{L^{\bar q}(Q_{4\kappa_0}^+)}\leq N_2, \quad \forall\ k\in\mathbb N,
\end{equation}
where $N_2=N_2(n,\nu,M_0,\bar q)>0$.

\smallskip
Due to \eqref{h-k-est-07-1-bar}, the Aubin-Lions theorem (see \cite{Simon}) on bounded upper-half balls, and by passing to a subsequence, we can find $g_0\in\W^{1,p_0}_*(D)$ such that for every $R>0$,
\begin{equation}\label{gk-convergence-bdry}
\begin{array}{lll}
g_k\rightarrow g_0 &\quad\text{strongly in}&\quad L^{p_0}(B_R^+\times(-16,0)),\\
\nabla g_k\rightharpoonup\nabla g_0 &\quad\text{weakly in}&\quad L^{p_0}(D),\\
\partial_t g_k\rightharpoonup\partial_t g_0 &\quad\text{weakly in}&\quad L^{p_0}((-16,0),W^{-1,p_0}(\mathbb R^n_+)),
\end{array}
\end{equation}
as $k\rightarrow\infty$. Using \eqref{weak star convergence--bdry}, the definition of $H_k$, and \eqref{gk-convergence-bdry}, and passing to the limit in the weak formulation of \eqref{g-k-bdry-eqn} as in the proof of Lemma \ref{L2-comparision}, we infer that $g_0\in\W^{1,p_0}_*(D)$ is a weak solution of the equation
\[
\left\{
\begin{array}{cccl}
\partial_t g_0-\text{div}(\bar{\bA}(t)\nabla g_0)&=&0&\quad\text{in }D,\\
g_0&=&0&\quad\text{on }\partial' D.
\end{array}\right.
\]
By the uniqueness of the weak solution to this equation, $g_0=0$ in $D$. Consequently, from the first convergence in \eqref{gk-convergence-bdry}, we obtain
\[
\lim_{k\rightarrow\infty}\|g_k\|_{L^{p_0}(Q_4^+)}=0.
\]
Also, as $g_k=0$ on $T_{4\kappa_0}\times(-16\kappa_0^2,0)$, the uniform $L^\infty(Q_{4\kappa_0}^+)$ bound of $\nabla g_k$ implies that $\|g_k\|_{L^\infty(Q_{4\kappa_0}^+)}\leq N$. Combining the last two estimates, we infer that
\begin{equation}\label{gk-zero-bdry-mixed}
\lim_{k\rightarrow\infty}\|g_k\|_{L^2((-16\kappa_0^2,0),L^{p_0}(B_{4\kappa_0}^+))}=0.
\end{equation}

\smallskip
Applying Lemma \ref{L-q-2-wei} and the Poincar\'{e} inequality (Lemma \ref{imbed-wei-0309}-(ii)) for $u_k$ vanishing on $T_4$, and using \eqref{a-mu-k-bdry}, \eqref{gradient small-bdry}, and the equation for $u_k$, we obtain
\[
\|u_k\|_{L^2((-16,0),W^{1,p_0}(B_4^+))}+\|\partial_t u_k\|_{L^2((-16,0),W^{-1,p_0}(B_4^+))}\leq N.
\]
Applying the Aubin-Lions theorem, we infer, after passing to a subsequence, that
\begin{equation}\label{uk-convergence-bdry-mixed}
\lim_{k\rightarrow\infty}\|u_k-u_0\|_{L^2((-16,0),L^{p_0}(B_4^+))}=0.
\end{equation}
Finally, applying Lemma \ref{property}-(i) and using \eqref{u-zero-bound-0314-bdry} and \eqref{h-k-grad-est-bdry} with $\bar q=2$, we infer that
\[
\Big(\fint_{Q_{3,\beta_k}^+}|\nabla v_k|^2 dxdt\Big)^{1/2}\leq N,
\]
which is the second assertion of Lemma \ref{L2-comparision-bdry} for $v=v_k$. Using \eqref{gk-zero-bdry-mixed} and \eqref{uk-convergence-bdry-mixed}, we obtain a contradiction to \eqref{w-k-contra-bdry} as in the proof of \eqref{inter-claim}, applying the boundary Sobolev-Poincar\'{e} inequality (Lemma \ref{Sobolev-Poincare}-(ii)) to $w_k=u_k-v_k$ instead of $\hat w_k$. We skip the details.
\end{proof}

From Lemma \ref{L2-comparision-bdry}, we obtain the following proposition which is the main result of the subsection.
\begin{proposition} \label{L2-gradient-comparision-bdry} For every $\nu\in (0,1)$, $M_0 \geq 1$, and for a sufficiently small number $\epsilon \in (0,1)$, there exist a sufficiently small constant $\kappa = \kappa(n, \nu, M_0, \epsilon) \in (0,1)$ and a constant $N = N(n, \nu, M_0)>0$ such that the following assertions hold. Suppose that $\bA$  satisfies \eqref{ellip-cond} in $Q^+_{4, \beta}$ and its corresponding  $\beta$ satisfies \eqref{beta-cond}. Suppose also that
\begin{align*}
 \fint_{Q^+_{4,\beta}}\Theta_{\bA, 4}^2 \beta(dxdt)  + \fint_{Q^+_{4,\beta}}|F|^2 \beta(dxdt) \leq \kappa^2.
\end{align*}
Then, for every weak solution $u \in \hat\W^{1,2}(Q^+_{4,\beta}, \beta)$ of \eqref{eqn-bdry} satisfying
\[
\fint_{Q^+_{4,\beta}} |\nabla u(x,t)|^2 \beta(dxdt) \leq 1,
\]
there exists a weak solution $v \in \hat\W^{1,2}_{\textup{loc}}(Q^+_{4,\beta})$ of \eqref{bdry-approx} such that
\begin{equation*} 
\Big(\fint_{Q^+_{2, \beta}} |\nabla u-\nabla v|^2 \beta(dxdt)\Big)^{\frac{1}{2}}  \leq \epsilon \quad \text{and} \quad \|\nabla v\|_{L^\infty(Q^+_{5/2, \beta})}  \leq N.
\end{equation*}
\end{proposition}

\smallskip
\begin{proof}  With Lemma \ref{Bdr-Lipschitz-constant}, Lemma \ref{bdry-Caccioppoli}, and Lemma \ref{L2-comparision-bdry}, we can prove the proposition exactly the as that of Proposition \ref{L2-gradient-comparision}. We skip the details.
\end{proof}

\smallskip
Now, we give the proof of Theorem \ref{inter-theorem-bdry}.
\begin{proof}[Proof of Theorem \ref{inter-theorem-bdry}] From Proposition \ref{L2-gradient-comparision} and Proposition \ref{L2-gradient-comparision-bdry}, we can apply to level-set estimate to prove that
\begin{equation} \label{grad-bdry-est-0707}
\|\nabla u\|_{L^p(Q_{1,\beta}^{+}, \beta)} \leq  N\Big(\|\nabla u\|_{L^2(Q_{2\Lambda,\beta}^+, \beta)}+\|F\|_{L^p(Q_{2\Lambda,\beta}^+,\beta)}\Big).
\end{equation}
As this is standard, we skip the details. Nevertheless, similar proof can be found as in \cite[Section 6]{FP-1}, for example.

\smallskip
Now, from \eqref{grad-bdry-est-0707}, the PDE of $u$, the weighted Poincar\'e inequality \cite[Theorem 1.6]{Fabes}, we infer that
\begin{align*}
&\|u\|_{L^p(Q_{1,\beta}^+, \beta)} + \| \nabla u\|_{L^p(Q_{1,\beta}^+, \beta)}+\|u_t\|_{L^p(\Gamma_{\beta}(1),W^{-1,p}(B_1^+, \beta))}\\
&\leq N\Big(\|\nabla u\|_{L^2(Q_{2\Lambda,\beta}^+, \beta)}+\|F\|_{L^p(Q_{2\Lambda,\beta}^+,\beta)}\Big).
\end{align*}
The proof of the theorem is completed.
\end{proof}
\section{Global Estimates and Proof of Theorem \ref{main-theorem}} \label{global-section}
To begin, let us introduce some notation and definitions. For $p \in (1, \infty)$, $\W^{1,p}_0(\Omega_T, \beta)$ is the weighted Sobolev space consisting of functions $u \in L^p((0, T), W^{1,p}_0(\Omega, \beta))$ such that
\[
u_t \in L^p((0, T), W^{-1, p}(\Omega, \beta)).
\]
The space $\W^{1,p}_{0} (\Omega_T, \beta)$ is endowed with the norm
\begin{equation*}
\|u\|_{\W^{1,p}_{0} (\Omega_T, \beta)} = \|u\|_{L^p((0, T), W^{1,p}_0(\Omega, \beta))} + \|u_t\|_{L^p((0, T), W^{-1, p}(\Omega, \beta))}, 
\end{equation*}
for $u \in \W^{1,p}_{0} (\Omega_T, \beta)$. Let $\bA$ satisfy \eqref{ellip-cond} in which its corresponding $\beta$ satisfies \eqref{beta-cond}, we consider the boundary value problem
\begin{equation} \label{main-eqn-2}
\left\{
\begin{aligned}
u_t - \textup{div}(\bA(x,t) \nabla u) &= \textup{div}(\beta(x) F(x,t)) \quad  &&\text{in} \quad \Omega_T,\\
u &= 0  \quad &&\text{on} \quad \partial\Omega\times (0, T].
\end{aligned} \right.
\end{equation}
To prove Theorem \ref{main-theorem}, we need a result on regularity estimates for weak solutions of \eqref{main-eqn-2} in which there is no required initial condition at $t=0$ in \eqref{main-eqn-2}.  For a given function $F \in L^1_{\text{loc}}(\Omega_T, \beta)^n$, a function $u \in \W^{1,p}_0(\Omega_T, \beta)$ with $p \in (1, \infty)$ is said to be a weak solution to \eqref{main-eqn-2} if
\begin{align*}
& -\int_{\Omega_T} u(x,t) \partial_t \varphi(x,t) dx dt + \int_{\Omega_T} \wei{\bA(x,t) \nabla u(x,t), \nabla \varphi (x,t)} dx dt \\
& = -\int_{\Omega_T} \wei{F(x,t), \nabla \varphi(x,t)}\beta(x) dx dt, \quad \forall \varphi \in C^\infty_0(\Omega_T).
\end{align*}

\smallskip
From Theorem \ref{inter-theorem} and Theorem \ref{inter-theorem-bdry}, we apply the partition of unity and the flattening boundary technique to derive the following result that gives the regularity estimates for weak solutions of \eqref{main-eqn-2}. Applying the time dilation as in \eqref{scale-2} with $\lambda_0$ in place of $\E_\beta(r)$, we normalize the weight so that $\int_\Omega \hat\beta(x)^{-n_0} dx=1$. Under this dilation, the weighted Lebesgue and parabolic Sobolev norms in \eqref{global-time-cut-est} are invariant, and the time interval becomes $(0,T/\lambda_0)$. By the doubling and reverse H\"older properties of $\hat\beta^{-n_0}$ and \eqref{heights}, the radii and number of cylinders in the covering argument can be chosen in terms of $n,\nu,p,M_0,R_0,\Omega,\tau_0/\lambda_0$, and $T/\lambda_0$ only. As the rest of the proof is standard, we skip the details. One can find the details in \cite[Section 7]{FP-1}, for example.
\begin{proposition} \label{main-thm-2} For every $\nu \in (0,1)$, $M_0 \geq 1$, and $p \in [2, \infty)$, there exists a sufficiently small constant $\delta = \delta (n, \nu, p, M_0) \in (0,1)$ such that the following assertions hold. Suppose that  \eqref{ellip-cond}, \eqref{beta-cond} hold, $\partial \Omega \in C^1$, and $[[\bA]]_{\textup{BMO}(\Omega_T, R_0)}< \delta$ with some $R_0 \in (0,1)$. Suppose also that $u \in \W^{1,2}_{0} (\Omega_T, \beta)$ is a weak solution to \eqref{main-eqn-2} with $F \in L^p(\Omega_T, \beta)^n$. Then, for any $\tau_0 \in (0, T)$, it holds that $u \in \W^{1,p}_0(\Omega\times (\tau_0, T),\beta)$. Moreover, there is a constant $N = N(n, \nu, p, M_0, R_0, \Omega, \tau_0/\lambda_0, T/\lambda_0) >0$ such that
\begin{equation} \label{global-time-cut-est}
\|u\|_{\W^{1,p}_0(\Omega\times (\tau_0, T),\beta)}\leq N\Big[ \|\nabla u\|_{L^2(\Omega_T, \beta)} + \|F\|_{L^p(\Omega_T, \beta)} \Big]
\end{equation}
where $\lambda_0 = \big(\int_\Omega \beta(x)^{-n_0} dx\big)^{1/n_0}$. 
\end{proposition}
\smallskip
Now, we have all ingredients ready for the proof of Theorem \ref{main-theorem}.
\begin{proof}[Proof of Theorem \ref{main-theorem}]  Applying the time dilation as in \eqref{scale-2} with $\lambda_0$ in place of $\E_\beta(r)$, and replacing $T$ by $T/\lambda_0$, we may assume without loss of generality that $\lambda_0=1$.

\smallskip
Consider the case that $p =2$. Let $u \in \W^{1,2}_*(\Omega_T, \beta)$  be a weak solution to \eqref{main-eqn}, by using Steklov's average as in  \cite[p. 18]{DiB}, we can formally take $u$ as the test function to equation \eqref{main-eqn} to obtain
\begin{align*}
& \frac{1}{2}\frac{d}{dt}\int_{\Omega} u(x,t)^2 dx + \int_{\Omega}\wei{\bA(x,t) \nabla u(x,t), \nabla u(x,t)} dx \\
&= -\int_{\Omega} \wei{F(x,t), \nabla u(x,t)} \beta(x) dx.
\end{align*}
From this, \eqref{ellip-cond}, and the standard energy estimate using Young's inequality, we obtain
\[
\sup_{t\in (0,T)}\int_{\Omega} u(x,t)^2 dx + \int_{\Omega_T} |\nabla u(x,t)|^2 \beta(x) dxdt \leq N \int_{\Omega_T} |F(x,t)|^2\beta(x) dxdt
\]
for $N = N(n, \nu)>0$. From this, and due to the homogeneous boundary condition, we can apply the Poincar\'e inequality (\cite[Theorem 1.6]{Fabes}) to get
\[
\|u\|_{L^2((0, T), W^{1,2}(\Omega, \beta))} \leq N \|F\|_{L^2(\Omega_T, \beta)}.
\]
Then, using the PDE in \eqref{main-eqn}, we infer that
\begin{equation} \label{main-p=2}
\|u\|_{\W^{1,2}_{*}(\Omega_T, \beta)} \leq N \|F\|_{L^2(\Omega_T, \beta)}.
\end{equation}
Hence, \eqref{main-thm-est} is proved. From the estimate, we also obtain  the existence of a solution $u \in \W^{1,2}_{*}(\Omega_T, \beta)$ via the Galerkin method.  The uniqueness of the solution also follows directly from \eqref{main-thm-est}.  Theorem \ref{main-theorem} is proved when $p=2$.

\smallskip
We consider the case $p >2$. As $F \in L^p(\Omega_T, \beta)^n$, we see that $F \in L^2(\Omega_T, \beta)^n$. Note that as $\lambda_0=1$, it follows from \eqref{beta-cond} and the doubling property of $\beta^{-n_0}$ that $\beta(\Omega)\leq N(n,M_0,\Omega)$. Then, using the case $p=2$ that we just proved, we find a unique weak solution $u \in \W^{1,2}_{*}(\Omega_T, \beta)$ to \eqref{main-eqn}, and \eqref{main-p=2} holds. Then, we extend $u$ and $F$ by zero on $\Omega\times (-1,0)$ and set
\[
\bA(x,t)=\frac{1}{T}\int_0^T \bA(x,s)\,ds, \quad (x,t)\in\Omega\times(-1,0).
\]
Note that \eqref{ellip-cond} holds for the extended matrix. Also, for each $x_0\in\Omega$ and $\rho\in(0,R_0)$, partitioning $(0,T]$ into intervals of length $\rho^2\E_\beta(x_0,\rho)$ starting from $T$ (with the last interval possibly shorter), and applying \eqref{BMO-defi} and Jensen's inequality, we infer that
\[
[[\bA]]_{\textup{BMO}}(\Omega\times(-1,T],R_0)\leq\delta.
\] It is not hard to check that $u \in \W^{1,2}_{*}(\Omega \times (-1, T), \beta)$ is a weak solution to the equation
\begin{equation} \label{main-eqn-extend}
\left\{
\begin{array}{cccl}
u_t - \textup{div}(\bA(x,t) \nabla u) & =& \textup{div}(\beta(x) F(x,t)) & \quad  \text{in} \quad \Omega \times (-1, T), \\[4pt]
u & = & 0 & \quad \text{on} \quad \partial' ( \Omega \times (-1, T)).
\end{array} \right.
\end{equation}
Now, taking $\delta$ in Theorem \ref{main-theorem} smaller than the number defined in Proposition \ref{main-thm-2}, and applying Proposition \ref{main-thm-2} to \eqref{main-eqn-extend}, we infer that $u \in \W^{1,p}_{*}(\Omega_T, \beta)$  and
\[
\|u\|_{\W^{1,p}_{*}(\Omega_T, \beta)} \leq N \Big[ \|\nabla u\|_{L^2(\Omega_T, \beta)} + \|F\|_{L^p(\Omega_T, \beta)} \Big].
\]
Then, \eqref{main-thm-est} follows from this last estimate,  \eqref{main-p=2} and the H\"{o}lder inequality. Theorem \ref{main-theorem} is proved when $p > 2$.

\smallskip
It remains to consider the case $p \in (1, 2)$. Note that $\beta \in A_p$. In this case, the existence, uniqueness, and the estimate \eqref{main-thm-est} can be obtained by the duality argument using the case $p'=p/(p-1)\in(2,\infty)$. Here we take $\delta$ sufficiently small so that the assumption \eqref{BMO-defi} is satisfied with the smallness constant corresponding to $p'=p/(p-1)$. As the argument is standard and it can be found in \cite{DK, DP}, we skip the details. The proof of the theorem is completed.
\end{proof}
\section*{Acknowledgement}
\noindent
T. Phan's research was partially supported by Simons Foundation, grant \# 769369.  The author would like to express his sincere thanks to professor Jae-Hwan Choi for his comments on the error in Lemma \ref{imbed-epsilon-0309}.


\begin{thebibliography}{99}

\bibitem{AFV} A. Audrito, G. Fioravanti, {and} S. Vita, {\it Schauder estimates for parabolic equations with degenerate or singular weights},  Calc. Var. 63, 204 (2024). https://doi.org/10.1007/s00526-024-02809-2.

\bibitem{AFV2} A. Audrito, G. Fioravanti, {and} S. Vita, {\it Higher order Schauder estimates for degenerate or singular parabolic equations},  Rev. Mat. Iberoam.  41 (2025), no. 4, 1513-1554.

\bibitem{BDGP} A. Kh. Balci,  L. Diening, R. Giova,  and A. Passarelli di Napoli, {\it Elliptic equations with degenerate weights}. SIAM J. Math. Anal. 54 (2022), no. 2, 2373-2412{.}

\bibitem{BS} P. Bella and M. Sch\"{a}ffner, {\it Local boundedness and Harnack inequality for solutions of linear nonuniformly elliptic equations},  Comm. Pure Appl. Math. 74 (2021), no. 3, 453-477.

\bibitem{Brezis} H. Brezis, {\it On a conjecture of J. Serrin}, Atti Accad. Naz. Lincei Rend. Lincei Mat. Appl. 19 (2008), no. 4, 335-338. http://dx.doi.org/10.4171/RLM/529.

\bibitem{BW1} S.-S. Byun and L. Wang,  {\it Parabolic equations in Reifenberg domains}. Arch. Ration. Mech. Anal. 176 (2005), no. 2, 271-301.

\bibitem{CP} L.~A.~Caffarelli and I.~Peral{,} {\it On $W^{1,p}$ estimates for elliptic equations in divergence form}, Comm. Pure Appl. Math. 51 (1998), no. 1, 1-21.

\bibitem{CMP} D. Cao, T. Mengesha, and T. Phan,  {\it Weighted $W^{1,p}$-estimates for weak solutions of degenerate and singular elliptic equations}, Indiana University Mathematics Journal, Vol. 67, No. 6 (2018), 2225-2277.

\bibitem{CMP-1} D. Cao, T. Mengesha, and T. Phan,  {\it Gradient estimates for weak solutions of linear elliptic systems with singular-degenerate coefficients}, AMS Contemporary Mathematics, Nonlinear Dispersive Waves and Fluids, Volume 725, 2019, 13-33{.}

\bibitem{Chia-1} F. Chiarenza and R. Serapioni, {\it Degenerate parabolic equations and Harnack inequality}, Ann. Mat. Pura Appl. {(4) 137 (1984),} 139-162.

\bibitem{Chia-3} F.  Chiarenza and R. Serapioni, {\it 
A Harnack inequality for degenerate parabolic equations}, Comm. Partial Differential Equations 9 (1984), no. 8, 719-749.

\bibitem{Chia-2} F. Chiarenza and R. Serapioni, {\it A remark on a Harnack inequality for degenerate parabolic equations}, {Rend. Sem. Mat. Univ. Padova} 73 (1985), 179-190.

\bibitem{CJP} S. Cho, J. Fang, and T. Phan, {\it Harnack inequality for singular or degenerate parabolic equations in non-divergence form},  Journal of Differential Equations, Volume 457, 15 March 2026, 113997, \url{https://doi.org/10.1016/j.jde.2025.113997}.

\bibitem{DiB} E. DiBenedetto, {\it Degenerate Parabolic Equations}, Springer-Verlag, New York, 1993.

\bibitem{DIT} M. M. Disconzi, M. Ifrim, and D. Tataru, {\it The relativistic Euler equations with a physical vacuum boundary: Hadamard local well-posedness, rough solutions and continuation criterion}, Arch. Rational Mech. Anal. 245, 127-182 (2022).

\bibitem{Dong-Kim-2011} { H. Dong and D. Kim, {\it On the $L_p$-solvability of higher order parabolic and elliptic systems with BMO coefficients}, Arch. Ration. Mech. Anal. 199 (2011), 889--941.}

\bibitem{DK} H. Dong and D. Kim, {\it On $L_p$-estimates for elliptic and parabolic equations with $A_p$ weights}. Trans. Amer. Math. Soc. 370 (2018), no. 7, 5081-5130.

\bibitem{DP1} H. Dong and T. Phan, {\it Parabolic and elliptic equations with singular or degenerate coefficients: the Dirichlet problem}, Trans. Amer. Math. Soc. 374 (2021), no. 9, 6611-6647, DOI 10.1090/tran/8397.

\bibitem{DP3} H. Dong and T. Phan, {\it Weighted mixed-norm {$L_p$-}estimates for elliptic and parabolic equations in non-divergence form with {singular coefficients}}, Revista Matematica Iberoamericana, Vol. 37, {(2021), no. 4, pp.} 1413-1440, DOI: 10.4171/rmi/1233.

\bibitem{DP} H. Dong and T. Phan, {\it On parabolic and elliptic equations with singular or degenerate coefficients}, Indiana University Mathematics Journal, {72} (2023), no. 4, 1461-1502.

\bibitem{DP-0} H. Dong and T. Phan, {\it Weighted mixed-norm $L_p$ estimates for equations in non-divergence form with singular coefficients: The Dirichlet problem}, Journal of Functional Analysis, Volume 285, Issue 2, 109964 (2023).

\bibitem{DPS} H. Dong, T. Phan, and Y. Sire, {\it Sobolev estimates for singular-degenerate quasilinear equations beyond the $A_2$ class},  J. Geom. Anal. 34 (2024), 286. https://doi.org/10.1007/s12220-024-01729-z.

\bibitem{DPT2} H. Dong, T. Phan, and H.~V.~ Tran, {\it Degenerate linear parabolic equations in divergence form on the upper half space}, Trans. Amer. Math. Soc. 376 (2023), 4421-4451. DOI: https://doi.org/10.1090/tran/8892.

\bibitem{DPT1} H. Dong, T. Phan, and H.~V.~ Tran, {\it Nondivergence form degenerate linear parabolic equations on the upper half space}, Journal of Functional Analysis, Volume 286, Issue 9, 1 May 2024, 110374.

\bibitem{EPL} H. J. Eberl, D. F. Parker, and M. C. M. van Loosdrecht, {\it A new deterministic spatio-temporal continuum model for biofilm development}, Journal of Theoretical Medicine, 3(3), 161-175 (2001).

\bibitem{Fabes-1} E. B. Fabes, {\it Properties of nonnegative solutions of degenerate elliptic equations}, Proceedings of the International Conference on Partial Differential Equations dedicated to Luigi Amerio on his 70th birthday, Milan/Como, 1982, Milan J. Math. 52:11-21, 1982.

\bibitem{Fabes} E. B. Fabes, C. E. Kenig, and R. P. Serapioni, {\it The local regularity of solutions of degenerate elliptic equations}, Comm. Partial Differential Equations 7 (1982), no. 1, 77-116.

\bibitem{FP-1} J. Fang and T. Phan, {\it Well-posedness for a class of parabolic equations with singular-degenerate coefficients}, J. Lond. Math. Soc. (2026), DOI:10.1112/jlms.70769, arXiv:2510.20051.

\bibitem{FP-2} J. Fang {and} T. Phan, {\it On $W^{2,\epsilon}$-estimates for a class of singular-degenerate parabolic equations},  arXiv:2601.04324.

\bibitem{FKRH}  M. M. Farid, A. M. Khudhair, S. Ali K{.} Razack, and S. Al-Hallaj, {\it A review on phase change energy storage: materials and applications}, Energy Conversion and Management, 45 {(2004), no. 9-10,} 1597-1615, https://doi.org/10.1016/j.enconman.2003.09.015.

\bibitem{Pop-2} P. M. N. Feehan and C. A. Pop, {\it Schauder {a priori} estimates and regularity of solutions to boundary-degenerate elliptic linear second-order partial differential equations}, J. Differ. Equ. 256 (2014), no. 3, 895-956.

\bibitem{GS} M. Giaquinta {and M. Struwe}, {\it On the partial regularity of weak solutions of nonlinear parabolic systems}, Math. Z. 179 (1982) 437-451.

\bibitem{Grafakos-2} L. Grafakos, {\it Classical Fourier Analysis}, Graduate Texts in Mathematics 249, Springer, Third Edition, 2014.

\bibitem{Lin} Q. Han and F. H. Lin, {\it Elliptic Partial Differential Equations}, Courant Institute of Mathematical Sciences, New York University, New York, 1997.

\bibitem{HNP}  L. T. Hoang, T. V. Nguyen, and T. V. Phan, {\it Gradient estimates and global existence of smooth solutions to a cross-diffusion system}. SIAM J. Math. Anal. 47 (2015), no. 3, 2122-2177.

\bibitem{JS}  D. Jesus and Y. Sire, {\it Gradient regularity for fully nonlinear equations with degenerate coefficients}. Ann. Mat. Pura Appl. (4) 204 (2025), no. 2, 625-642.

\bibitem{JX} J. Ji and J. Xiong, {\it On some divergence-form singular elliptic equations with codimension-two boundary: $L_p$-estimates}, arXiv:2510.06716.

\bibitem{JX1} T. Jin and J. Xiong, {\it Optimal boundary regularity for fast diffusion equations in bounded domains}, Amer. J. Math., vol. 145, no. 1, 151-219, 2023.

\bibitem{KK1} D. Kim and N. V. Krylov, {\it Parabolic Equations with Measurable Coefficients}. Potential Anal 26, 345-361 (2007). https://doi.org/10.1007/s11118-007-9042-8.

\bibitem{KK2} D. Kim and N. V. Krylov, {\it Second-order elliptic equations with variably partially VMO coefficients}, Journal of Functional Analysis Vol. 257, Issue 6, 1695-1712 (2009).

\bibitem{Le} N. Q. Le and O. Savin, {\it Schauder estimates for degenerate Monge-Amp\'ere equations and smoothness of the eigenfunctions}, Invent. Math. 207 (2017), no. 1, 389-423.

\bibitem{Li-Xi} J. Li and Z. Xin, {\it Entropy-bounded solutions to the one-dimensional heat conductive compressible Navier-Stokes equations with far field vacuum},  Comm. Pure Appl. Math. 75 (2022), no. 11, 2393-2445.

\bibitem{MP}  T. Mengesha and T. Phan, {\it Weighted $W^{1,p}$-estimates for weak solutions of degenerate elliptic equations with coefficients degenerate in one variable}. Nonlinear Anal. 179 (2019), 184-236.

\bibitem{MNS-1}  G. Metafune, L. Negro, and C. Spina, {\it  Singular parabolic problems in the half-space}. Studia Math. 277 (2024), no. 1, {1-44}.

\bibitem{muckenhoupt} B.~Muckenhoupt, \emph{Weighted norm inequalities for the Hardy maximal function}, Trans. Amer. Math. Soc. \textbf{165} (1972), 207-226.

\bibitem{M-W} B. Muckenhoupt and R. L. Wheeden, {\it Weighted bounded mean oscillation and {Hilbert} transform}, Studia Math. 54 (1975/76), no. 3, 221-237.

\bibitem{PSW} J. Pr\"{u}ss,  G. Simonett, and M. Wilke, {\it On thermodynamically consistent Stefan problems with variable surface energy}.  Arch. Ration. Mech. Anal. 220 (2016), no. 2, 603-638.

\bibitem{Simon}  J. Simon, {\it Compact Sets in the Space} $L^p(0,T; B)$,  Annali di Matematica pura ed applicata 146, 65-96 (1986). https://doi.org/10.1007/BF01762360.

\bibitem{STT} Y. Sire, S. Terracini, {and} G. Tortone, {\it On the nodal set of solutions to degenerate or singular elliptic equations with an application to s-harmonic functions},  Journal de Math\'{e}matiques Pures et Appliqu\'{e}es, Vol 143 (2020),  376-441.

\bibitem{STV1} Y. Sire, S. Terracini, and S. Vita, {\it Liouville type theorems and regularity of solutions to degenerate or singular problems part I: even solutions}, Comm. Partial Differential Equations 46 (2021), no. 2, 310-361. MR 4207950, DOI 10.1080/03605302.2020.1840586.

\bibitem{STV2}  Y. Sire, S. Terracini, and S. Vita, {\it Liouville type theorems and regularity of solutions to degenerate or singular problems part II: odd solutions}, Math. Eng. 3 (2021), no. 1, Paper No. 5, 50, DOI 10.3934/mine.2021005.

\bibitem{Tru-1} N. S. Trudinger,  {\it On the regularity of generalized solutions of linear, non-uniformly elliptic equations}. Arch. Rational Mech. Anal. 42 (1971), 50-62. doi:10.1007/BF00282317.

\bibitem{Tru-2} N. S. Trudinger, {\it Harnack inequalities for nonuniformly elliptic divergence structure equations}, Invent. Math.  64 (1981), no. 3, 517-531.

\bibitem{Vazquez} J. L. V\'{a}zquez, {\it The Porous Medium Equation: Mathematical Theory}, Oxford University Press (2007).

\bibitem{Visintin} A. Visintin, {\it Models of Phase Transitions}, Progress in Nonlinear Differential Equations and their Applications, 28. Birkh\"{a}user Boston, Inc., Boston, MA, 1996.

\end{thebibliography}
\end{document}